\documentclass[11pt]{article}
\usepackage{amsmath}
\usepackage{cases}
\usepackage[left=1.1in,top=1.1in,right=1.1in,bottom=1.1in]{geometry}
\usepackage{graphics,epsfig}

\usepackage{url}
\usepackage[square, numbers]{natbib}
\usepackage{amsthm,amssymb}
 \newtheorem{theorem}{Theorem}[section]
 \newtheorem{lemma}[theorem]{Lemma}
 \newtheorem{condition}[theorem]{Condition}
  \newtheorem{assumption}[theorem]{Assumption}
 
 \newtheorem{proposition}[theorem]{Proposition}
 \newtheorem{remark}[theorem]{Remark} 
 
 \newtheorem{corollary}[theorem]{Corollary}
 \newtheorem{example}[theorem]{Example}

\usepackage[mathscr]{eucal}
\usepackage[textsize=small]{todonotes}
\usepackage{graphicx}
\usepackage{latexsym}
\usepackage{amssymb}
\usepackage{psfrag}
\usepackage{epsfig}
\DeclareMathAlphabet{\mathpzc}{OT1}{pzc}{m}{it}
\usepackage{enumerate}
\usepackage{comment}
\include{psbox}
\usepackage{amsfonts}
\usepackage{graphicx}
\usepackage{color}
\usepackage{macro}
\newcommand{\mcp}{\varsigma}
\newcommand{\const}{\mathscr{C}}
\newcommand{\cnst}{\mathscr{B}}
\newcommand{\consta}{\mathscr{K}}

\newcommand{\mdpsc}{\delta(\vep)}
\newcommand{\mdpscii}{\delta^2(\vep)}
\newcommand{\mdpscb}{\beta(\vep)}
\newcommand{\modu}{\mfk r}

\newcommand{\mart}{\mathscr{M}}
\newcommand{\mar}{\bar{\mathscr{M}}}

\begin{document}
\title{Inhomogeneous functionals and approximations of invariant distributions of ergodic diffusions: Error analysis through central limit theorem and moderate deviation asymptotics}
\author{Arnab Ganguly\footnote{Research supported in part by Louisiana Board of Regents through the Board of Regents Support Fund (contract number: LEQSF(2016-19)-RD-A-04)} \ and P. Sundar\\
	Department of Mathematics\\
	Louisiana State University\footnote{Email: aganguly@lsu.edu, psundar@lsu.edu}\\
	} 

\maketitle

\begin{abstract}
\np
The paper considers an Euler discretization based numerical scheme for approximating functionals of invariant distribution of an ergodic diffusion. Convergence of the numerical scheme is shown for suitably chosen discretization step, and a thorough error analysis is conducted by proving central limit theorem and moderate deviation principle for the error term. The paper is a first step in understanding efficiency of  discretization based numerical schemes for estimating invariant distributions, which is comparatively much less studied than the schemes used for generating approximate trajectories of diffusions over finite time intervals. The potential applications of these results also extend to other areas including mathematical physics, parameter inference of ergodic diffusions and  analysis of multiscale dynamical systems with averaging.\\

\np {\bf AMS 2010 subject classifications:} 60F05, 60F10, 60H10, 60H35, 65C30.\\

\np {\bf Keywords:} estimation of invariant distributions, ergodic diffusions, Euler approximation, central limit theorem, large deviations,  moderate deviations, stochastic algorithms.

\end{abstract}


\setcounter{equation}{0}
\renewcommand {\theequation}{\arabic{section}.\arabic{equation}}
\section{Introduction.}\label{intro}
Consider the stochastic differential equation (SDE)
\begin{align}\label{SDE0}
	X(t) &= x_0 + \int_0^t b(X(s))ds+  \int_0^t\s(X(s)) dB(s), \qquad x_0 \in \R^d,
\end{align}	
where $B$ is an $m$-dimensional Brownian motion. Assume that the coefficients   $b:R^d \rt \R^d$ and $\s:R^d \rt \R^{d\times m}$  are such that \eqref{SDE0} admits a unique strong solution $X$ and that $X$ is ergodic with invariant distribution $\pi$. We are interested in estimation of $\pi$. Of course, $\pi$ satisfies the stationary Kolmogorv forward equation (also known as stationary Fokker-Plank equation),  $\SC{L}^*\pi = 0$, in the weak sense, where $\SC{L}^*$ is the adjoint of the generator $\SC{L}$ of $X$ given by
\begin{align}\label{gen-diff}
	\SC{L}g(x) = \sum_{i}b_i(x)\partial_ig(x) + \f{1}{2}\sum_{ij} a_{ij}(x)\partial_{ij}g(x), \qquad g \in C^2(\R^d,\R).
\end{align}
Here $a=\s\s^T$. But a closed form expression  of the solution of the above partial differential equation (PDE) is almost always unavailable, except in some simple examples, and numerical schemes for estimation of solutions of this PDE turn out to be computationally expensive, even in dimension $d=3$.  

An alternate approach is to use probabilistic method, where one uses the ergodic theorem to observe that under some standard conditions
\begin{align*}
\f{1}{T} \int_0^T f(X(s)) ds \rt \pi (f) \doteq \int_{\R^d} f(x)\pi(dx) \quad a.s
\end{align*}	
as $T \rt \infty$. So one could potentially use $\f{1}{T} \int_0^T f(X(s)) ds$ as an estimate for $\pi(f)$ for large $T$, and because of ergodic theorem this estimator will be asymptotically unbiased (or consistent as is called in statistics literature). But the problem is that even this integral is hard to evaluate. Even though exact simulation schemes of the diffusion $X$ are available \cite{BR05, BPR06}, the easiest and oftentimes the most practical approach in realistic models is to use an Euler-Maruyama discretization, which results in approximating this integral by a Riemann sum of the form $\f{1}{N}\sum_{k=1}^{N} f(Z^\Delta(t_k))$, where
$$Z^\Delta(t_{k+1}) = Z^\Delta(t_k)+ b(Z^\Delta(t_k)) \Delta + \s(Z^\Delta(t_k)) (W(t_{k+1}) - W(t_k)), \quad t_{k+1} - t_{k} = \Delta.$$ 
Obviously, for such a scheme to be accurate, $N$ has to be large and $\Delta$ small. But the right choices of $\Delta$ and $N$ are often not obvious for many models. We now elaborate on this issue.

Euler-Maruyama schemes for simulating  trajectories of $X$ and estimates for weak and strong error over finite time intervals have been extensively studied, and we mention only a few comprehensive surveys and books for references \cite{KlPl92, Pla99, Hig01} (also see \cite{AGK11} for error analysis of Euler approximation for density-dependent jump Markov process). In comparison, much less is available on theoretical error analysis of its use in approximation of invariant measure for ergodic diffusions. To understand the issues here, note that although the error between $X$ and $Z$ over a fixed time interval $[0,T]$ is typically  $O(\Delta)$ (weak error order), for many stochastic models, the constant involved grows with $T$. Thus estimating the error of such approximations of invariant measure, for example, by `naively' bounding $\f{1}{N}\sum_{k=1}^{N} f(Z^\Delta(t_k)) - \int_{0}^{N\Delta}f(X(s))ds$ does not work since long-time integration is involved. Even a small but fixed discretization step $\Delta$ can lead to infinite error!

  This shows that much care has to taken for a rigorous error analysis, and important early results in this context were obtained by Talay \cite{Tal90, TaTu90, Tal02}. The discretized chain $\{Z^\Delta(t_k)\}$ will often have an invariant distribution $\pi^{\Delta}$, at least if the discretization method is conveniently chosen. Then, under some favorable conditions,  
$$\f{1}{N}\sum_{k=1}^{N} f(Z^\Delta(t_k)) \rt \pi^\Delta(f), \quad \mbox{ as } N \rt \infty,$$
and the total error can be split into two parts:
$$\lf(\f{1}{N}\sum_{k=1}^{N} f(Z^\Delta(t_k)) - \pi^\Delta(f)\ri) + \lf(\pi^\Delta(f) - \pi(f)\ri).$$
The second error is `purely' due to the discretization step, while the first depends on the integration time interval $[0,T]$ ($T= N\Delta$). 
Talay  provides estimates on the second error in terms of $\Delta$ in \cite{Tal90} and \cite{TaTu90}, and notes that the first  term is extremely hard to estimate (also see \cite{Tal02}). But even the estimate on the second error term is given under some strict conditions, which in particular include boundedness of the coefficients (along with $C^\infty$ smoothness). For many stochastic models, where the drift terms satisfy a recurrence condition, including the Ornstein-Uhlenbeck process (where, $b(x) \sim -x$), the boundedness assumption on the drift could restrict applications of such a result.  For SDEs on torus, Mattingly etal. \cite{MST10} gives estimates on the error terms in terms of both $N$ and $\Delta$ (also see \cite{MSH01} for some results in the case of additive noise), but the extension of these results to non-compact case is highly non-trivial.

 It is clear, that a proper scaling between $N$ and $\Delta$ is needed for designing a suitable numerical scheme and a thorough error analysis, which is what this paper is about. Specifically, we not only prove convergence of our numerical scheme, but also establish optimality of rate of convergence through central limit theorem (CLT) and investigate moderate deviation asymptotics. 

We now briefly describe the results in the paper and make some comments about the mathematical technicalities. To discover the right scaling regime, it is convenient to speed up time by the transformation $t \rt t/\vep$, where $\vep \rt 0$. Then by a simple change of variable formula, it could be seen that the dynamics of  $X(\cdot/\vep)$ is given by the SDE \eqref{SDE1}, in the sense that its distribution is same as that of, $X^\vep$, the solution of \eqref{SDE1}. Consequently, $\int_0^t f(X^\vep(s))ds \rt t\pi(f)$ as $\vep \rt 0.$ Thus $X^\vep$ could be viewed as a fast moving process which converges to the invariant distribution $\pi$ in finite time, in contrast to $X$, which does this in infinite time. Although the two formulations are equivalent mathematically, this interpretation is useful in identifying the right scaling regimes for different limit theorems that are presented in this paper, and simplifying derivations of some of the estimates required for their proofs.

 Letting $Z^\vep$ denote the (continuous) Euler approximation of $X^\vep$ (see \eqref{SDE-Eu-1}) corresponding to the discretization step $\Delta(\vep)$, we in fact consider more general inhomogeneous integral functionals of the form $\int_{0}^\cdot f(s,Z^\vep(s)) ds$ (that is, we allow $f$ to depend explicitly on time $t$ as well), and we show that if $\Delta(\vep)  = o(\vep)$, then  $\int_{0}^t f(s,Z^\vep(s)) ds \rt \pi(f)t$. Inhomogeneous functionals are more difficult to handle, but arise naturally in many applications including statistical inference of SDEs and in averaging of dynamical systems whose trajectories are modulated by fast moving diffusions. We then investigate the central limit theorem, which not only establishes the rate of convergence, but also indicates the optimality of the order of the numerical scheme. More specifically, we prove that if $\Delta(\vep)$ is such that $\Delta(\vep) \rt 0$ sufficiently fast (faster than $o(\vep)$) then $\f{1}{\sqrt \vep}\lf(\int_{0}^t f(s,Z^\vep(s)) ds - \pi(f)t\ri)$ converges to a Gaussian process with independent increments, which can  actually be expressed by an appropriate stochastic integral (see Theorem \ref{th:cltZ}). Notice that the above CLT only implies that if $\sqrt \vep \ll \mdpsc \ll 1$, then $\PP\lf(\f{1}{\mdpsc}\lf|\int_{0}^t f(s,Z^\vep(s)) ds - \pi(f)t\ri|>x\ri)\rt 0$ as $\vep \rt 0$, but does not give any information about the rate of decay. This information can be extracted by a moderate deviation analysis, which in fact shows that the decay rate is exponential with certain speed. The precise statement on moderate deviation principle (MDP), actually at a more general process level, is the content of Theorem \ref{th:mdpZ}, and we deem it to be the most important contribution of the paper.  

The MDP is proved by a weak convergence approach, which has been developed in several works of Budhiraja, Dupuis, Ellis, and others \cite{DE97, BoDu98, BDM08, BDM11}, and which has been successful in proving large and moderate deviation principles for a variety of stochastic systems (also see \cite{BDG14} for moderate deviation principles of stochastic equations driven by Poisson random measures and \cite{DuJo15} for a result on moderate deviation for a class of recursive algorithms). The starting point in this approach is a variational representation of expectations of exponential functionals of Brownian motion, from which it can be argued that proving an LDP or equivalently, a Laplace principle entails studying tightness and weak convergence of certain controlled version of the original process. One advantage of this approach is  that it avoids some complicated exponential probability estimates which are particularly hard to obtain for our Euler approximation problem. A crucial role in the study of the tightness of both the original and the associated controlled process is played by the solution of the Poisson equation $\SC{L}u = -f,$ and its regularity properties. Many of the results which provide sufficient conditions for this required regularity properties can be found in the work of Pardoux and Veretennikov \cite{PaVe01} (also see \cite{PaVe03}). However, we do note that, although not explicitly mentioned  in \cite{PaVe01}, the proof of  the estimate on the growth rate of the derivative of the solution of the Poisson equation requires the drift $b$ to be bounded -- a condition which, as mentioned, is restrictive for ergodic diffusions. In our paper, this has been adapted to cover the case for $b$ having some growth properties. For more on this, see Remark  \ref{Ver-rem}.

As expected, similar versions of many estimates that have been developed for studying tightness of the controlled process,  are also used in proving the CLT result. Since the proofs for the controlled versions were already given, they were not repeated when a similar version is required for the original uncontrolled process. The latter proofs are often much simpler, and  only the important changes have been pointed out. Although there are quite a few methods available to prove a central limit type theorem or diffusion approximation, this paper takes a `martingale approach' and uses the martingale central limit theorem to obtain the desired result. 

A different kind of numerical scheme and related error analysis for approximation of invariant measure has been studied in a series of papers \cite{LaPa02, LaPa03, PaPa12, PaPa14, Pan08}. There, a weighted estimator of the form $\sum_{k=1}^Nw_kf(Y_k)/\sum_{k=1}^N w_k$ is considered where $\{Y_k\}$ is a Markov chain obtained by discretizing the SDE \eqref{SDE0} with decreasing time step $\Delta_k$ such that $\Delta_k \rt 0$ as $k \rt \infty$,  $\sum_{k=1}^N\Delta_k \rt \infty, \ \sum_{k=1}^Nw_k \rt \infty$ as $N \rt \infty.$ In contrast, our $\Delta$ does not change with iteration step $k$, but is suitably scaled with   $N$. The weights $w_k$ in these algorithms  could be chosen as $\Delta_k$ or could be chosen as some other values subject to some relations with $\Delta_k$. The recurrent or stability condition is in terms of a Lyapunov function,  and although the convergence of the numerical scheme is shown for a broad class of functions (like our paper),  a CLT for the error is proved only for a smaller class of test functions. These test functions are of the form $\SC{L}\varphi$, with $\varphi$ satisfying several conditions including requirement of bounded derivatives up to second or higher order.  No moderate deviation analysis has been undertaken in any of these papers, and all the results are only for homogeneous functionals (that is, when $f$ just depends on state $x$ and not on time $t$).

Interestingly, but not surprisingly, the machineries which we develop here (actually, in their much simplified versions) also prove a moderate deviation principle of the inhomogeneous integral functionals of the original process $X^\vep$. This, by itself, is an interesting problem,  homogeneous version of which has been studied in quite a few papers \cite{LiSp99, GuLi06} (also see \cite{JaSp18} for such a result in the context of a stochastic model originating from finance).  For the  inhomogeneous case, to the best of our knowledge there exist only one paper \cite{Gui01} on moderate deviation problem, which assumes that $f$ is bounded (also see \cite{Gui03}). The weak convergence approach allows us to lift some of the restrictive conditions assumed before including boundedness of $f$ in \cite{Gui01} and stronger ergodicity conditions in \cite{GuLi06}. Since the treatment of this problem is similar and actually simpler compared to the one which is the main focus of this paper, we only mention the result in Theorem \ref{th:mdpX} without proof. 

Before outlining the organization of our paper, we note that although we motivated the usefulness of these results in terms of estimation of functionals of the invariant distribution, $\pi$, when $\pi$ is unknown or complicated, these results are equally useful in many other contexts. Indeed, understanding asymptotics of integral functionals  is important for many other applications including mathematical physics where they often appear in forms of energy functionals, statistical inference of SDEs,  multiscale dynamical systems where trajectories of a differential equation is influenced by a fast moving Markov process, and option pricing in financial markets \cite{JaSp18}. For example, consider the area of parameter inference and consider the simple but widely used Ornstein-Uhlenbeck (OU) process (see Example \ref{ex:OU}), whose invariant distribution is Normal$(\mu, \s^2/2\kappa)$ (and not something complicated). If $\mu$ is unknown, then a simple, effective and asymptotically unbiased estimator of $\mu$ is $\hat\mu_T = \f{1}{T}\int_0^TX(s)ds.$ But since the data can only be collected in discrete time, it is practical to use the Riemann sum-estimator of the form $\f{1}{N} \sum_{k=1}^{N}Z(t_k)$, where the data $\{Z(t_k)\}$ could be realistically assumed to be coming from the stochastic model corresponding to the Euler approximation of the original one.  In fact, for many stochastic dynamical systems maximum likelihood or other kinds of estimators of the parameters are often nice functions of such integral functionals \cite{Bis08, Kut03}. We cite two other examples to illustrate this point. In the OU model, a {\em minimum contrast estimator} of $\kappa$ (assume $\mu=0$ and $\s=1$ for simplicity) is given by $\hat\kappa_T =  (\f{2}{T} \int_0^T X(s)^2 \ ds)^{-1}$ (see \cite{Bis08}). Estimation of $\kappa$, which measures the speed of reversion toward long term mean, is important in mathematical finance and mathematical physics (where it is the friction coefficient). Next consider the Gompertz diffusion model (see Example \ref{ex:Gomp}), which is  used in modeling commodity prices, freight shipping rates and also tumor growth. A minimum {\em minimum contrast estimator} of the parameter $\mu$ (assume for simplicity $\kappa =1$ and $\s$  is known), which in the tumor growth model is the intrinsic growth rate of tumor, is given by $\hat\mu_T = \f{1}{T}\int_0^T \ln X(s)ds + \s^2/2.$ Again, for a more realistic approach,  discretized versions should be considered. Thus  asymptotic results for many of these estimators for {\em high-frequency data} can be derived quite easily from the limit theorems proved in this paper by suitable applications of continuous mapping theorem or contraction principle (and possibly a little extra work in some cases).
   These results are instrumental in determining efficiency of these estimators, finding approximate confidence intervals or testing appropriate hypotheses.

The rest of the paper is organized as follows. In Section \ref{sec:form}, we give the mathematical formulation of our model and the statements of our main results. The variational representation and the controlled process have been described in Section \ref{sec:var}. Section \ref{sec:poi} contains the required results on  the Poisson equation. Section \ref{sec:eq-rate} gives equivalent forms of the MDP rate functions which are useful in proving upper and lower bounds, and which are proved, respectively, in Section \ref{sec:LP-ub-pf} and Section \ref{sec:LP-lb}. Estimates and related tightness results required for these proofs are discussed in Section \ref{sec:est} and the beginning of Section \ref{sec:tight}. The proof of CLT is given in Section \ref{sec:cltZ-pf}. Finally, the Appendix collects some necessary technical lemmas.\\

\np
{\em Notation:} The following mathematical notation and conventions will be 
used in the paper. For a Polish space $S$, we denote by $\mathcal{P}(S)$
(resp. $\mathcal{M}_{F}(S)$) the space of probability measures (resp. finite
measures) on $S$ equipped with the topology of weak convergence. We denote by
$C_{b}(S)$ the space of real continuous and bounded functions on $S$, and by $C_{b}^1(S)$ the space of bounded Lipschitz continuous functions on $S$. The
space of continuous functions from $[0,T]$ to $S$, equipped with the uniform
topology, will be denoted as $C([0,T]: S)$. For a bounded $\mathbb{R}^{d}$
valued function $g$ on $S$, we define $\|g\|_{\infty} = \sup_{x \in S}
\|g(x)\|$. For a measure $\nu$ on $S$, and an integrable function $g:S \rt \R^k$, $\nu(g) = \int_{S}g(x)\nu(dx)$. For $x\in \R^k$, $\|x\|$ will denote its Euclidean norm. For a matrix $M$, $\|M\|$ will denote some appropriate matrix norm. Since we are working in finite-dimension, and al norms are equivalent, we will not explicitly mention which norms are used, unless it is required. For $g:\R^d \rt \R^k$, $Dg$ will denote its derivative matrix, that is, the $l$-th row is given by $(Dg)_{l*} = \nabla g_{l}$.
$D^2g$ will denote its second derivative, that is, $(D^2g)_{lij} = \partial^{2}_{ij}g_l.$ The  {\em big $O$} and {\em little $o$} notations will be used sometimes. That is  $f(x) = O(g(x))$ as $x \rt a$ if $|f(x)| \leq C |g(x)|$ for $|x-a| \leq \kappa$ for some constants $C$ and $\kappa$, or if $a =\infty$, then for $x>B$ for some constant $B$ (or equivalently, $\limsup_{x \rt a} |f(x)/g(x)| < \infty$). Similarly, $f(x) = o(g(x))$ as $x \rt a$ if $|f(x)/g(x)| \rt 0$, as $x \rt a$.  These notations will be used mostly for the limiting regimes $x \rt \infty$ and $\vep \rt 0$, and the regime intended for such a use of big $O$ or little $o$ notation will be clear from the context. Sometimes, $f(x) \sim g(x)$ will be  used to mean that $f$ and $g$ have same rate of growth, that is, $f(x) = O(g(x))$ and $g(x) = O(f(x))$. This symbol will only be used informally for illustration purposes.

\setcounter{equation}{0}
\renewcommand {\theequation}{\arabic{section}.\arabic{equation}}
\section{ Mathematical framework and some prerequisites}\label{sec:frame}
%
%
%
\subsection{Formulation and main results} \label{sec:form}
For each $\vep>0$, let $X^\vep$ be an $\R^d$-valued diffusion process given by
\begin{align}\label{SDE1}
	X^\vep(t) &= x_0 + \f{1}{\vep}\int_0^t b(X^\vep(s))ds+  \f{1}{\sqrt \vep}\int_0^t\s(X^\vep(s)) dW(s).
\end{align}
We will always assume that the above SDE admits a unique strong solution $X^\vep$. 
 
 The following conditions on the coefficients $b:R^d \rt \R^d$ and $\s:R^d \rt \R^{d\times m}$ will be assumed.
\begin{condition} \label{cond_SDE_inv}  The coefficients $b:R^d \rt \R^d$ and $\s:R^d \rt \R^{d\times m}$ has the following properties:
	\begin{enumerate}[(i)]
            \item   there exist constants $\g>0, \alpha\geq 0$ and $B\geq 0$ such that 
		$$\<x, b(x)\> \leq - \g\|x\|^{1+\alpha}, \quad \mbox{ for } \|x\| >B;$$

		\item there exist strictly positive constants $\l_1$ and $\l_2$ such that for all $x,y \in \R^d$
		\begin{align*}
		0< \l_1 \leq y^T \lf(\s(x)\s^T(x)\ri)y/\|y\|^2 \leq \l_2.
		\end{align*}

		  
	\end{enumerate}
\end{condition}	

\begin{remark} {\rm In the above condition, 
(i) is needed for positive recurrence of $X$, which in turn guarantees existence of an invariant probability measure. Uniqueness of the invariant distribution then follows from non-degeneracy of the matrix $a=\s\s^T$ as formulated in (ii). Note that, in particular, (ii) implies that $\|\s\|_{\infty} = \sup_{x\in \R^d}\|\s(x)\|_{op} < \infty$, where $\|\cdot\|_{op}$ denotes the operator norm. Of course, this is true for any other matrix norm as well, since all such norms are equivalent, and we will drop the suffix {\em op} when considering matrix norm.	

The uniform ellipticity condition (as well as the boundedness assumption) on $a$ could be lifted, if it could be shown that a unique invariant distribution and a unique solution of the Poisson equation \eqref{Pois-1}  exist and satisfy some desired regularity properties. If the boundedness assumption on $a$ (or equivalently, $\s$) is removed, then its growth or decay rate could be incorporated into the assumptions quite easily.
} 
\end{remark}

Under Condition \ref{cond_SDE_inv}, it is well known that $X$ is an ergodic diffusion process with unique invariant measure $\pi$ \cite{Ver87, PaVe01}. Moreover there exist constants $\Theta$, $\theta_1$ and $\theta_2$ such that
\begin{align*}
\|P_t(x,\cdot) - \pi\|_{TV} \leq \Theta \exp\lf(\theta_1\|x\|\ri)\exp(-\theta_2 t), \quad \int_{\R^d} \exp(\theta_1\|x\|)\pi(dx)< \infty,
\end{align*}	
where $P_t(x,\cdot)$ denotes the transition probability kernel and $\|\cdot\|_{TV}$ denotes the total variation norm.

Since we will be dealing with discretization of the original process, much of the required estimates will need appropriate assumptions on the moduli of continuity of the coefficients. In particular, in this paper we work with H\"older continuity (and thus, of course, covering  the case of Lipschitz continuous coefficients), but we anticipate that these assumptions could be sufficiently weakened to cover more general stochastic equations, as long as existence and uniqueness of solutions are guaranteed. But we do note that the following condition is not needed for the MDP result of the original process $X^\vep$ (see Theorem \ref{th:mdpX}).

\begin{condition}  \label{cond_SDE_coef-ex}
\begin{enumerate}[(i)]
		\item
		$b:R^d \rt \R^d$ and $\s:R^d \rt \R^{d\times m}$  are Holder continuous functions with exponent $\nu \in [0,1]$ and Holder constants, $L_b$ and $L_\s$, respectively, that is,
		\begin{align} \label{def-lip}
			L_b = \sup_{x\neq x'} \f{\|b(x) - b(x')\|}{\|x-x'\|^\nu}, \quad L_\s = \sup_{x\neq x'} \f{\|\s(x) - \s(x')\|}{\|x-x'\|^\nu};
		\end{align}
	
	\item there exists a constant $\cnst$ such that $\|b(x)\|\leq  \cnst (1+ \|x\|^{\bar \alpha})$, for $\bar \alpha \leq \alpha \wedge 1$ ;

\end{enumerate}
\end{condition}

Next, for some $n \in \N$, let $f :[0,\infty)\times\R^d \rt \R^n$ be a  function satisfying the following assumption.

\begin{assumption}\label{assum-f}
 For each $t\geq0$, $\pi(f(t,\cdot)) \equiv \int f(t,x) \pi(dx)= 0$, that is, $f$ is centralized. Furthermore, there exist exponents $p_0, q_0 \in \R$ and a constant $\const(T)$ such that 
	\begin{enumerate}[(i)]
		 
		\item $\sup_{t\leq T}\|f(t,x)\| \leq \const(T) (1+ \|x\|)^{p_0};$
		\item $\omega_f(\Delta,x) \leq \const(T)\modu(\Delta)(1+\|x\|)^{q_0}$, where $\omega_f(\Delta,x) \doteq \sup_{|t-s|\leq \Delta, 0\leq s,t\leq T}\|f(x,t) - f(x,s)\|$ is the modulus of continuity of $f$.
     \end{enumerate}		
\end{assumption}
We will need $\modu(\Delta) = o(\sqrt\Delta)$ for the CLT and $\modu(\Delta) = O(\sqrt\Delta)$ for MDP. 	

\begin{remark} ({\bf about notational convention}) {\rm
If $p_0\geq 0$, then by a slight abuse of notation, we will use the same constant $\const(T)$ to write $\sup_{t\leq T}\|f(t,x)\| \leq \const(T) (1+ \|x\|^{p_0}).$ Similar convention will be followed throughout for such	 estimates.
}
\end{remark}

We now consider an appropriate Euler-Maruyama discretization of scheme for $X^\vep$. Let $\{t_k\}$ be a partition of $[0,T]$ such that $\Delta\equiv \Delta(\vep) = t_k - t_{k-1}$, and  let $Z^\vep$ denote the (continuous) Euler approximation of $X^\vep$. In other words, let $Z^\vep$ be the solution to the stochastic equation:
\begin{align}\label{SDE-Eu-1}
	Z^\vep(t) &= x_0 + \f{1}{\vep}\int_0^t b(Z^\vep(\vr_\vep(s)))ds+  \f{1}{\sqrt \vep}\int_0^t\s(Z^\vep(\vr_\vep(s))) dW(s),
\end{align}
where $\vr_\vep(s) = t_k$ if $t_k\leq s < t_{k+1}.$

Let $\Xi_\vep$, defined by $\Xi_\vep(A\times [0,t]) = \int_0^t1_{\{Z^\vep(s) \in A\}}ds$, denote the occupation measure of the process $Z^\vep$, and, as standard, $\Xi_\vep(f)$ will denote the following:
\begin{align*}
\Xi_\vep(f) = \int_{\R^d\times [0,\cdot]}f(s,x)\Xi_\vep(dx\times ds)=   \int_0^{\cdot} f(s, Z^\vep(s))ds.
\end{align*}
 The paper is devoted to study of  precise asymptotic estimates of probabilities like
$P(\|\Xi_\vep(f)(t)\|> x \mdpsc)$ for rightly scaled discretization step $\Delta(\vep)$ in the following scaling regimes:
\begin{itemize}
\item {\em Central limit scaling:} $\mdpsc = \vep^{1/2}.$	

\item {\em Moderate deviation scaling:}
$
\vep\rt 0,\quad	\mdpsc \rt 0, \quad  \mdpscb\equiv \vep/\mdpscii \rt 0.
$

\end{itemize}
Since in the second regime, $\sqrt \vep \ll \mdpsc$, it is clear that these probabilities cannot be estimated by a central limit theorem, which can only estimate probabilities of deviation near the mean. The study of these probabilities falls under the purview of moderate deviation asymptotics, while the case $\mdpsc = 1$ requires investigating large deviation asymptotics (which we do not undertake in this paper). 

In this paper, the notation $\mdpsc$ will be exclusively reserved for moderate deviation scaling regime. 

The paper actually proves a  more general result at the process level. Specifically, defining 
\begin{align*}
\Up_\vep(f) \doteq  \f{1}{\mdpsc} \Xi_\vep(f) = \f{1}{\mdpsc} \int_0^\cdot f(s, Z^\vep(s))ds,
\end{align*}
we establish a (functional) CLT and a large deviation principle (LDP) for $\Up_\vep(f)$  in $C([0,T],\R^d).$ The LDP of $\Up_\vep(f)$ is interpreted as a MDP of the process $\Xi_\vep(f)$. 

For implementation, it might be even more practical and convenient to use the Riemann sum, $\Xi^R_\vep(f) = \sum_{i=1}^{[t/\Delta(\vep)]} f(Z_\vep(t_i))\Delta(\vep)$ as the estimator (the superscript $R$ stands for Riemann sum). The associated limit theorems could be proved under either one of the following additional conditions on $f$.
\begin{assumption}\label{assum-f-2}Either
\begin{enumerate}[(A)] 
\item $f$ is H\"older continuous with H\"older exponent $\nu_f \in (0,1]$; or
\item $f$ is differentiable and $\sup_{t\leq T}\|Df(t,x)\| \leq \const(T)(1+\|x\|)^{p_0'},$ for some $p_0'\geq 0$.
\end{enumerate}	

\end{assumption}

For understanding asymptotics of the above Riemann sum-estimator, it is convenient to work with its integral representation:
\begin{align}\label{est-Riem}
\Xi^R_\vep(f) =   \int_0^{\cdot} f(\vr_\vep(s), Z^\vep(\vr_\vep(s)))ds.
\end{align}	 

Before we state our CLT and LDP results, we first state the result guaranteeing the convergence of our scheme.

\begin{theorem}\label{th:llnZ}
	Let  $f :[0,\infty)\times\R^d \rt \R^n$ satisfy Assumption \ref{assum-f}, with $\modu(\Delta) = O(\sqrt\Delta)$. Let $Z^\vep$ be defined by \eqref{SDE-Eu-1}, where the the step size $\Delta(\vep)$ is such that $\Delta(\vep)/\vep \rt 0$, as $\vep \rt 0$. Then under Condition \ref{cond_SDE_inv}, Condition \ref{cond_SDE_coef-ex},  for $T>0$, there exists a constant $\consta(T)$ such that
	\begin{align*}
		\EE\lf[\sup_{t\leq T}\|\Xi_\vep(f)(t)\|\ri] \leq \consta(T)\sqrt\vep.
	\end{align*}	
In particular, $\Xi_\vep(f) \rt 0$ in probability in $C([0,T], \R^n)$ 	as $\vep \rt 0.$	 (Recall that $f$ is already centralized).

Suppose, in addition, that Assumption \ref{assum-f-2} holds.
Then the above assertion is also true for $\Xi^R_\vep(f)$.

\end{theorem}	 

The proof of this theorem follows easily from the proof of the CLT (stated below) which is given in Section \ref{sec:cltZ-pf}. Indeed, multiplying \eqref{u-Ito-Z} by $\sqrt\vep$, one uses similar estimates (actually simpler versions) used in Section \ref{sec:cltZ-pf} and the proof of Theorem \ref{prop-tight}. In fact by Markov's inequality and Borel-Cantelli lemma, the subsequences along which the convergence is almost sure can be precisely constructed.

For the CLT and the MDP results, we first define the matrix $M_f(t)$ by 
	\begin{align}\label{M-mat-def}
	(M_f(t))_{i,j} = &\ \int_{\R^d}\int_0^\infty \lf[f_i(t,x) P_sf_j(t,\cdot)(x)+         f_j(t,\cdot) P_sf_i(t,\cdot)(x)\ri]ds d\pi(x),
    \end{align}	
where, by a slight abuse of notation, we used $\{P_t\}$ to denote the semigroup corresponding to the transition probability kernels $\{P_t\}$ of $X$; in other words, $P_tg(x) = \int_{\R^d}g(y)P_t(x,dy).$

\begin{theorem}\label{th:cltZ} Let  $f :[0,\infty)\times\R^d \rt \R^n$ satisfy Assumption \ref{assum-f}, with $\modu(\Delta) = o(\sqrt\Delta)$. Let $Z^\vep$ be defined by \eqref{SDE-Eu-1}, where the the step size $\Delta(\vep)$ is such that $(\Delta(\vep)/\vep)^{\nu/2}/\sqrt \vep \rt 0$, as $\vep \rt 0$. Then under Condition \ref{cond_SDE_inv}, Condition \ref{cond_SDE_coef-ex},
		\begin{align*}
	\vep^{-1/2}\Xi_\vep(f) \RT \int_0^{\cdot}M^{1/2}_f(s)dW(s),
		\end{align*}	
as $\vep \rt 0.$	

 Moreover  the above assertion is also true for $\vep^{-1/2}\Xi^R_\vep(f)$ if either one of the two conditions in Assumption \ref{assum-f-2} holds and  $\Delta(\vep)$ is such that $(\Delta(\vep)/\vep)^{\tilde \nu/2}/\sqrt \vep \rt 0$, as $\vep \rt 0$, where 
 \begin{itemize}
 \item 
 $\tilde \nu = \nu \wedge \nu_f$ for  Assumption \ref{assum-f-2}-(A), and
 \item $\tilde \nu = \nu$ for  Assumption \ref{assum-f-2}-(B).
 
 \end{itemize}

\end{theorem}

Finally, we state our MDP result, which we deem to be the most important contribution of the present paper. The full statement requires some assumptions on the solution $u$ of the Poisson equation, $\SC{L}u = -f$, which is the topic of Section \ref{sec:poi}. 

\begin{remark}{\rm
As the reader might observe, we did not explicitly include similar assumptions (Condition \ref{u-reg-assum}) for statements of Theorem \ref{th:llnZ} and Theorem \ref{th:cltZ}.
The reason for this is that under the hypotheses of those theorems, the existence of the solution $u$ with some polynomial growth rate is already given by Proposition \ref{u-est}, which is essentially the result of Pardoux and Veretennikov \cite{PaVe01}. That was enough for proof of these two theorems. 

Now for our MDP result, although such existence is also guaranteed, the growth rate of $u$ and its derivatives have to satisfy some further restrictions (Assumption \ref{assum-exp-u}), which we don't require for Theorem \ref{th:llnZ} and Theorem \ref{th:cltZ}. This is because for proof of MDP result, we need to establish tightness of certain controlled versions of $Z^\vep$.

 In this connection, Proposition \ref{u-est} is only a `sufficient type' result, and the growth rate coming out of Proposition \ref{u-est} might not always be optimal. In other words, for some functions $f$, there might be an alternate way (for example, by direct computation) of computing  the actual growth rates of $u$ and its derivatives. It can turn out that these actual rates satisfy Assumption \ref{assum-exp-u}, whereas the growth rates given by Proposition \ref{u-est} are higher and do not satisfy Assumption \ref{assum-exp-u}! It would then seem that our MDP result will not apply to those functions $f$, where in reality it does. That is why we decided to state the result in more generality. 
 
 Lastly, we do make the obvious observation that if we in fact first choose a $u$ satisfying Condition \ref{u-reg-assum} and Assumption \ref{assum-exp-u} (such functions are of course abundant), then our MDP results apply to $\SC{L}u$. In other words,  our results are applicable to a large class of  `test' functions of the form $\SC{L}u$, where $u$ satisfies Condition \ref{u-reg-assum} and Assumption \ref{assum-exp-u}. It is the inverse problem, that is where $f$ is given first and MDP results are needed for integral functionals of $f$, which requires finding the solution $u$ and verifying its regularity properties.
}
\end{remark}

\begin{theorem}\label{th:mdpZ} Let  $f :[0,\infty)\times\R^d \rt \R^n$ satisfy Assumption \ref{assum-f} with $\modu(\Delta) = O(\sqrt \Delta)$. Let $Z^\vep$ be defined by \eqref{SDE-Eu-1}, where the the step size $\Delta(\vep)$ is such that $(\Delta(\vep)/\vep)^{\nu/2}/\sqrt \vep \rt 0$, as $\vep \rt 0$. Then under Condition \ref{cond_SDE_inv} (with $\alpha>0$), Condition \ref{cond_SDE_coef-ex} and Condition \ref{u-reg-assum} and Assumption \ref{assum-exp-u} , as $\vep\rt 0$, $\{\Up_\vep(f)\}$ satisfies a LDP on $C([0,T],\R^n)$ with speed $\mdpscb  \equiv \vep/\mdpscii$ and rate function $I_f$ given by
\begin{align}\label{rate2}
	I_f(\xi) =
	\begin{cases}
		\f{1}{2} \int_0^T (\dot \xi(s))^T M_f(s)^{-1} \dot \xi(s) ds, & \quad \xi \mbox{ is absolutely continuous};\\
		\infty, & \quad \mbox{ otherwise.}
	\end{cases}
\end{align}
That is,
\begin{align*}
	\liminf_{\vep\rt 0} \mdpscb \log \PP(\Up_{\vep}(f) \in O) \geq -I_f(O), \ \ \mbox{for every open set } \ O \in C([0,T],\R^n) ,\\
\end{align*}
and
\begin{align*}
	\limsup_{\vep\rt 0} \mdpscb \log \PP(\Up_{\vep}(f) \in C) \leq -I_f(C), \ \ \mbox{for every closed set } \ C \in C([0,T],\R^n).
\end{align*}

Moreover,  $\Xi^R_\vep(f)/\mdpsc$ also satisfies a LDP  on $C([0,T],\R^n)$ with the same speed and the same rate function if
 \begin{itemize}
 	\item Assumption \ref{assum-f-2}-(A) holds and $\Delta(\vep)$ is chosen such that $$\min\lf\{(\Delta(\vep)/\vep)^{\nu/2}/\sqrt\vep, (\Delta(\vep)/\vep)^{\nu_f/2}/\mdpsc\ri\} \rt 0,$$ 
 	as $\vep \rt 0$ (thus, in particular, if $(\Delta(\vep)/\vep)^{\tilde \nu/2}/\sqrt \vep \rt 0$, where $\tilde \nu = \nu \wedge \nu_f$); OR,
 	\item Assumption \ref{assum-f-2}-(B) holds with $p_0' \leq \alpha$, and $\Delta(\vep)$ is chosen such that 
 	$(\Delta(\vep)/\vep)^{\nu/2}/\sqrt\vep \rt 0$ as $\vep \rt 0$.
 	
 \end{itemize}

Here for a set $A$, $I_f(A) = \inf_{x\in A}I_f(x)$.

\end{theorem}
		
To prove the above theorem we will actually prove the {\em Laplace principle} which is equivalent to proving LDP \cite[Section 1.2]{DE97}. In other words, we will show that for all $F \in C^1_b(C([0,T]:\R^n))$

\begin{align}\label{Lap-P}
	\lim_{\vep \rt 0}  \mdpscb \ln \EE \lf[\exp\Big(-F(\Up_\vep(f))/\mdpscb \Big)\ri] = - \inf_{\xi \in C([0,T], \R^d)} [I(\xi) + F(\xi)].
\end{align} 		

Some remarks are now in order.

\begin{remark}		\label{rem-MDPZ-1}
	{\rm The following observations and comments are clear from the proofs of the above theorems.
	\begin{itemize}

\item To simplify the notations a bit in the proof, we assumed that $b$ and $\s$ have same H\"older exponent $\nu$. Of course, for many stochastic models, this might not be true. If $b$ and $\s$ are  H\"older  continuous with  H\"older exponents $\nu_b$ and $\nu_\s$ respectively, then for the above results to hold  the discretization step $\Delta(\vep)$ needs to be chosen such that $(\Delta(\vep)/\vep)^{\nu/2}/\sqrt \vep \rt 0$ with $\nu = \nu_b\wedge \nu_\s$. 

\item If $\s(x) \equiv \s$ (a constant), then $\nu \equiv \nu_b$, and for the MDP result to hold, we only need $(\Delta(\vep)/\vep)^{\nu/2}/\mdpsc \rt 0$. This means that the discretization steps can be chosen slightly bigger. Also, in this case, the assumptions on growth of $D^2u$ (Condition \ref{u-reg-assum}-(v) and Assumption \ref{assum-exp-u}-(iv)) are not needed.   
	
\item Finally, a rather obvious comment is that if we are only considering homogeneous functionals, that is, we assume $f$ is only a function of $x$ and not of $t$, then the assumption on moduli of continuity of $f$, $u$ and $Du$ are not needed. In other words, for MDP of $\int_0^\cdot f(Z^\vep(s))ds$, $q_0$ in  Assumption \ref{assum-f}-(ii) and $q_1, q_2$ in (iii) -(iv) of Condition \ref{u-reg-assum} can be assumed to be $0$.

\end{itemize}

}
\end{remark}

\np
As mentioned, not surprisingly, the same techniques prove a moderate deviation principle of the inhomogeneous functionals of the original process $X^\vep$ under less restrictive conditions.  Indeed, some of the estimates that are essential for study of MDP for $\Xi_\vep(f)$ do not come up while considering the case of $\G_\vep(f)$, defined by
$$\Gamma_\vep(f) = \int_{\R^d\times [0,\cdot]}f(s,x)\Gamma_\vep(dx\times ds)=   \int_0^{\cdot} f(s, X^\vep(s))ds.$$
Some assumptions can be removed (including H\"older continuity of $b$ and $\s$, provided existence and uniqueness of solution $X$ are available), and some complex arguments could be simplified as a result. 

\begin{theorem}\label{th:mdpX} Let  $f :[0,\infty)\times\R^d \rt \R^n$ satisfy Assumption \ref{assum-f}  with $\modu(\Delta) = O(\sqrt \Delta)$. Let $X^\vep$ be the unique solution to \eqref{SDE1} Then under Condition \ref{cond_SDE_inv} (with $\alpha>0$), (i) - (iv) of Condition \ref{u-reg-assum}, and (i) - (iii) of Assumption \ref{assum-exp-u}, as $\vep\rt 0$, $\lf\{U_\vep(f)\equiv \f{1}{\mdpsc}\Gamma_\vep(f) =\f{1}{\mdpsc}\int_0^{\cdot} f(s, X^\vep(s))ds\ri\}$ satisfies a LDP on $C([0,T],\R^n)$ with speed $\mdpscb  \equiv \vep/\mdpscii$ and rate function $I_f$ given by \eqref{rate2}.
	\end{theorem}

\subsection{Variational representation and controlled processes} \label{sec:var}
Here, we briefly describe the result on variational representation of expectations of exponential functionals of $\Up_\vep(f)$ and the control process associated with $Z^\vep$. These form the backbone of a weak convergence approach to large deviation asymptotics.

Let $\mathcal{P}$ denote the predictable $\sigma$-field on $[0,T]\times\Omega$
associated with the filtration $\left\{  \mathcal{F}_{t}:0\leq t\leq
T\right\}$.
Let 
\begin{align*}
	P^M_2 \equiv \{h: [0,T] \rt \R^m:  \int_0^T \|h(s)\|^2ds \leq M\},
\end{align*}	
and 
\begin{align*}
  \SC{P}_2^M \equiv \left\{  \psi:\psi\text{ is }\mathcal{P}%
  \backslash\mathcal{B}(\mathbb{R}^{m})\text{ measurable and }\psi\in P_{2}%
  ^{M}\text{, a.s. }\mathbb{P}\right\}  ,\quad\mathcal{P}_{2}\doteq\cup
  _{M=1}^{\infty}\mathcal{P}_{2}^{M},
  \end{align*}
  Then  by the variational representation  and an application of Girsanov's theorem \cite{BoDu98, BDM08},
  \begin{align}\label{vrep}
  	-\mdpscb \ln \EE \lf[\exp\Big(-F(\Up_\vep(f))/\mdpscb \Big)\ri] = \inf_{\psi \in \SC{P}_2}\EE\lf\{\f{1}{2} \int_0^T \|\psi(s)\|^2 ds + F(\bar \Up^{\psi}_\vep(f))\ri\},
  \end{align}	
  where $\bar \Up^{\psi}_\vep(f)(t) = \f{1}{\mdpsc} \int_0^t f(s,\bar Z^\psi_\vep (s))ds$ and $\bar Z^\psi_\vep$ solves the controlled stochastic equation:
  \begin{align}\non
  	\bar Z^\psi_\vep(t) =&\ x_0 + \f{1}{\vep}\int_0^t b(	\bar Z^\psi_\vep(\vr_\vep(s)))ds+  \f{1}{\sqrt \vep}\int_0^t\s(\bar Z^\psi_\vep(\vr_\vep(s))) dW(s) \\ \label{ctrl-EuSDE1}
  	& \ + \f{\mdpsc}{\vep} \int_0^t \s(	\bar Z^\psi_\vep(\vr_\vep(s)))\psi(s) ds.
  \end{align}
 Similarly,  
  \begin{align*}
  	-\mdpscb \ln \EE \lf[\exp\Big(-F(\Xi^{R,\psi}_\vep(f)/\mdpsc)/\mdpscb \Big)\ri] = \inf_{\psi}\EE\lf\{\f{1}{2} \int_0^T \|\psi(s)\|^2 ds + F(\bar\Xi^{R}_\vep(f)/\mdpsc)\ri\},
  \end{align*}	
  where $\bar\Xi^{R,\psi}_\vep(f)(t) = \int_0^t f(\vr_\vep(s),\bar Z^\psi_\vep (\vr_\vep(s)))ds.$

Since $P_{2}^{M}$ is a closed ball in $L^{2}([0,T])$, it is compact under the
weak topology, which is metrizable, and throughout the paper, this topology will be used on $P_{2}^{M}$. The overbar on a process will denote its controlled version, and for most part of the paper,  superscripts like $\psi$ will be dropped from the notation of the controlled process, for  convenience.

 \subsection{Poisson equation} \label{sec:poi}
  Fix $t>0$. For each $l=1,2,\hdots, n,$ we consider the Poisson equation
\begin{align}\label{Pois-1}
\SC{L}u_l(t,\cdot)(x) = -f_l(t,x),
\end{align}
where $f = (f_1,f_2,\hdots,f_n)$ and $\SC{L}$, defined by \eqref{gen-diff}, is generator of the diffusion process $X$.
From \cite[Theorems 1, 2]{PaVe01}, under Condition \ref{cond_SDE_inv}, for each $l$ and $t>0$,  \eqref{Pois-1} admits a unique solution $u_l(t,\cdot)$ in the class of functions belonging to $W^{2,p}_{loc}$ for any $p>1$. $u_l(t,\cdot)$ is given by \begin{align}\label{u-form}
u_l(t,x) = \int_0^\infty P_sf_l(t,\cdot)(x) ds = \int_0^{\infty} \int_{\R^d} f_l(t,y)P_s(x,dy)ds,
\end{align}
where recall that $P_s(x,dy)$ denotes the transition kernel of the diffusion process $X$ given by \eqref{SDE0}, and by a slight abuse of notation, $P_s$ is also used to denote the corresponding semigroup. \\

The central idea is to show that the integral in \eqref{u-form} is convergent, and $u$ defined by \eqref{u-form} is continuous and does not increase rapidly to infinity. Also, note that by choosing $p>d$ and using Sobolev embedding theorem \cite[Section 7.7]{GiTu01}, it follows that for each $t>0$, $Du(t, \cdot)$ is continuous. Moreover, if we assume that the coefficients $b$ and $a$ are $C^1$, $f$ is (weakly) differentiable and  $\sup_{t\leq T}\|Df(t,x)\| \leq \const(T) (1+\|x\|)^{p_0'}$ for some $p_0' \in \R$, then by \cite[Theorem 9.19]{GiTu01}, it follows that $f \in W^{3,p}_{loc}$ for all $p>1$. As before, choosing $p>d$ and using Sobolev embedding theorem, it  now follows that $D^2u(t,\cdot)$ is continuous. \\

We will make the following assumptions on regularity of $u$.

\begin{condition}\label{u-reg-assum}
There exist a constant $\const_1(T)$ and exponents $p_1, p_2, p_3, q_1$ and $q_2$ such that for each $l=1,2,\hdots,n$, the following estimates hold:	
\begin{enumerate}[(i)]
\item $\dst \sup_{t\leq T}\|u_l(t,x)\|  \leq \ \const_1(T)\lf(1+\|x\|\ri)^{p_1}$,
\item $\dst \sup_{t\leq T}\|\nabla u_l(t,x)\|  \leq \ \const_1(T)\lf(1+\|x\|\ri)^{p_2}$,
\item $\dst  \omega_{u_l}(\Delta,x) \doteq \ \sup_{\{|t-s|\leq \Delta, \ 0\leq s,t \leq T\}}\|u_l(t,x) - u_l(s,x)\| \leq\ \const_1(T) \modu(\Delta) (1+\|x\|)^{q_1}$,
\item $\dst \omega_{\nabla u_l}(\Delta,x) \doteq \ \sup_{\{|t-s|\leq \Delta, \ 0\leq s,t \leq T\}}\|\nabla u_l(t,x) - \nabla u_l(s,x)\| \leq\ \const_1(T) \modu(\Delta) (1+\|x\|)^{q_2}$,
\item $\dst \sup_{t\leq T}\|D^2 u_l(t,x)\|  \leq \ \const_1(T)\lf(1+\|x\|\ri)^{p_3}.$
\end{enumerate}
\end{condition}


%

\begin{assumption}\label{assum-exp-u} The exponents in Condition \ref{u-reg-assum} satisfy the following bounds:
\begin{align*}
&(i) \ p_1 \leq (1+\alpha)/2, \qquad (ii) \ p_2<\alpha \mbox{ if } \alpha \leq 1, \mbox{ and } \ p_2\leq (1+\alpha)/2   \mbox{ if }  \alpha>1, \\
 &(iii)\ \max\{q_0/2, q_2\} \leq \alpha,  \quad  q_1 \leq 2\alpha1_{\{\alpha \leq 1\}}+\alpha1_{\{\alpha > 1\}}, \qquad (iv)\ p_3 \leq \alpha.
\end{align*}
%
\end{assumption}

\np
For some models, the solution $u$ can be computed directly and the above assumptions can be directly checked. 

\begin{example}{\rm
Consider the following 1-dimensional SDE:
\begin{align*}
X(t) = x_0+ \int_0^t b(X(s))ds + W(t), \quad x_0 \in \R,
\end{align*}
where $xb(x) = -|x|^{1+\alpha}$. Then clearly, $|b(x)| \sim |x|^{\alpha}$. Let $\pi$ denote the invariant probability measure. Let $f(x) = -b(x)$ (notice that $\int_{\R}b(x)\pi(dx) = 0$). Then $u(x) = x$ and clearly, if $\alpha \geq 1$, Assumption \ref{assum-exp-u} holds.
}
\end{example}

However, in most models, a closed form expression of the Poisson equation is not available, and a general result describing the different exponents of Condition  \ref{cond_SDE_coef-ex} is needed. Toward this end, \cite[Theorem 2]{PaVe01} could be useful.

\begin{remark} \label{Ver-rem}

{\rm However, as mentioned in the introduction, we do note that the proof of  the estimate on the growth rate of $Du$, \cite[Theorem 2, eq. (21)]{PaVe01}, requires the drift $b$ to be globally bounded. This is not explicitly mentioned  in \cite{PaVe01}, where $b$ is said to be locally bounded (although in the statement of Theorem 1 of \cite{PaVe01}, it did mention once that the constant depends on $\sup_{i,x}|b_i(x)|$). To see why this is indeed the case, first observe that
 the proof uses the result on interior $L^p$-estimates of solutions of the elliptic equation from Gilbarg and Trudinger \cite[Theorem 9.1]{GiTu01}. However, the constant in this result depends on the bounds of the coefficients, $b$ and $a$, in the domain of interest, $\Om$. The coefficient $a$ is assumed to be bounded, but the drift term $b$ in most examples will be not. More specifically, since the domain $\Om = B(x,1)$ in the part (e) of proof of  \cite[Theorem 9.1]{GiTu01}, the constant $C$ in \cite[Eq. (9.4)]{GiTu01}, and hence the constant $C'$ in the first display of \cite[Page 1070]{PaVe01} will actually depend on $x$. For example, for Ornstein-Uhlenbeck SDE, where $b(x) \sim -x$, it is not hard to see following the chain of arguments leading to \cite[Eq. (9.4)]{GiTu01} that this particular $C \sim x^2$.  This affects the growth rate of the gradient of the solution $u$ in \cite[Theorem 2, eq. (21)]{PaVe01}.
 
 The statement as stated in \cite[Theorem 2, eq. (21)]{PaVe01} might still be true for more general $b$, but unfortunately, we cannot find a way to adapt the proof given by Pardoux and Veretennikov or find an alternate proof -- except in one-dimension. For one-dimensional SDEs, the original statement of \cite{PaVe01} (at least, a very similar one) is indeed true, and we were able to find an alternate way to prove it.  For multi-dimensional SDEs, through a closer inspection of the proof of \cite[Theorem 9.1]{GiTu01}, we were able to give a modified statement where the growth rate of $Du$ needed to be changed.

 This modified statement is the content of Proposition \ref{u-est} below. Just like the techniques used in \cite{PaVe01}, its proof relies on \cite[Theorem 9.1]{GiTu01}, or more specifically, a version of it.  This  version, under the growth condition of the coefficients of $\SC{L}$ (c.f. Condition \ref{cond_SDE_coef-ex}), provides a more closer look into the $L^p$-estimate of the solution $u$, which is needed in our paper. For sake of completeness the proof of this version of \cite[Theorem 9.1]{GiTu01} is presented in  Lemma \ref{Lemma-D2v} in the Appendix.

}
\end{remark}


\begin{proposition}\label{u-est} Suppose that Condition \ref{cond_SDE_inv}, Assumption \ref{assum-f}, and Condition \ref{cond_SDE_coef-ex} hold.
Then $u\in C^1(\R^d,\R^n)$, and for each $l=1,2,\hdots,n$, (i) - (iv) of Condition \ref{u-reg-assum} hold, with
the following relations between the exponents:
\begin{align*}
& p_1 = (p_0-\alpha+1)^+,\quad p_2 = \max\{p_1+2\bar\alpha,  p_0\}, \quad  q_1 = (q_0-\alpha+1)^+, \quad q_2 = \max\{q_1+2\bar\alpha), q_0\}.
\end{align*}
Here $p_0$ and $q_0$ are as in Assumption \ref{assum-f}. 

Furthermore, assume that $b$ and $a$ are in $C^1(\R^d)$, $\|D^2a\|_\infty < \infty$, $\|Db\|\leq   \cnst (1+ \|x\|^{\bar \alpha})$, $f$ is (weakly) differentiable and  $\sup_{t\leq T}\|Df(t,x)\| \leq \const(T) (1+\|x\|)^{p_0'}$ for some $p_0' \in \R$ and some constant $\const(T)>0$. Then $u \in C^2(\R^d,\R^n)$, and Condition \ref{u-reg-assum}-(v) also holds with $$p_3 =\max \{p_0+2\bar\alpha, p_1+4\bar \alpha\}.$$
\end{proposition}

\begin{proof} The fact that $u \in C^1(\R^d,\R^n)$ (or $C^2(\R^d,\R^n)$, under  additional hypotheses) follows from the discussion above Condition \ref{u-reg-assum}.

Condition \ref{u-reg-assum}-(i) follows from \cite[Theorem 2]{PaVe01}. Condition \ref{u-reg-assum}-(ii) and Condition \ref{u-reg-assum}-(v) now follow from Lemma \ref{Lemma-D2v}, Remark \ref{rem:deriv-v-bd} and Lemma \ref{Lemma-D2v-pt}, applied to $f_l$ and $u_l$ for each $l=1,2, \hdots,n$. 

  In fact from the proof of \cite[Theorem 2]{PaVe01}, it is clear that  if $\|g_\kappa(x)\| \leq \kappa (1+\|x\|^{p_0})$ for some parameter $\kappa$, and $u_\kappa$ given by \eqref{u-form} (with $f_l$ replaced by $g_\kappa$) is the solution to the Poisson equation $\SC{L}u_\kappa = -g_\kappa$, then 
\begin{align}\label{u-bd-obs}
\|u_\kappa(x)\| \leq \bar C \kappa(1+\|x\|^{p_1}),
\end{align}
 where the constant  $\bar C$ does not depend on $\kappa$.  Now notice that for a fixed $l$, $t$ and $\Delta$
 $$u_l(t+\Delta,x) - u_l(t,x) = \int_0^{\infty} P_s \bar f^{t,\Delta}_l(x)ds$$
 is the solution to the equation $\SC{L}v = -\bar f^{t,\Delta}_l$, where $\bar f^{t,\Delta}_l(x) \doteq f(t+\Delta,x) - f(t,x)$ satisfies
 $\|\bar f^{t,\Delta}_l(x)\| \leq \const(T)\modu(\Delta)(1+\|x\|)^{q_0}$ (by Assumption \ref{assum-f}-(ii)). It follows from \eqref{u-bd-obs} that Condition \ref{u-reg-assum}-(iii) holds, and again  Condition \ref{u-reg-assum}-(iv) follows from  Lemma \ref{Lemma-D2v}, Remark \ref{rem:deriv-v-bd} and Lemma \ref{Lemma-D2v-pt}.


\end{proof}

Although the above theorem is nice and might be the only tool available to check Condition \ref{u-reg-assum} and Assumption \ref{assum-exp-u} for many stochastic models, it is not optimal. Consider an one dimensional model, where we have $xb(x)=  -|x|^{1+\alpha}$. Clearly, then it is natural to assume that the drift $b$ satisfies, $|b(x)| \sim |x|^{\alpha}$. Then if Proposition \ref{u-est} is used to determine the exponents of $u, Du$, then it follows from Assumption \ref{assum-exp-u}  that $f$ has to be chosen from the class for which $p_0 <-1$, that is, $|f(x)| \sim 1/(1+|x|)$. This restricts the applicability of the theorem to a  smaller class of functions than desired. 

However, for  one-dimensional SDEs, Proposition \ref{u-est} could actually be vastly improved, and tighter bounds on growth rate of $u$ and $u'$ can be obtained. This result is presented in Proposition \ref{u-reg-1D}. This makes our MDP results applicable to a wide class of stochastic models, and to  functions $f$ having polynomial-like growth -- without doing any extra work for checking regularity of Poisson equation.


\subsection*{ Regularity of Poisson equation for one dimensional SDE}

When $d=1$, the invariant distribution of $X$ is given by
\begin{align*}
\pi(z) =  \f{\cnst}{a(z)}\exp\lf(2\int_0^z \f{b(y)}{a(y)}dy\ri),
\end{align*}
where $\cnst$ is the normalizing constant, and by a slight abuse of notation, we used $\pi(\cdot)$ to denote the density of the invariant distribution $\pi$ . In this case the solution of the Poisson equation, $u(t,\cdot)$, have the following explicit representation:
\begin{align}\label{poi-eq-sol}
u_f(t,x) \equiv u(t,x) = -\int_{-\infty}^x\f{2}{a(z)\pi(z)} \int_{-\infty}^z f(t,y)\pi(y)dy \ dz.
\end{align}

Since in  \eqref{poi-eq-sol}, $t$ is just a parameter, for notational convenience, we will drop $t$ from the following result.

\begin{assumption} \label{assum-poi-1d}There exist exponents $p$, $\theta (> -1)$ and constants $\mfk{c}_0, \mfk{c}_1$ and $\mfk b$ such that 
\begin{enumerate}[(i)]
\item $|f(x)| = O(|x|^{p_0}), |b(x)| = O(|x|^\alpha)$
\item $|f(x)/b(x)| = O(|x|^{p_0-\alpha})$
\item $\mfk c_0 |x|^\theta \leq |b(x)/a(x)| \leq \mfk c_1 |x|^\theta$, for $|x| \geq \mfk b$
\end{enumerate}
\end{assumption}

\begin{proposition} \label{u-reg-1D}Suppose that Condition \ref{cond_SDE_inv}-(i) and Assumption \ref{assum-poi-1d}  hold.   
Then, $u_f$ defined by \eqref{poi-eq-sol}, is a solution to the Poisson equation, and 
\begin{enumerate}[(i)]
\item $|u_f(x)| = O(|x|^{p_0-\alpha+1})$ for $p_0-\alpha \neq -1$; if $p_0 - \alpha=-1$, then $|u(x)| = O(|\ln x|)$;
\item $|u_f'(x)| = O(|x|^{p_0-\alpha})$;
\item $|u_f^{''}(x)| = O(|x|^{p_0-\alpha+\theta})$.
\end{enumerate}
\end{proposition}

\begin{proof} Direct computation shows that $u_f$ defined by \eqref{poi-eq-sol}, is a solution to the Poisson equation.
Notice that
\begin{align}\label{pi-deriv-eq}
(\pi(z)a(z))' = 2b(z)\pi(z).
\end{align}
Also, it is clear from (a) Assumption \ref{assum-poi-1d}-(iii), (b) the expression of invariant distribution $\pi$, and (c) the fact that $\theta+1 >0$, that for any $m$
\begin{align}\label{ap-lim}
x^ma(x)\pi(x) \rt 0, \quad \mbox{ as } |x| \rt \infty.
\end{align}

Notice that since $f$ is centered, that is $\pi(f)=0$,
\begin{align}\label{u-poi-rep}
u'_f(x) = & -  \f{2}{a(x)\pi(x)} \int_{-\infty}^x f(y)\pi(y)dy =  \f{2}{a(x)\pi(x)} \int_{x}^\infty f(y)\pi(y)dy 
\end{align}

Since for $|x| > B$ ($B$ was introduced Condition \ref{cond_SDE_inv}-(i)) , $xb(x)<0$, we have that $b(x) <0$ for all $x>B$  and $b(x)>0$ for $x < - B$
 For our purposes, the second equality in \eqref{u-poi-rep} needs to be used when $x> B$, and the first  needs to be used when $x< - B$.

We first consider the case when $x> B$. 
Observe by Assumption \ref{assum-poi-1d}-(ii) and the fact that for $x>B$, $|b(x)| = -b(x)$, we have for some constant $\mfk{c}_2$
\begin{align*}
|u_f'(x)| \leq &\   \f{2}{a(x)\pi(x)} \int_{x}^\infty \lf|\f{f(y)}{b(y)}\ri| |b(y)| \pi(y) dy \\
            \leq & -\f{2\mfk{c}_2}{a(x)\pi(x)} \int_{x}^\infty y^{p_0-\alpha} b(y) \pi(y) dy.
\end{align*}
If $p_0 \leq \alpha$, then by \eqref{pi-deriv-eq} and \eqref{ap-lim}, it follows that
$$|u_f'(x)|  = O(|x|^{p_0-\alpha}).$$
If $p_0 > \alpha$, then we use  \eqref{pi-deriv-eq} and integration by parts to get,
\begin{align*}
|u_f'(x)| \leq     &\ -\f{\mfk{c}_2}{a(x)\pi(x)}\lf[y^{p_0 -\alpha}a(y)\pi(y)\Big|_{x}^\infty - \int_{x}^{\infty}y^{p_0-\alpha-1} a(y)\pi(y)dy\ri]\\
         =& \ \mfk{c}_2 x^{p_0-\alpha} + \f{\mfk{c}_2}{a(x)\pi(x)}\int_{x}^{\infty}y^{p_0-\alpha-1}a(y)\pi(y)dy\\
%
       = & \mfk{c}_2 x^{p_0-\alpha}+  \f{\mfk{c}_2}{a(x)\pi(x)}\int_{x}^{\infty}y^{p_0-\alpha-1}\f{a(y)}{|b(y)|}|b(y)|\pi(y)dy\\
      \leq & \ \mfk{c}_2 x^{p_0-\alpha} -  \f{\mfk{c}_2/\mfk{c}_0}{a(x)\pi(x)}\int_{x}^{\infty}y^{p_0-\alpha-1-\theta}b(y)\pi(y)dy.
      \end{align*}
If  $p-\alpha-\theta \leq 1$, then it follows that
\begin{align*}
|u'_f(x)| \leq &\ \mfk{c}_2 x^{p_0-\alpha}-  \f{\mfk{c}_2x^{p_0-\alpha-1-\theta}/\mfk{c}_0}{a(x)\pi(x)}\int_{x}^{\infty}b(y)\pi(y)dy\\
       = &\ 2\mfk{c}_2x^{p_0-\alpha}+  \mfk{c}_2x^{p-\alpha-1-\theta}/2\mfk{c}_0 = O(|x|^{p_0-\alpha})
\end{align*}
where we have used \eqref{pi-deriv-eq} and \eqref{ap-lim}.  
If   $p_0-\alpha-\theta >1$, then let $k > 1$ be the smallest integer such that $p_0-\alpha-\theta \leq k$. Now we repeat the integration by parts technique $k$ times to prove the assertion.

If $x<-B$ then we use the first equality in \eqref{u-poi-rep} and the same techniques to prove the assertion.  

To prove the bound on $u''_f$ simply observe that
\begin{align*}
|a(x)u_f''(x)| \leq &\  |b(x)u'_f(x)| + |f(x)|.
\end{align*}
and now the assertion follows from (ii) and (iii) of Assumption \ref{assum-poi-1d}.

\end{proof}

\begin{example}\label{ex:OU}{\rm
Let $X$ be the {\em mean-reverting} Ornstein-Uhlenbeck process satisfying
\begin{align*}
X(t) = x_0 + \kappa\int_0^t(\mu - X(s))ds + \s W(t).
\end{align*}
The invariant distribution of $X$ is  of course the Normal$(\mu, \s^2/2\kappa)$. Here $\alpha = \bar\alpha=\nu = 1$. Then for $f$, with $|f(x)| \leq \const (1+\|x\|)^{p_0}$, Proposition \ref{u-reg-1D} gives the exponents of Condition \ref{u-reg-assum}: $p_1 = p_0,\ p_2 = p_0-1$. Note that $p_3$ is not needed as the diffusion coefficient is constant $\s$ (see Remark \ref{rem-MDPZ-1}). Thus if $p_0 \leq 1$, then  Assumption \ref{assum-exp-u} holds, and the MDP result (Theorem \ref{th:mdpZ}) holds for such functions $f$. 
}
\end{example}

\begin{example}{\rm
We next consider the Cox-Ingersoll-Ross (CIR) model, which describes the dynamics of the instantaneous interest rates. Let $X(t)$ be the solution to
\begin{align*}
dX(t) = \kappa(\mu-X(t))+ \s\sqrt{X(s)} dW(s), \quad X(0) =x_0>0,
\end{align*}
where $\kappa, \mu$ and $\s$ are positive constants.
Then it is a known fact that if $\kappa \mu \geq \s^2/2$, then $X(t)$ takes values in $(0,\infty)$. The invariant distribution of $X$ is given by {\em Gamma$( 2\mu\kappa/\s^2, 2\kappa/\s^2)$}, that is,
$$\pi(x) = \f{(2\kappa/\s^2)^{2\mu\kappa/\s^2}}{\G(2\mu\kappa/\s^2)} x^{2\mu\kappa/\s^2-1}\exp(-2\kappa x/\s^2 ), \quad x>0.$$
Here $\alpha=1$, but $\s(x) = \s\sqrt x$ has degeneracy at $0$, which, however, $X$ never hits. But as mentioned before, degeneracy is not an issue, if a unique invariant measure and solution of Poisson equation  exist, which they do in this case. For $f$, with $|f(x)| \leq \const (1+\|x\|)^{p_0}$, Proposition \ref{u-reg-1D} gives the desired exponents: $p_1 = p_0, \ p_2 = p_0 - 1, p_3 = p_0.$ Assumption \ref{assum-exp-u} needs to be modified for the MDP and CLT results, since $\s(x)$ is not bounded in this case. However, with little extra effort, the right assumption to work with in this case can be formulated. 
Indeed, for the MDP result to hold, one needs  $\max\{p_1, p_2, p_3\} \leq 3/4$.
This means that Theorem \ref{th:mdpZ} will apply to functions $f$ with $p_0 \leq 3/4$ (and $q_0\leq 3/4$ if we are considering inhomogeneous functionals as well), and the discretization step $\Delta(\vep)$ needs to be chosen such that $\Delta(\vep)/\vep^3 \rt 0$.  The CLT result for the discretized process $Z^\vep$ (Theorem \ref{th:cltZ}), of course, holds  for any $f$ satisfying Assumption \ref{assum-f} for some $p_0$, and $\modu(\Delta) = o(\sqrt\Delta),$ under the same choice of $\Delta(\vep)$.

}
\end{example}

\begin{example}\label{ex:Gomp} {\rm
Let $X(t)$ be the {\em geometric mean-reversion process}or the {\em Gompertz diffusion model}  defined as the solution to the SDE:
\begin{align*}
dX(t) = \kappa(\mu- \ln X(t))X(t) dt + \s X(t) dW(t), \quad X(0) = x_0 > 0,
\end{align*} 
where $\kappa, \mu$ and $\s$ are positive constants. This model is used not only in commodity pricing \cite{Sch97}, but also in determining freight rates in shipping \cite{Tve97}. It is also used to model the {\em in vitro} tumor growth \cite{Bis08} with $X$ representing the volume of tumor and the drift parameters capturing the growth rate.
The solution $X$ can be written explicitly in this case and, in particular,  $X(t)>0$.  The invariant distribution, $\pi$, of $X$ is {\em log-normal}$(\mu - \f{\s^2}{2\kappa},  \f{\s^2}{2\kappa})$. Since the drift term is not Lipschitz a CLT or MDP result for $\int_0^\cdot f(Z^\vep(s))ds$, for the Euler discretized process $Z^\vep$, will not directly follow from our results. However, we do note that  the MDP result for $\int_0^\cdot f(X^\vep(s))ds$ (Theorem \ref{th:mdpX}) still holds (after some simple adjustments to its hypotheses) since its validity does not require Lipschitz or H\"older continuity of the coefficients.)

But interestingly, with a little trick, we can still get both CLT and MDP result for processes of  the form $\int_0^\cdot f(\hat Z^\vep(s))ds$  for a slightly different discretization scheme. Indeed, the transformation $x \rt \ln x$, transform the above SDE into an OU process, given by 
\begin{align*}
dY(t) = \kappa\lf(\mu- \f{\s^2}{2\kappa} -Y(t)\ri) dt + \s dW(t), \quad Y(0) = \ln x_0.
\end{align*} 
If $Y^\vep$ denotes the corresponding scaled process (as in \eqref{SDE1}), and $\hat Y^\vep$ denotes the (continuous) Euler discretization of $Y^\vep$, then defining $\hat Z^\vep(\cdot) = \exp(\hat Y^\vep(\cdot))$ gives a discretized version of $X^\vep$. The desired MDP and CLT results for $\int_0^\cdot f(\hat Z^\vep(s))ds$ now apply to all functions $f$ satisfying $|f(x)| \leq \const(1+|\ln x|)^{p_0}$, with MDP requiring $p_0 \leq 1$.

}
\end{example}





\setcounter{equation}{0}
\renewcommand {\theequation}{\arabic{section}.\arabic{equation}}

\section{Equivalent forms of the rate function} \label{sec:eq-rate}
In this section we describe two equivalent forms of the rate function $I_f$ that will be convenient to work with in the proof of upper and lower bounds of Laplace principle.

Let $\l_T$ denote the Lebesgue measure on $[0,T]$. Let $\mathbb{B}_T= [0,T]\times\R^d\times \R^m,$ and let $\SC{M}_1(\B_T)$ be the space of finite measures $R$ on $\B_T$ such that  $R_{(1)}  = \l_T$ and $R_{(2,3|1)}$ is a probability measure on $\R^d\times\R^m$. Here for $i=1,2,3$, $R_{(i)}$ denotes the $i$-th marginal of $R$ and $R_{(i,j|k)}$ denotes the conditional distribution of $i$-th and $j$-th coordinate given the $k$-th coordinate. 

For each $\xi \in C([0,T], \R)$, let $\SC{R}_\xi$ denote the family of measures $R \in \SC{M}_1(\B_T)$ such that
\begin{align}
\label{Rset-1}
&	\int_{\B_T} \|z\|^2 R(d\y)  <\infty;\\
\label{Rset-2}
&	\xi(t) = \int_{\B_t} D u(s, x) \s(x) z R(d\y);\\
\label{Rset-3}
&	\int_{\B_t} \SC{L} g(x) R(d\y) =0, \quad \mbox { for all } t \in [0,T], \ g \in C^2_b(\R^d, \R),
\end{align}	
where the $l$-th row of the derivative matrix $Du$ is given by 
$$(Du(s, x))_{l*} = \nabla^T u_l(s, x) = (\partial _1 u_l(s, x), \partial _2 u_l(s, x), \hdots, \partial _d u_l(s, x))$$
 and a typical tuple $(s,x,z) \in \B_T$ is denoted by $\y$. Define $\bar I_f : C([0,T],\R^d) \rt [0,\infty]$ by
	\begin{align}\label{rate0}
		\bar I_f(\xi) = \inf_{R \in \SC{R}_\xi} \lf\{\f{1}{2} \int_{\B_T} \|z\|^2 R(d\y)\ri\}.
	\end{align}

Next, let $\SC{A}_\xi$ denote the space of $\phi \in L^2(\R^d\times [0,T], \pi\times \l_T)$ such that 
\begin{align*}
\xi(t) = \int_{\R^d \times [0,t]} D u(s, x) \s(x) \phi(x,s) \pi(dx) ds.
\end{align*}	
Define $\hat I_f : C([0,T],\R^d) \rt [0,\infty]$ by
\begin{align}\label{rate1}
\hat I_f(\xi) = \inf_{\phi \in \SC{A}_\xi} \lf\{\f{1}{2} \int_{\R^d \times [0,T]} \|\phi(x,s)\|^2 \pi(dx) ds\ri\}.
\end{align}	

\begin{lemma}\label{M-mat-alt}
$M_f(t) = \int_{\R^d}Du(t, x)a(x)(Du(t, x))^T \pi(dx)$, where $a=\s\s^T$, $u$ is defined by \eqref{Pois-1} and $M_f$ is defined by \eqref{M-mat-def}.
\end{lemma}
	
\begin{proof}
Fix $t>0$. By It\^o's lemma, we have
\begin{align*}
u_i(t, X(r))  = &\ u_i(t, X(r) ) +\int_{0}^r\SC{L}u_i(t,\cdot)(X(s)) ds +\int_{0}^r \nabla^T u_i(t,  X(s))\s(X(s))dB(s)\\
= &\ u_i(t, X(r) ) - \int_{0}^r f_i(t_0, X(s)) ds +\int_{0}^r \nabla^T u_i(t,  X(s))\s(X(s))dB(s).
\end{align*}	
Then by integration by parts and observing that the last term on the right side is a martingale, we have, for any $t>0$, after taking expectation with $X(0)$ distributed as $\pi$
\begin{align*}
	\EE_\pi\lf(u_i(t, X(r))u_j(t, X(r))\ri) &= \ \EE_\pi\lf(u_i(t, X(0))u_j(t, X(0))\ri) - \int_0^r  \EE_\pi\lf(u_i(t, X(s))f_j(t, X(s))\ri) ds  \\ 
	& \ -  \int_0^r  \EE_\pi\lf(u_j(t, X(s))f_i(t, X(s))\ri) ds \\
	& \ + \int_{0}^r \EE_\pi\lf(\nabla^T u_i(t,  X(s))\s(X(s))\s^T(X(s))\nabla^T u_j(t,  X(s))\ri)ds.
\end{align*}	
The result now  easily follows from \eqref{u-form} and from the observation that the left side is equal to the first term on the right side as for all $r>0$, $X(r)$ is distributed as $\pi$  ($\pi$ is the invariant measure).
\end{proof}

\begin{theorem}\label{rate-equivalence}
	$\bar I_f = \hat I_f = I_f,$ where these quantities are defined in \eqref{rate0}, \eqref{rate1} and \eqref{rate2}, respectively.
\end{theorem}	

\begin{proof} We first show that $\bar I_f(\xi) = \hat I_f(\xi)$. Fix $\kappa>0$. Let $R \in \SC{R}_\xi$ be such that 
\begin{align}
\f{1}{2} \int_{\B_T} \|z\|^2 R(d\y) \leq \bar I_f(\xi) + \kappa.
\end{align}
Writing $R(d\y) = R_{(2,3|1)}(dx\times dz|s)ds$ and using \eqref{Rset-3}, for any $g\in C^2_b(\R^d, \R)$,  we have for a.a $s \in[0,T]$
\begin{align*}
0=& \int_{\R^d\times\R^m} \SC{L}g(x) R_{(2,3|1)}(dx\times dz|s) = \int_{\R^d} \SC{L}g(x) R_{(2|1)}(dx |s).
\end{align*}
By the uniqueness of $\pi$, we have $ R_{(2|1)}(dx |s) = \pi(dx)$ for a.a $s \in[0,T]$ and thus we have
$$R(d\y) =  R_{(3|1,2)}(dz|x,s)R_{(2|1)}(dx|s)ds = R_{(3|1,2)}(dz|x,s)\pi(dx)ds.$$
Define $\phi(x,s) = \int_{\B_t}z R_{(3|2,1)}(dz|x,s).$ Clearly, by Cauchy-Schwarz inequality,
\begin{align*}
\int_{\R^d\times[0,T]}\|\phi(x,s)\|^2 \pi(dx) ds \leq & \int_{\R^m\times \R^d\times[0,T]}\|z\|^2 R_{(3|2,1)}(dz|x,s) \pi(dx) ds = \int_{\B_T} \|z\|^2 R(d\y).
\end{align*}
Also,
\begin{align*}
\xi(t) = \int_{\B_t} D u(s, x) \s(x) z R(d\y) =& \int_{\B_t} Du(s, x) \s(x) z R_{(3|2,1)}(dz|x,s) \pi(dx) ds\\
     = &  \int_{\R^d\times [0,t]} Du(s, x) \s(x) \phi(s, x)\pi(dx) ds.
\end{align*}
Hence $\phi \in \SC{A}_{\xi}$.
\begin{align*}
\hat I_f(\xi) & \leq \f{1}{2}\int_{\R^d\times[0,T]}\|\phi(x, s)\|^2 \pi(dx) ds \leq \f{1}{2}\int_{\B_T} \|z\|^2 R(d\y) \leq \bar I_f(\xi) +\kappa.
\end{align*}
Since this is true for all $\kappa$, $\hat I_f(\xi) \leq \bar I_f(\xi)$. 

Conversely, for a fixed $\kappa>0$, let $\phi \in \SC{A}_{\xi}$ be such that
\begin{align} \label{rate2-lbd}
 \f{1}{2}\int_{\R^d\times[0,T]}\|\phi(x,s)\|^2 \pi(dx) ds \leq & \hat I_f(\xi) +\kappa.
\end{align}
Define the measure $R$ on $\B_t$ by 
\begin{align*}
R([0,t]\times A\times B) = \int_{A \times [0,t]} 1_{\{\phi(x,s) \in B\}}\pi(dx)ds.
\end{align*}
Clearly, by the definition of $R$,
\begin{align*}
\int_{\B_T}\|z\|^2 R(d\y) = \int_{\R^d\times [0,T]} \|\phi(x,s)\|^2 \pi(dx) ds,
\end{align*}
and 
\begin{align*}
\xi(t)=\int_{\R^d\times [0,t]} Du(s, x) \s(x) \phi(x,s)\pi(dx) ds =  \int_{\B_t} D u(s, x) \s(x) z R (d\y).
\end{align*}
Thus 
\begin{align*}
\bar I_f(\xi) & \leq  \f{1}{2}\int_{\B_T} \|z\|^2 R(d\y)= \f{1}{2}\int_{\R^d\times[0,T]}\|\phi(x,s)\|^2 \pi(dx) ds  \leq \hat I_f(\xi) +\kappa.
\end{align*}
Consequently, $\bar I_f(\xi) \leq \hat I_f(\xi).$

We next show that $\hat I_f(\xi) = I_f(\xi).$ Let $\kappa >0$ and  let $\phi \in \SC{A}_{\xi}$ be such that \eqref{rate2-lbd} holds.
Notice that
\begin{align*}
\dot \xi(s) = \int_{\R^d}  Du(s, x), \s(x) \phi(x,s) \pi(dx).
\end{align*}
By Lemma \ref{quad-max} (taking $(\Omega, \PP) = (\R^d, \pi)$,  $H(s, x) = Du(s, x)\s(x)$, $b = \dot\xi(s)$) for a.a $s$
\begin{align*}
(\dot \xi(s))^T M_f(s)^{-1} \dot \xi(s) \leq \int_{\R^d}\|\phi(x,s)\|^2 \pi(dx),
\end{align*}	 
where we used the fact that by Lemma \ref{M-mat-alt},
$$M_f(s) = \int_{\R^d} H(s, x)H(s, x)^T\pi(dx) = \int_{\R^d}Du(s,x)a(x)(Du(s,x))^T \pi(dx).$$
It now readily follows that $I_f(\xi) \leq \hat I_f(\xi)+\kappa$, and since this is true for all $\kappa>0$, we have  $I_f(\xi) \leq \hat I_f(\xi).$
Conversely, for an absolutely continuous $\xi$, define $\phi(x,s)= H^T(s,x)M_f(s)^{-1}\dot \xi(s)$. Clearly, $\phi \in \SC{A}_\xi$, and
\begin{align*}
\f{1}{2}\int_{\R^d\times [0,T]} \|\phi(x,s)\|^2\pi(dx) ds = &  \f{1}{2} \int_{[0,T]} (\dot \xi(s))^T M_f(s)^{-1} \dot \xi(s) ds.
   \end{align*}
It follows that $I_f(\xi) \geq \hat I_f(\xi)$.

\end{proof}

\setcounter{equation}{0}
\renewcommand {\theequation}{\arabic{section}.\arabic{equation}}

\section{Some estimates} \label{sec:est}
We begin by making the following simple observation.
Let $\{\tilde t_k\}$ be a partition of $[0,t]$ such that $\tilde t_k  - \tilde t_{k-1} = \tilde \Delta$. Let $\eta(s) = \tilde t_{k},$ if $\tilde t_k \leq s <\tilde t_{k+1}$.
Then for any locally integrable function $h$, by changing the order of integration, we get
	\begin{align}\non
		\int_0^t \int_{\eta(s)}^s|h(r)|dr\ ds = &\ \sum_k \int_{\tilde t_k}^{\tilde t_{k+1}}\int_{\tilde t_k}^s|h(r)|dr\ ds\\ \non
		= &\ \sum_k \int_{\tilde t_k}^{\tilde t_{k+1}}\int_{r}^{\tilde t_{k+1}}|h(r)|ds\ dr\\ 
		\leq &\ \tilde \Delta \sum_k \int_{\tilde t_k}^{\tilde t_{k+1}}|h(r)|dr =  \tilde \Delta \int_{0}^{ t}|h(r)| dr. 
		\label{int-diff-est}
	\end{align}

For the following result, we just need to assume Condition \ref{cond_SDE_coef-ex} and boundedness of $\s$. Actually, the latter boundedness assumption can easily be relaxed.
\begin{lemma}\label{lem-diff-est}
Let $\bar Z^\psi_\vep$ as in \eqref{ctrl-EuSDE1}, and assume that Condition \ref{cond_SDE_coef-ex} holds, and that $\s$ is bounded. Let $\Delta(\vep)$ be such that $\Delta(\vep)/\vep \rt 0$ as $\vep \rt 0$. Then for any $M\geq 0 $ and $m\geq 0$, there exist  $\vep_0> 0$, and constants $\tilde C^1$,  $\tilde C^2$ such that  for any $\psi \in \SC{P}^M_2$ and $\vep\leq \vep_0$
\begin{enumerate}[(i)]
	\item $\dst
	\EE\lf[\|\bar Z^\psi_\vep(s) - \bar Z^\psi_\vep (\vr_\vep(s))\|^m \Big | \SC{F}_{\vr_\vep(s)}\ri]  \leq   \tilde C^1\lf(\mcp^m(\vep)+ \|\bar Z^\psi_\vep (\vr_\vep(s))\|^{m \bar \alpha}(\Delta(\vep)/\vep)^m\ri),$
	\item $\dst
\EE\lf[\|\bar Z^\psi_\vep(s) - \bar Z^\psi_\vep (\vr_\vep(s))\|^m\ri] \leq   \tilde C^2\lf(\mcp^m(\vep)+\EE\|\bar Z^\psi_\vep (s)\|^{m\bar \alpha}(\Delta(\vep)/\vep)^m\ri),$
\end{enumerate}	
where $\mcp(\vep) = 	\mdpsc\Delta^{1/2}(\vep)/\vep $. Furthermore, if $m \leq 2$, then
\begin{enumerate}[(i)]
	\setcounter{enumi}{2}
\item
$\dst \int_0^T\EE\lf[\|\bar Z^\psi_\vep(s) - \bar Z^\psi_\vep (\vr_\vep(s))\|^m\ri]ds \leq   \tilde C^3(T)\lf(\lf(\f{\Delta(\vep)}{\vep}\ri)^{m/2}+\lf(\f{\Delta(\vep)}{\vep}\ri)^m\int_0^T\EE\|\bar Z^\psi_\vep (\vr_\vep(s))\|^{m\bar \alpha}ds\ri),$
	\item
$\dst \int_0^T\EE\lf[\|\bar Z^\psi_\vep(s) - \bar Z^\psi_\vep (\vr_\vep(s))\|^m\ri]ds \leq   \tilde C^4(T)\lf(\lf(\f{\Delta(\vep)}{\vep}\ri)^{m/2}+\lf(\f{\Delta(\vep)}{\vep}\ri)^m\int_0^T\EE\|\bar Z^\psi_\vep (s)\|^{m\bar \alpha}ds\ri).
$
\end{enumerate}	
Here $\tilde C^1, \tilde C^2, \tilde C^3(T), \tilde C^4(T)$ and $\vep_0$ depend only on $\cnst, L_b, \|\s\|_\infty, \nu, \bar \alpha, M, m.$
\end{lemma}	

\begin{proof}
Notice that
  \begin{align}
  	\non
  	\bar Z^\psi_\vep(s) - \bar Z^\psi_\vep (\vr_\vep(s)) = &\  \f{1}{\vep} b( \bar Z^\psi_\vep (\vr_\vep(s))(s - \vr_\vep(s) ) + \f{1}{\sqrt \vep} \s(\bar Z^\psi_\vep(\vr_\vep(s))) (W(s) - W(\vr_\vep(s)) \\ \label{eq-diff-0}
  	& \  + \f{\mdpsc}{\vep}  \s(	\bar Z^\psi_\vep(\vr_\vep(s))) \int_{\vr_\vep(s)}^s\psi(r) dr.
  \end{align}	
   Using the facts that (a) for any $m>0$, there exists a constant $\tilde C_m$ such that $\|x+y\|^m \leq \tilde C_m(\|x\|^m + \|y\|^m)$, (b) $\EE(\|W(h)\|^m)= O (h^{m/2})$, (c) $\|b^m(x)\| \leq  \cnst^m(1+\|x\|^{m\bar\alpha})$  and (d)  $\|s - \vr_\vep(s)| \leq \Delta$, we have for some constant $\tilde C^0$.
  \begin{align}
  	\non
  \EE\lf[\|\bar Z^\psi_\vep(s) - \bar Z^\psi_\vep (\vr_\vep(s))\|^m \Big | \SC{F}_{\vr_\vep(s)}\ri] 
         \leq &\  \tilde C^0\lf[\cnst^m(1+\|\bar Z^\psi_\vep (\vr_\vep(s))\|^{m\bar \alpha})\lf(\f{\Delta}{\vep}\ri)^m + \|\s\|_\infty^m \lf(\f{\Delta}{\vep}\ri)^{m/2}\ri. \\ 
    \label{diff-prelim}     
         & \ \lf. + \|\s\|_\infty^m \lf(\f{\mdpsc}{\vep}\ri)^m\EE\lf( \lf(\Delta \int_{\vr_\vep(s)}^s\|\psi(r)\|^2 dr\ri)^{m/2}\Big| \SC{F}_{\vr_\vep(s)}\ri)\ri],\\ \non
         \leq & \tilde C^0\Big(\cnst^m \lf(1+\|\bar Z^\psi_\vep (\vr_\vep(s))\|^{m\bar \alpha}\ri) \lf(\f{\Delta}{\vep}\ri)^m + \|\s\|_\infty^m \lf(\f{\Delta}{\vep}\ri)^{m/2} \\ \non
         & \ + \|\s\|_\infty^m M^{m/2}\lf(\f{\mdpsc \Delta^{1/2}}{\vep}\ri)^m \Big)\\ 
         \label{diff-est-1}
          \leq &\ \tilde C^1\lf(\mcp^m(\vep)+\|\bar Z^\psi_\vep (\vr_\vep(s))\|^{m\bar \alpha}(\Delta/\vep)^m\ri),
         \end{align}      
   where $\mcp(\vep) = \mdpsc\Delta^{1/2}/\vep.$ The last inequality follows because
   \begin{align*}
   	\lf(\Delta/\vep\ri)^{1/2} =\f{\sqrt \vep}{\mdpsc} \mcp(\vep) \leq \mcp(\vep),\quad \Delta/\vep \leq (\Delta/\vep)^{1/2}\mcp(\vep) \leq \mcp(\vep).
   	\end{align*} 
   	Recall that $\sqrt \vep/\mdpsc \rt 0$ and $\Delta/\vep \rt 0$ as $\vep \rt 0$, and we can assume $\mdpsc$ and $\Delta(\vep)$ are such that $\max\{\sqrt \vep/\mdpsc, \Delta/\vep\} \leq 1$. This proves $(i)$.

To prove (ii), notice that writing $b(\bar Z^\psi_\vep (\vr_\vep(s)) = b(\bar Z^\psi_\vep(s) +(b(\bar Z^\psi_\vep (\vr_\vep(s))- b(\bar Z^\psi_\vep(s))$ in \eqref{eq-diff-0}, it follows from the H\"older continuity of $b$, and the fact that H\"older exponent $\nu \leq 1$, that
\begin{align*}
\EE\lf[\|\bar Z^\psi_\vep(s) - \bar Z^\psi_\vep (\vr_\vep(s))\|^m\ri] \leq &\   \tilde C^2 \lf(\mcp^m(\vep)+(\Delta/\vep)^{m}\lf(\EE \|\bar Z^\psi_\vep (s)\|^{m\bar \alpha}+\EE \|\bar Z^\psi_\vep (s) -\bar Z^\psi_\vep (\vr_\vep(s)) \|^{m}\ri)  \ri).
\end{align*}
for some constant $ \tilde C^2$.
Choosing  $\vep_0$ sufficiently small so that for all $\vep \leq \vep_0$, $\tilde C^2 (\Delta/\vep)^m \leq 1/2$, we have $(ii)$.

For $(iii)$ and $(iv)$ first notice that by Holder's inequality (with $p = 2/m$) and \eqref{int-diff-est}
\begin{align*}
\int_0^T \lf( \int_{\vr_\vep(s)}^s\|\psi(r)\|^2 dr\ri)^{m/2}ds \leq &\ T^{1-m/2} \lf(\int_0^T  \int_{\vr_\vep(s)}^s\|\psi(r)\|^2 dr ds\ri)^{m/2}\\
\leq &\ T^{1-m/2} \lf(\Delta \int_0^T  \|\psi(r)\|^2  ds\ri)^{m/2}\\
\leq &\ M^{m/2}T^{1-m/2} \Delta^{m/2}.
\end{align*}	
The assertion $(iii)$ now follows after integrating both sides in \eqref{diff-prelim} and from the simple observation
$\max\{\mdpsc\Delta/\vep, (\Delta/\vep)^{1/2}\} \leq (\Delta/\vep)^{1/2}$. (iv) now follows using the same splitting used above to obtain (ii).

\end{proof}	

\np
The following is the corresponding result for the original process $Z^\vep$.
\begin{lemma}\label{lem-diff-est-Z}
	Let $ Z^\vep$ be as in \eqref{SDE-Eu-1}, and assume that Condition \ref{cond_SDE_coef-ex} holds, and that $\s$ is bounded. Let $\Delta(\vep)$ be such that $\Delta(\vep)/\vep \rt 0$ as $\vep \rt 0$. Then there exist constants $\hat C^1$,  $\hat C^2$ and $\vep_0$ such that for all $\vep \leq \vep_0$, 
	\begin{enumerate}[(i)]
		\item $\dst
		\EE\lf[\| Z^\vep(s) -  Z^\vep (\vr_\vep(s))\|^m \Big | \SC{F}_{\vr_\vep(s)}\ri]  \leq    \hat C^1(\Delta/\vep)^{m/2}\lf(1+ \| Z^\vep (\vr_\vep(s))\|^{m \bar \alpha}\ri),$
		\item $\dst
		\EE\lf[\| Z^\vep(s) -  Z^\vep (\vr_\vep(s))\|^m  \ri]  \leq    \hat C^2(\Delta/\vep)^{m/2}\lf(1+\EE \| Z^\vep (s))\|^{m \bar \alpha}\ri),$
	\end{enumerate}	
	Here $\hat C^1, \hat C^2$ and $\vep_0$ depend only on $\cnst, L_b, \|\s\|_\infty, \nu, \bar \alpha, M, m.$
\end{lemma}	
Note that, in particular, for come constants $\hat C^3$ and $\hat C^4$,
\begin{align}\label{est-mmt-comp}
	\EE(\| Z^\vep(s))\|^{m \bar \alpha}) \leq&\ \hat C^3(1+E\| Z^\vep (\vr_\vep(s)))\|^{m \bar \alpha}), \quad
\EE(\| Z^\vep (\vr_\vep(s))\|^{m \bar \alpha}) \leq\ \hat C^4(1+E\| Z^\vep (s))\|^{m \bar \alpha}).
\end{align}	

For the proof of central limit theorem, we also need a similar result corresponding to a coarser partition than $\{t_k\}$. A similar estimate in the controlled setting is also needed for the MDP result, and this is discussed in Remark \ref{int-disc-est-1}.
\begin{lemma} \label{lem-diff-est-Z-2}
Assume the hypotheses of Lemma \ref{lem-diff-est-Z}. Let $\{\tilde t_k\}$ be a partition of $[0,t]$ such that $\tilde t_k  - \tilde t_{k-1} = \tilde \Delta \equiv  \tilde \Delta(\vep)$, and $\eta_\vep$ the corresponding step function, that is, $\eta_\vep(s) = \tilde t_{k},$ if $\tilde t_k \leq s <\tilde t_{k+1}$. Then there exist costants $\hat C^3,$  $\hat C^4$ and  $\vep_0$ such that for all $\vep \leq \vep_0$, 
\begin{enumerate}[(i)]
	\item $\dst
	\int_0^T\EE\lf[\| Z^\vep(s) -  Z^\vep (\eta_\vep(s))\|^m \ri]ds  \leq    \hat C^5(T)(\tilde\Delta/\vep)^{m/2}\lf(1+ \int_0^T\EE\| Z^\vep (\vr_\vep(s))\|^{m \bar \alpha}ds\ri),$
	\item $\dst
	\int_0^T\EE\lf[\| Z^\vep(s) -  Z^\vep (\eta_\vep(s))\|^m  \ri]ds  \leq    \hat C^6(T)(\tilde\Delta/\vep)^{m/2}\lf(1+\int_0^T\EE \| Z^\vep (s))\|^{m \bar \alpha}ds\ri),$
\end{enumerate}	
The constants $\hat C^5(T), \hat C^6(T)$ and $\vep_0$ depend only on $\cnst, L_b, \|\s\|_\infty, \nu, \bar \alpha, m.$
\end{lemma}

\begin{proof}
Let $m \geq 1$. Then for some constant $\hat K_0$ 
\begin{align*}
	\| Z^\vep(t) -  Z^\vep (\eta_\vep(t))\|^m = &\ \hat K_0\lf[ \f{1}{\vep^m} \lf(\int_{\eta_\vep(t)}^t \cnst(1+ \| Z^\vep (\vr_\vep(s))\|^{\bar\alpha} )ds\ri)^m + \f{1}{ \vep^{m/2}} \lf\|\int_{\eta_\vep(t)}^t \s( Z^\vep(\vr_\vep(s))) dW(s)\ri\|^m\ri]. 
\end{align*}
Taking expectation and applying H\"older's inequality
\begin{align*}
	\EE\| Z^\vep(t) -  Z^\vep (\eta_\vep(t))\|^m \leq &  \hat K_0\lf[ \f{ \cnst^m\tilde\Delta^{m-1}}{\vep^m} \int_{\eta_\vep(t)}^t \EE(1+ \| Z^\vep (\vr_\vep(s))\|^{\bar\alpha} )^m ds + \|\s\|^m_\infty\lf(\f{ \tilde \Delta}{ \vep}\ri)^{m/2}\ri].
\end{align*}	
The first assertion now follows by integrating both sides over $[0,T]$ and using  \eqref{int-diff-est}, and the second by \eqref{est-mmt-comp}.

\end{proof}

%

\setcounter{equation}{0}
\renewcommand {\theequation}{\arabic{section}.\arabic{equation}}
\section{Tightness results} \label{sec:tight}


%
 	  	\begin{lemma}\label{ctrlproc-bd-2}
 	  		Suppose that  $\bar Z^\psi_\vep$ satisfies \eqref{ctrl-EuSDE1} and that $\Delta(\vep)/\vep \rt 0$. Assume that Condition \ref{cond_SDE_inv} (with $\alpha>0$) and Condition \ref{cond_SDE_coef-ex} hold. Then for all $M>0$, there exists an $\vep_0>0$, such that
 	  		$$\sup_{\vep \in (0,\vep_0]} \sup_{\psi \in \SC{P}_2^M} \EE\lf[\int_0^T \|\bar Z^\psi_\vep(t)\|^{2\alpha} dt \ri] <\infty. $$
 	  	\end{lemma}	
 	  	
 	  	 	\begin{proof} The main idea is to use It\^o's lemma to the function $x \rt \|x\|^{(1+\alpha)/2}$ and then obtain estimates on different expectations. However, if $\alpha < 2$, some technical issues arise (because of singularity of the map $x \rt \|x\|^{\alpha-2}$ at origin) for obtaining bounds on certain terms. One way to avoid them is to use a  $C^\infty\lf([0,\infty), [0,\infty)\ri)$ - function $\vart$ defined by 	  	 		\begin{align*}
 	  	 			\vart(x)=
 	  	 			\begin{cases}
 	  	 				x^{(1+\alpha)/2},& \quad {x>1}\\
 	  	 				0,& \quad{ 0<x<0.9}.
 	  	 			\end{cases}      
 	  	 		\end{align*}
 	  	 	  For notational convenience, we will drop the superscript $\psi$ and use $\bar Z_\vep$ instead of $\bar Z^\psi_\vep.$
 	  	 	
 	  	 	By It\^o's lemma,
 	  	 	\begin{align}\non
 	  	 		\vart(\|\bar Z_\vep(t)\|^2)  = & \vart(\|x_0\|^2) + \mar_\vep(t)+ \int_0^t \vart'(\|\bar Z_\vep(s)\|^2) d\|\bar Z_\vep(s)\|^2 +\f{1}{2}\int_0^t  \vart''(\|\bar Z_\vep(s)\|^2)d[\|\bar Z_\vep\|^2]_s\\ \non
 	  	 		=& \vart(\|x_0\|^2)+ \mar_\vep(t) + \f{2}{\vep} \int_0^t   \vart'(\|\bar Z_\vep(s)\|^2) \<\bar Z_\vep(s), b(\bar Z_\vep(\vr_\vep(s)))\>ds \\ \non
 	  	 		&+ \f{1}{\vep} \int_0^t  \vart'(\|\bar Z_\vep(s)\|^2) \|\s(\bar Z_\vep(\vr_\vep(s)))\|^2\ ds\\ \non
 	  	 		&+  \f{2\mdpsc}{\vep}  \int_0^t \vart'(\|\bar Z_\vep(s)\|^2) \<\bar Z_\vep(s), \s(\bar Z_\vep(\vr_\vep(s))) \psi(s)\> ds\\ \label{Ito-l-MC}
 	  	 		&+\f{1}{\vep}\int_0^t \vart''(\|\bar Z_\vep(s)\|^2)\|\bar Z_\vep(s)\s(\bar Z_\vep(\vr_\vep(s)))\|^2 ds,	\\ \non
 	  	 		\leq &\ \vart(\|x_0\|^2)+ \mar_\vep(t) + \f{2}{\vep} \int_0^t   \vart'(\|\bar Z_\vep(s)\|^2) \<\bar Z_\vep(s), b(\bar Z_\vep(s))\>ds \\ \non
 	  	 		& + \f{2}{\vep} \int_0^t   \vart'(\|\bar Z_\vep(s)\|^2) \<\bar Z_\vep(s), b(\bar Z_\vep(\vr_\vep(s)) - b(\bar Z_\vep(s))\>ds\\ \non
 	  	 		&+ \f{\|\s\|_\infty^2}{\vep} \int_0^t  \vart'(\|\bar Z_\vep(s)\|^2) ds
 	  	 		+  \f{2\mdpsc \|\s\|_\infty}{\vep}  \int_0^t \vart'(\|\bar Z_\vep(s)\|^2)\|\bar Z_\vep(s)\| \|\psi(s)\| ds\\ \label{Ito-l-MC-2}
 	  	 		&+\f{\|\s\|_\infty^2}{\vep}\int_0^t \vart''(\|\bar Z_\vep(s)\|^2)\|\bar Z_\vep(s)\|^2 ds,		 		
 	  	 	\end{align}	
 where
 $$\mar_\vep(t) = \f{2}{\sqrt \vep} \int_0^t \vart'(\|\bar Z_\vep(s)\|^2)  \bar Z_\vep(s)^T\s(\bar Z_\vep(\vr_\vep(s)))dW(s)$$
 is a martingale. We now estimate some of these terms individually.
 Let $\bar B = B\vee 1$ ($B$ was introduced in Condition \ref{cond_SDE_inv}-(i)) and
 \begin{align*}
 A_1(t) = &\  \int_0^t \vart'(\|\bar Z_\vep(s)\|^2) | \<\bar Z_\vep(s), b(\bar Z_\vep(\vr_\vep(s)) - b(\bar Z_\vep(s))\> |1_{\{\|\bar Z_\vep(s)\| >\bar B\}}ds;\\
 A_2(t) = &\  \int_0^t \vart'(\|\bar Z_\vep(s)\|^2) | \<\bar Z_\vep(s), b(\bar Z_\vep(\vr_\vep(s)) - b(\bar Z_\vep(s))\> |1_{\{\|\bar Z_\vep(s)\| \leq \bar B\}}ds;
 \end{align*}
 First, observe that using H\"older  continuity of $b$ and  Lemma \ref{lem-diff-est}-(iv) we have for some constant $\hat C^0_B(T)$,
 \begin{align*}
 	\EE(A_2(t)) \leq &\ \|\vart\|_{\infty, \bar B^2} \|\bar B\|L_b \int_0^t\EE\lf(\|\bar Z_\vep(\vr_\vep(s)) -\bar Z_\vep(s)\|^\nu1_{\{\|\bar Z_\vep(s)\| \leq \bar B\}}\ri)ds\\
 	\leq & \|\vart\|_{\infty, \bar B^2} \|\bar B\|L_b\tilde C^4 (\Delta/\vep)^{\nu/2} \lf(1+ (\Delta/\vep)^{\nu/2}\int_0^t\EE\|\bar Z^\psi_\vep (s)\|^{\nu\bar \alpha}ds\ri).
 \end{align*}
 Next, by (a)  H\"older continuity of $b$, (b) the fact that $|xy| \leq \f{1}{2}\lf(\theta|x|^2 + |y|^2/\theta\ri)$ for any $\theta>0$, and (c) Lemma \ref{lem-diff-est} - (iv),
 \begin{align*}
  \EE(A_1(t)) \leq & L_b(1+\alpha)\EE\int_0^t\|\bar Z_\vep(s)\|^{\alpha}\|\bar Z_\vep(\vr_\vep(s)) -\bar Z_\vep(s)\|^\nu1_{\{\|\bar Z_\vep(s)\| >\bar B\}}ds \\
   \leq & \  \f{L_b(1+\alpha)}{2}\int_0^t\lf(\EE\lf(\|\bar Z_\vep(s)\|^{2\alpha}\ri)(\Delta/\vep)^{\nu}+ \EE(\|\bar Z_\vep(\vr_\vep(s)) -\bar Z_\vep(s)\|^{2\nu})(\vep/\Delta)^{\nu}\ri)ds\\
    \leq &\ \f{L_b(1+\alpha)}{2}  \lf( (\Delta/\vep)^{\nu}\int_0^t\EE\|\bar Z_\vep(s)\|^{2\alpha}ds \ri.\\
    &\hs{1cm} \lf.+\tilde C^4(T) (\Delta/\vep)^{\nu}\lf(1+(\Delta/\vep)^{\nu} \int_0^t\EE\|\bar Z_\vep (s)\|^{2\nu\bar \alpha}ds\ri) (\vep/\Delta)^{\nu}\ri) \\
    \leq& \ \hat C^{1,B}(T)  \lf(1+ (\Delta/\vep)^\nu\int_0^t  \EE\|\bar Z_\vep (s)\|^{2\alpha}\ ds \ri).
 \end{align*}
 for some constant $\hat C^{1,B}(T) $ (for the last inequality, we used $\nu\bar \alpha \leq \alpha$). Also,
 \begin{align*}
  \int_0^t \vart'(\|\bar Z_\vep(s)\|^2)\|\bar Z_\vep(s)\| \|\psi(s)\|1_{\{\|\bar Z_\vep(s)\| >\bar B\}} ds            \leq & \f{(1+\alpha)}{2} \int_0^t \|\bar Z_\vep(s)\|^{\alpha} \|\psi(s)\|1_{\{\|\bar Z_\vep(s)\| > \bar B\}} ds\\
                  \leq & \ \f{(1+\alpha)}{4} \int_0^t \lf(\|\bar Z_\vep(s)\|^{2\alpha} + \|\psi(s)\|^2\ri)1_{\{\|\bar Z_\vep(s)\| > \bar B\}} ds\\
                  \leq & \ \f{(1+\alpha)}{4}\int_0^t \|\bar Z_\vep(s)\|^{2\alpha}1_{\{\|\bar Z_\vep(s)\| > \bar B\}}ds + \f{M(1+\alpha)}{4}.
 \end{align*}

 	  	 	Now splitting each term according to $\{\|\bar Z_\vep(s)\| \leq \bar B\}$ and $\{\|\bar Z_\vep(s)\| > \bar B\}$ (and noting that $\nu\bar \alpha \leq \alpha$) it follows that there exists a  constant $\hat C^{2,B}(T)$ such that
	 	\begin{align}
	 		\non
	 	\EE \lf(\vart(\|\bar Z_\vep(t)\|^2)\ri) 	\leq   & \|x_0\|^{1+\alpha} + \hat C^{2,B}(T)\lf(1+ (\Delta/\vep)^\nu\int_0^t  \EE\|\bar Z_\vep (s)\|^{2\alpha}\ ds \ri)/\vep \\ \non
	 			 		&\ - \f{\g(1+\alpha)}{\vep}\int_0^t \EE\lf(\|\bar Z_\vep(s)\|^{2\alpha}1_{\{\|\bar Z_\vep(s)\| > \bar B\}}\ri) ds\\
	 		\non
	 		& + \f{\alpha(1+\alpha)\|\s\|_\infty^2}{2} \int_0^t \EE\lf(\|\bar Z_\vep(s)\|^{\alpha-1}1_{\{\|\bar Z_\vep(s)\| > \bar B\}}\ri) ds \\
	 		\label{ctrl-proc-mmt-bd-int}
	 		&+  \f{ 2\mdpsc \|\s\|_\infty}{\vep} \lf(\f{(1+\alpha)}{4}\int_0^t \EE\|\bar Z_\vep(s)\|^{2\alpha}1_{\{\|\bar Z_\vep(s)\| >\bar B\}}ds + \f{M(1+\alpha)}{4}\ri).
	 	\end{align}

 	Putting things together,   	it follows that
 	   		\begin{align*}
 	   			\int_0^t\EE\lf( \|\bar Z_\vep(s)\|^{2\alpha} 1_{\{\|\bar Z_\vep(s)\| > \bar B\}}\ri)  ds \leq &\ \f{\vep}{\g(1+\alpha)}\|x_0\|^{1+\alpha}+\hat C^{2,B}(T)/\g(1+\alpha) +\f{\mdpsc M(1+\alpha)}{2\g(1+\alpha)}\\
 	   			& \ +  \lf(\hat C^{2,B}(T) (\Delta/\vep)^\nu+ \f{\mdpsc \|\s\|_\infty}{2\g}\ri)\int_0^t \EE\|\bar Z_\vep(s)\|^{2\alpha}ds\\
 	   			&\ + \f{\alpha\|\s\|_\infty^2}{2\g\bar B} \int_0^t \EE\|\bar Z_\vep(s)\|^{\alpha} ds. 
 	   			%
 	   		\end{align*}		
 	Notice that for any $\theta >0$, $ \EE\|\bar Z_\vep(s)\|^{\alpha} \leq \theta  \EE\|\bar Z_\vep(s)\|^{2\alpha}+\theta ^{-1}.$   		
Now  choose $\vep_0>0$ such that $\hat C^{2,B}(T) (\Delta/\vep)^\nu+ \f{\mdpsc \|\s\|_\infty}{2\g}\leq 1/4$ for $\vep\leq \vep_0$, and  choose $\theta >0$ such that $\theta  \alpha\|\s\|_\infty^2/2\g\bar B \leq 1/4$. Then for all $\vep \leq \vep_0$,
 \begin{align*}
 	\f{1}{2}\int_0^t\EE\lf( \|\bar Z_\vep(s)\|^{2\alpha} \ri)  ds \leq & \ \bar B^{2\alpha}T+\f{\vep\|x_0\|^{1+\alpha}}{\g(1+\alpha)}+\f{\hat C^{2,B}(T)}{\g(1+\alpha)} +\f{\mdpsc M(1+\alpha)}{2\g(1+\alpha)} + \f{\alpha\|\s\|_\infty^2 T}{2\theta \g\bar B}
 \end{align*}	
 which proves the assertion.	  	 	
\end{proof}

We now state the similar result for the original process $Z^\vep$.
\begin{lemma}\label{Z-proc-mmt-bd}
	Suppose that  $ Z^\vep$ satisfies \eqref{SDE-Eu-1} and that $\Delta(\vep)/\vep \rt 0$. Assume that Condition \ref{cond_SDE_inv} and Condition \ref{cond_SDE_coef-ex} hold. Then for all $q>0$, there exists a constant $\vep_0$ such that
	$$\sup_{0<\vep \leq \vep_0} \EE\lf[\int_0^T \| Z^\vep(t)\|^{q} dt \ri] <\infty. $$
\end{lemma}

\begin{proof}
	Let $p\geq2$. Then
	\begin{align}\non
		\| Z^\vep(t)\|^p = & \|x_0\|^p + \mart^\vep(t)+\f{p}{2} \int_0^t \| Z^\vep(s)\|^{p-2} d\| Z^\vep(s)\|^2 +\f{p(p-2)}{8}\int_0^t  \|Z^\vep(s)\|^{p-4}d[\| Z^\vep\|^2]_s\\ \non
		=& \|x_0\|^p+ \mar_\vep(t) + \f{p}{\vep} \int_0^t   \| Z^\vep(s)\|^{p-2} \< Z^\vep(s), b(\bar Z^\vep(\vr_\vep(s)))\>ds \\ \non
		&+ \f{p}{2\vep} \int_0^t  \|Z^\vep(s)\|^{p-2} \|\s(Z^\vep(\vr_\vep(s)))\|^2\ ds\\ \non
		&+\f{p(p-2)}{2\vep}\int_0^t \| Z^\vep(s)\|^{p-4}\|Z^\vep(s)\s( Z^\vep(\vr_\vep(s)))\|^2 ds,	\\ \non
	\end{align}	
where
$$\mart^\vep(t) = \f{p}{\sqrt \vep} \int_0^t \|\bar Z_\vep(s)\|^{p-2} \bar Z_\vep(s)^T\s(\bar Z_\vep(\vr_\vep(s)))dW(s)$$
is a martingale.
	Splitting the third term according as $\|Z^\vep(s)\| >B$ or not, we have for some constant $\tilde \const^{1,B}(T)$
	\begin{align}\non
		\| Z^\vep(t)\|^p \leq &\ \|x_0\|^{p} + \mart^\vep(t)+\f{p}{\vep}A^\vep(t)+\tilde \const^{1,B}(T)/\vep-\f{p}{\vep} \int_0^t   \| Z^\vep(s)\|^{p+\alpha-1} 1_{\{\|Z^\vep(s)\| >B\}}ds \\ \label{mmt-est-0}
		& \ + \f{\|\s\|_\infty^2(2p+p(p-2))}{4\vep} \int_0^t\| Z^\vep(s)\|^{p-2}ds,
	\end{align}	
	where 
	\begin{align*}
		A^\vep(t) = &\  \int_0^t \|\bar Z_\vep(s)\|^{p-2}  \<\bar Z_\vep(s), b(\bar Z_\vep(\vr_\vep(s)) - b(\bar Z_\vep(s))\> ds.
	\end{align*}		
	This term has to be handled a bit differently than the estimate of the corresponding term (c.f $ A_1(t)$) in the previous proof.
	Notice that by Lemma \ref{lem-diff-est}
	\begin{align*}
		\EE|A^\vep(t)| \leq &\ 2^{p-2}L_b\lf(\EE\int_0^t \|Z^\vep(\vr_\vep(s))\|^{p-1}  \| Z^\vep(\vr_\vep(s)) - Z^\vep(s)\|^\nu ds + \EE\int_0^t  \| Z^\vep(\vr_\vep(s)) - Z^\vep(s)\|^{p+\nu-1} ds  \ri)\\
		& \leq \tilde\const^{2}(T) \Bigg(\int_0^t \EE\lf(\|Z^\vep(\vr_\vep(s))\|^{p-1} \EE\lf( \| Z^\vep(\vr_\vep(s)) - Z^\vep(s)\|^\nu \Big |\SC{F}_{\vr_\vep(s)}\ri)\ri) ds\\
		& \hs*{1cm} + \lf(\f{\Delta}{\vep}\ri)^{(p+\nu-1)/2} \lf(1+\int_0^t \EE  \|Z^\vep(s)\|^{(p+\nu-1)\bar\alpha} ds\ri)\Bigg)\\
		& \leq \tilde\const^{3}(T) \Bigg(\lf(\f{\Delta}{\vep}\ri)^{\nu/2}\int_0^t\lf(1+\EE\|Z^\vep(\vr_\vep(s))\|^{p-1+\nu\bar\alpha}\ri)ds\\
		& \hs*{1cm}   + \lf(\f{\Delta}{\vep}\ri)^{(p+\nu-1)/2} \lf(1+\int_0^t \EE  \|Z^\vep(s)\|^{(p+\nu-1)\bar\alpha} ds\ri)\Bigg)\\	  
		& \leq \tilde\const^{4}(T) \lf(\f{\Delta}{\vep}\ri)^{\nu/2}\lf(1+ \int_0^t \EE\|Z^\vep(s)\|^{p+\alpha-1} ds\ri),
	\end{align*}
	where the $\tilde\const^i$ are appropriate constants.
	Here the last inequality uses \eqref{est-mmt-comp} and the fact that $\max\{p-1+\nu\bar \alpha, (p-1+\nu)\bar\alpha\} \leq p+\alpha-1.$
	Now following similar steps as in the proof of the previous theorem, it follows after choosing  $\Delta(\vep)/\vep$ sufficiently small and rearranging terms in \eqref{mmt-est-0}, that there exists an $\vep_0>0$ such that
	$$\sup_{0<\vep \leq \vep_0}\EE\lf[\int_0^T \| Z^\vep(t)\|^{p+\alpha-1} dt \ri] <\infty. $$

\end{proof}	

\begin{corollary}\label{mmt-bd-ctrl}
Under the assumptions of Lemma \ref{ctrlproc-bd-2},
\begin{align*}
\sup_{0<\vep<\vep_0} \sup_{\psi \in \SC{P}_2^M} \vep \EE\lf[\sup_{r \leq t} \|\bar Z_\vep ^\psi (r)\|^{1+\alpha}\ri] < \infty.
\end{align*}
\end{corollary}

\begin{proof}

It follows from \eqref{Ito-l-MC-2}  that (again after denoting $\bar Z_\vep ^\psi$ by $\bar Z_\vep$) for some constant $\hat C^{3,B}(T)$

\begin{align*}
\vep\sup_{r \leq t} \|\bar Z_\vep ^\psi (r)\|^{1+\alpha} 1_{\{\|\bar Z_\vep(r)\| > \bar B\}} \leq & \ \vep\|\vart(\|x_0\|^2)\|+  \vep\sup_{r\leq t}\|\mar_\vep(r)\| + \hat C^{3,B}(T)+ 2(A_1(t) + A_2(t))  \\ 
 	  	 		&\ + (1+\alpha)\cnst \int_0^t  \|\bar Z_\vep(s)\|^{\alpha} (1+ \|\bar Z_\vep(s)\|^{\bar \alpha})1_{\{\|\bar Z_\vep(r)\| > \bar B\}} ds\\
				& \ + \f{\alpha(1+\alpha)\|\s\|_\infty^2}{2}  \int_0^t \|\bar Z_\vep(s)\|^{\alpha-1}1_{\{\|\bar Z_\vep(s)\| > \bar B\}} ds \\
&\ +  \mdpsc (1+\alpha) \|\s\|_\infty \int_0^t \|\bar Z_\vep(s)\|^{\alpha} \psi(s)1_{\{\|\bar Z_\vep(s)\| > \bar B\}} ds.
\end{align*}

 Lemma \ref{ctrlproc-bd-2} implies  that the  expectation of $A_1(t), A_2(t)$ and the last three terms of the above display are bounded above by a constant $\hat C^{4}_{T}$.
Finally, to bound $ \vep\EE\lf[\sup_{r\leq t}|\bar M_\vep(t)|\ri]$ notice that by Burkholder-Davis-Gundy inequality
\begin{align*}
 \vep \EE\lf[\sup_{r\leq t} |\bar M_\vep(t)|\ri] \leq & 2 \sqrt \vep \EE \lf(\int_0^t  \vart'(\|\bar Z_\vep(s)\|^2)^2 \|\bar Z_\vep(s)^T\s(\bar Z_\vep(\vr_\vep(s)))\|^2 ds \ri)^{1/2}\\
 \leq &\  \sqrt \vep \|\s\|_\infty^2 \EE \Bigg(\int_0^t  \vart'(\|\bar Z_\vep(s)\|^2)^2 \|\bar Z_\vep(s)^T\|^2  1_{\{ \|\bar Z_\vep(s)\| \leq  \bar B\}}ds \\
 &  + \f{1+\alpha}{2}\int_0^t \|\bar Z_\vep(s)\|^{2(\alpha-1)} \|\bar Z_\vep(s)^T\|^2 1_{\{\|\bar Z_\vep(s)\| > \bar B\}} ds \Bigg)^{1/2}\\
  \leq &  \sqrt \vep \EE \Bigg(\|\vart'\|^2_{\bar B^2} \bar B^2 T  +  \int_0^t \|\bar Z_\vep(s)\|^{2\alpha} 1_{\{\|\bar X_\vep(s)\| > \bar B\}} ds \Bigg)^{1/2}\\
  \leq & \sqrt \vep \hat C^{5,B}(T),
\end{align*}
for some constant $\hat C^{5}(T)>0$. Here the last inequality made use of Lemma \ref{ctrlproc-bd-2}.

\end{proof}

We now state the similar result for the original process, which is a corollary to Lemma \ref{Z-proc-mmt-bd}, and whose proof uses similar (in fact much simpler) steps used in the proof of Corollary \ref{mmt-bd-ctrl}.
\begin{corollary}\label{mmt-bd-Z}
	Under the assumptions of Lemma \ref{Z-proc-mmt-bd}, for any $q>0$, there exists a constant $\vep_0$ such that
	\begin{align*}
		\sup_{0<\vep<\vep_0}  \vep \EE\lf[\sup_{r \leq t} \| Z^\vep  (r)\|^{q}\ri] < \infty.
	\end{align*}
\end{corollary}

\begin{remark} \label{int-disc-est-0}
{\rm
If $\alpha \leq 1$, then it follows from Lemma \ref{lem-diff-est}-(iv) that 
\begin{align*}
\sup_{\vep \in (0,\vep_0]} \sup_{\psi \in \SC{P}_2^M} \EE\lf[\int_0^T \|\bar Z^\psi_\vep(\vr_\vep(t))\|^{2\alpha} dt \ri] <\infty.
\end{align*}
Recall that $\Delta(\vep)/\vep \rt 0$.
But if $\alpha > 1$, then it is not enough to have $\Delta(\vep)/\vep \rt 0$. However, if $\Delta(\vep)$ is such that $\beta(\vep)= \mdpsc\Delta^{1/2}(\vep)/\vep \rt 0$, then the above display still holds.
}
\end{remark}	

Lastly, as mentioned before, for technical reasons, we also need to need to consider a partition $\{\tilde t_k\}$ which is coarser than $\{t_k\}$ and have bounds for integral moments of  $Z^\vep\circ\eta_\vep$ (for the original process) and $\bar Z^\psi_\vep\circ\eta_\vep$ (for the controlled process), where  $\eta_\vep(s)$  is the step function corresponding to a partition $\{\tilde t_k\}$. Specifically, let $\{\tilde t_k\}$ be a partition of $[0,T]$ such that $\tilde \Delta \equiv \tilde \Delta(\vep)= \tilde t_k -\tilde t_{k-1} \leq \vep$, and define $ \eta_\vep(t) = \tilde t_{k}$ for $\tilde t_k\leq t < \tilde t_{k+1}$.

\begin{remark}  \label{int-disc-est-Z}{\rm
From Lemma \ref{lem-diff-est-Z-2} and Lemma \ref{Z-proc-mmt-bd}, it is immediately clear that for any $q>0$, there exist a constant $\bar K^1(T)$ and $\vep_0$ such that
\begin{align*}
\sup_{0<\vep \leq \vep_0} \sup_{\psi \in \SC{P}_2^M} \EE\lf[\int_0^T \| Z^\vep(\eta_\vep(t))\|^{q} dt \ri] \leq \bar K^1(T).
\end{align*}	
}
\end{remark}		

\begin{remark} \label{int-disc-est-1}
{\rm 
If $0 \leq \alpha \leq 1$, then it is not difficult to see that  similar techniques give the following result for the controlled process $\bar Z^\psi_\vep$. Specifically, when the partition $\{\tilde t_k\}$ is chosen with $\tilde \Delta \leq \vep$, then
	$$\sup_{0<\vep \leq \vep_0} \sup_{\psi \in \SC{P}_2^M} \EE\lf[\int_0^T \|\bar Z^\psi_\vep(\eta_\vep(t))\|^{2\alpha} dt \ri] < \bar K^2(T),$$
	for some constant $\vep_0$ and $\bar K^1(T)$.
%
However for  $\alpha >1$, it is not difficult to see that this technique does not work. 
 The problem lies with the fact that we only have $L^2$-boundedness of the controls $\psi$, and the fact that requiring $\tilde \Delta(\vep) \leq \vep$ is not enough to guarantee boundedness of $\mdpsc\tilde \Delta^{1/2}(\vep)/\vep$. However, if $1<\alpha\leq 2$, we still have a similar result on the $\alpha$-th order integral moment of $\bar Z^\psi_\vep \circ \eta_\vep$. That is, for $1<\alpha\leq 2$
 	$$\sup_{0<\vep \leq \vep_0} \sup_{\psi \in \SC{P}_2^M} \EE\lf[\int_0^T \|\bar Z^\psi_\vep(\eta_\vep(t))\|^{\alpha} dt \ri] < \bar K^2(T).$$
For other $\alpha$, this technique doesn't work and the following lemma, which makes fresh use of  It\^o's lemma, gives the same result. 

}
\end{remark}

\begin{corollary}\label{dis-mmt-bd-ctrl}
	Let $\{\tilde t_k\}$ be a partition of $[0,T]$ such that $\tilde \Delta = \tilde t_k -\tilde t_{k-1} \leq \vep $. Then, under the assumptions in Lemma \ref{ctrlproc-bd-2}, for all $M>0$, there exists an $\vep_0>0$ such that
	$$\sup_{\vep \in (0,1)} \sup_{\psi \in \SC{P}_2^M} \EE\lf[\int_0^T \|\bar Z^\psi_\vep(\eta_\vep(t))\|^{\alpha} dt \ri] <\infty, $$
where $ \eta_\vep$ was defined before Remark \ref{int-disc-est-1}.	
\end{corollary}	

\begin{proof}

%
%

As before,  we denote $\bar Z^\psi_\vep$ by $\bar Z_\vep$ for notational simplicity. Next we write
	\begin{align}\non
	 \int_0^t  \|\bar Z_\vep(\eta_\vep(s))\|^{\alpha} ds 
			& \ = \int_0^t  \lf( \|\bar Z_\vep(\eta_\vep(s))\|^{\alpha} 1_{\{\|\bar Z_\vep(\eta_\vep(s))\| \leq 1\}} + \|\bar Z_\vep(\eta_\vep(s))\|^{\alpha} 1_{\{\|\bar Z_\vep(\eta_\vep(s))\| \geq 1\}}\ri) ds\\ \non
				& \ \leq  \int_0^t  \lf(1+  \bar\vartheta\lf(\|\bar Z_\vep(\eta_\vep(s))\|^2\ri)\ri) ds\\ \non
		& = \  \int_0^t  \lf(1+  \bar\vartheta\lf(\|\bar Z_\vep(\eta_\vep(s))\|^2\ri) -  \bar\vartheta\lf(\|\bar Z_\vep(s)\|^2\ri) +  \bar \vartheta\lf(\|\bar Z_\vep(s)\|^2\ri) \ri) ds,\\ \non 
		& \leq  \  \int_0^t  \Big(1+ \|\bar \vartheta\|_{\infty,1}+ \bar \vartheta\lf(\|\bar Z_\vep(\eta_\vep(s))\|^2\ri) - \bar \vartheta\lf(\|\bar Z_\vep(s)\|^2\ri)\\ \label{split-1}
		&  \hs*{.5in} + \|\bar Z_\vep(s)\|^{\alpha} 1_{\{\|\bar Z_\vep(\eta_\vep(s))\| > 1\}} \Big) ds,
	\end{align}	
	where  $\bar\vartheta$ is a $C^\infty\lf([0,\infty), [0,\infty)\ri)$ function such
	\begin{align*}
		\bar \vart(x)=
		\begin{cases}
			x^{\alpha/2},& \quad {x>1}\\
			0,& \quad{ 0<x<0.9}, 
		\end{cases}      
	\end{align*}
	 and $\|\bar \vartheta\|_{\infty,r}$ denotes the maximum of $\vartheta$ on $[0,r]$.


From It\^o's lemma \eqref{Ito-l-MC} (with $\bar\vart$ in place of $\vart$), after splitting each term according to $\{\|\bar Z_\vep(s)\| \leq \bar B\}$ and $\{\|\bar Z_\vep(s)\| > \bar B\}$ (where $\bar B = B\vee 1$ with $B$ as in Condition \ref{cond_SDE_inv}-(iii)), we deduce that for some constant $\hat  C^{6,B}(T)$, 
\begin{align*}	
\bar \vart(\|\bar Z_\vep(\eta_{\vep}(s))\|^2)- \bar \vart(\|\bar Z_\vep(s)\|^2)	\leq & \ \hat  C^{6,B}(T)/\vep +   \mar_\vep(\eta_\vep(s)) - \mar_\vep(s)\\
&\ + \f{\alpha L_b}{\vep}\int^{s}_{\eta_\vep(s)}\|\bar Z_\vep(s)\|^{\alpha-1}\|\bar Z_\vep(\vr_\vep(s)) -\bar Z_\vep(s)\|^{\nu}1_{\{\|\bar Z_\vep(r)\| > \bar B\}}ds  \\
& \ +\f{\|\s\|_\infty^2}{2\vep} \int^{s}_{\eta_\vep(s)}  \|\bar Z_\vep(r)\|^{\alpha-2}1_{\{\|\bar Z_\vep(r)\| > \bar B\}}\ dr \\
	& +  \f{\alpha\|\s\|_\infty \mdpsc}{\vep}  \int^{s}_{\eta_\vep(s)} \|\bar Z_\vep(r)\|^{\alpha-1}\| \psi(r)\| 1_{\{\|\bar Z_\vep(r)\| > \bar B\}} dr\\
	&+\f{\alpha(\alpha-2)\|\s\|^2_\infty}{4\vep}\int^{s}_{\eta_\vep(s)} \|\bar Z_\vep(r)\|^{\alpha -2}1_{\{\|\bar Z_\vep(r)\| > \bar B\}} dr,\\ 	
\end{align*}
where $\mar_\vep$ is as in the proof of Lemma \ref{ctrlproc-bd-2} with $\vart$ replaced by $\bar \vart$.	
Since $\tilde\Delta \leq \vep$, by \eqref{int-diff-est} we have
\begin{align*}	
\int_0^T \EE\lf(\bar \vart(\|\bar Z_\vep(\eta_{\vep}(s))\|^2)- \bar \vart(\|\bar Z_\vep(s)\|^2)\ri)ds	\leq &\    \alpha L_b \EE\int_{0}^{T}\|\bar Z_\vep(s)\|^{\alpha-1}\|\bar Z_\vep(\vr_\vep(s)) -\bar Z_\vep(s)\|^{\nu}1_{\{\|\bar Z_\vep(r)\| > \bar B\}}ds  \\
& \ +\f{\|\s\|_\infty^2}{2} \EE\int_0^{T}   \|\bar Z_\vep(r)\|^{\alpha-2}1_{\{\|\bar Z_\vep(r)\| > \bar B\}}\ dr \\
	& \ +  \alpha\|\s\|_\infty \mdpsc \EE \int_0^{T} \|\bar Z_\vep(r)\|^{\alpha-1}\| \psi(r)\| 1_{\{\|\bar Z_\vep(r)\| > \bar B\}} dr\\
	&\ +\f{\alpha(\alpha-2)\|\s\|^2_\infty}{4}\EE\int_0^{T} \|\bar Z_\vep(r)\|^{\alpha -2}1_{\{\|\bar Z_\vep(r)\| > \bar B\}} dr\\
	& \ +\hat C^{6,B}(T). 	
\end{align*}
By the assertion of Lemma \ref{ctrlproc-bd-2} and by the steps used in the same lemma it easily follows that each of the expectations in the above display is bounded by a constant $\hat C^7(T)$ depending on parameters, $B, \cnst, L_b,  \|\s\|_\infty, \alpha, \nu, M.$ 
The assertion now follows from \eqref{split-1}.

\end{proof}

\begin{proposition}\label{prop-tight}
	Let $\{\psi_\vep\}$ be such that $\int_0^T \|\psi_\vep(s)\|^2ds \leq M$ for some constant $M>0$. Let $\bar Z_\vep \equiv \bar Z_\vep^{\psi_\vep}$ satisfy \eqref{ctrl-EuSDE1} with $\psi$ replaced by $\psi_\vep$, and define the occupation measure
	 $\bar R_\vep$ on $\B_T$ by
	 \begin{align}\label{def-occmeas}
	 	\bar R_\vep([0,t] \times A \times B) = \int_0^t 1_{\{\bar Z_\vep (s) \in A\}}1_{\{\psi_\vep (s) \in B\}}ds.
	 \end{align} 
	 Assume that
	 \begin{enumerate}[(i)]
	 \item
	  the step size $\Delta(\vep)$ is such that $(\Delta(\vep)/\vep)^{\nu/2}/\sqrt \vep \rt 0$, as $\vep \rt 0$;
	  \item
	   $f :[0,\infty)\times\R^d \rt \R^n$ satisfies Assumption \ref{assum-f}, with $\modu(\Delta) = \sqrt\Delta$;
	   \item Condition \ref{cond_SDE_inv} (with $\alpha>0$), Condition \ref{cond_SDE_coef-ex}, Condition \ref{u-reg-assum}, and Assumption \ref{assum-exp-u} hold.
	  \end{enumerate} 
  Then $(\bar R_\vep, \bar \Up_\vep(f))$ is tight in $\SC{M}_1(\B_T)\times C([0,T]:\R^d)$, and any limit point $(R,\xi)$ satisfies \eqref{Rset-1} - \eqref{Rset-3}, where $\bar\Up_\vep(f)$ was defined before \eqref{ctrl-EuSDE1}.

  Moreover, the same assertion is true for $(\bar R_\vep,  \Xi^R_\vep(f)/\mdpsc)$ if
 \begin{itemize}
 	\item Assumption \ref{assum-f-2}-(A) holds and $\Delta(\vep)$ is chosen such that $$\min\lf\{(\Delta(\vep)/\vep)^{\nu/2}/\sqrt\vep, (\Delta(\vep)/\vep)^{\nu_f/2}/\mdpsc\ri\} \rt 0,$$ 
 	as $\vep \rt 0$ (thus, in particular, if $(\Delta(\vep)/\vep)^{\tilde \nu/2}/\sqrt \vep \rt 0$, where $\tilde \nu = \nu \wedge \nu_f$); OR,
 	\item Assumption \ref{assum-f-2}-(B) holds with $p_0' \leq \alpha$, and $\Delta(\vep)$ is chosen such that 
 	$(\Delta(\vep)/\vep)^{\nu/2}/\sqrt\vep \rt 0$ as $\vep \rt 0$.
 \end{itemize}

\end{proposition}	

\begin{proof}
	We start by establishing the tightness of $\bar R_\vep$ and toward this end,  we need to show that for every $\eta>0$, there exists a constant $C_\eta$ such that
	\begin{align} \label{R-tight}
		\sup_\vep \EE \bar R_\vep \lf\{\y: \|x\| + \|z\| >C_\eta\ri\} \leq \eta
	\end{align}
	where recall that $\y$ denotes a typical tuple $(s,x,z)$ in $\B_T$.
	Note that  for all $0<\vep<1$,
	\begin{align} \label{R-mmtbd-1}
		\int_0^T \|\psi^\vep(s)\|^2 ds  = \int_{\B_T} \|z\|^2 \bar R_\vep(d\y) \leq M,
	\end{align}
	and by Lemma \ref{ctrlproc-bd-2},
	\begin{align} \label{R-mmtbd-2}
		\sup_\vep \EE \int_{\B_T} \|x\|^{2\alpha} \bar R_\vep(d\y) = \sup_{\vep} \EE\int_{0}^T \|\bar Z_\vep(s)\|^{2\alpha} ds <\infty.
	\end{align}
	\eqref{R-tight} now follows after an application of Markov inequality. 
	
	Let $\{\tilde t_k\}_{k=0}^N$ be a partition of $[0,T]$ such that  $\tilde \Delta = \tilde t_k - \tilde t_{k-1} =\vep$.
	Applying It\^o-Krylov lemma \cite{Kry80} to each component $u_l$, we have for $r \in [t_k, t_{k+1}]$,
	\begin{align} \non
		u_l(\tilde t_k, \bar Z_\vep(r))  =& \ u_l(\tilde t_k, \bar Z_\vep(\tilde t_k) ) + \f{1}{\vep}\lf( \int_{\tilde t_k}^r \nabla^Tu_l(\tilde t_k, \bar Z_\vep(s))b(\bar Z_\vep(\vr_\vep(s))) ds \ri.\\ \non
&	\	 + \lf.\f{1}{2} \int_{\tilde t_k}^r tr\lf(D^2u_l(\tilde t_k, \bar Z_\vep(s))a(\bar Z_\vep(\vr_\vep(s)))\ri)  ds \ri)\\
 \non
 & +  \f{\mdpsc }{\vep} \int_{\tilde t_k}^r \nabla^Tu_l(\tilde t_k, \bar Z_\vep(s)) \s(\bar Z_\vep(\vr_\vep(s))) \psi_\vep(s) ds\\
\non
 & +\f{1}{\vep^{1/2}}\int_{\tilde t_k}^r \nabla^Tu_l(\tilde t_k, \bar Z_\vep(s))\s(\bar Z_\vep(\vr_\vep(s)))dW(s).
 \end{align}
 Let $k_0 = \max\{k: \tilde t_k <t\}$ and without loss of generality assume that $\tilde t_{k_0+1}=t$.  Summing over $k$, we then have
 \begin{align}\non
\sum_{k=0}^{k_0}\lf(u_l(\tilde t_k, \bar Z_\vep(\tilde t_{k+1}))  -  u_l(\tilde t_k, \bar Z_\vep(\tilde t_k) )\ri)= &\  \f{1}{\vep} \int_0^t \SC{L}u(\eta_\vep(s), \cdot)(\bar Z_\vep(s))ds + \mdpsc \SC{E}^\vep_0(t)/\vep	\\ \non
	 &\ +  \f{\mdpsc }{\vep} \int_{0}^t \nabla^Tu_l(\eta_\vep(s), \bar Z_\vep(s)) \s(\bar Z_\vep(\vr_\vep(s))) \psi_\vep(s) ds\\
	   \label{u-Ito-0}
 &\ +\f{1}{\vep^{1/2}}\int_{0}^t \nabla^Tu_l(\eta_\vep(s), \bar Z_\vep(s))\s(\bar Z_\vep(\vr_\vep(s)))dW(s),
	\end{align}	
where, as before, $\eta_\vep(s) = \tilde t_{k}$ if $\tilde t_k< s \leq \tilde t_{k+1},$ and
\begin{align*}
 \SC{E}^\vep_{0,l}(t)   =&\  \f{1}{\mdpsc }\lf(\int_{0}^t \lf[\nabla^Tu_l(\eta_\vep(s), \bar Z_\vep(s))b(\bar Z_\vep(\vr_\vep(s))) ds 
	 + \f{1}{2}  tr\lf(D^2u_l(\eta_\vep(s), \bar Z_\vep(s))a(\bar Z_\vep(\vr_\vep(s)))\ri)\ri]  ds\ri.\\
	&\ -  \lf. \int_0^t \SC{L}u_l(\eta_\vep(s), \cdot)(\bar Z_\vep(s))ds\ri),
\end{align*}	
and therefore from \eqref{u-Ito-0},
	\begin{align}\non
		\f{\vep}{\mdpsc }	\lf(u(t, \bar Z_\vep(t)) - u(0, x_0)\ri)  = &  \f{\vep}{\mdpsc }\sum_{k=0}^{k_0} \lf(u(\tilde t_{k+1}, \bar Z_\vep(\tilde t_{k+1})) - u(\tilde t_k, \bar Z_\vep(\tilde t_{k+1}))\ri) \\ \non
		& + \f{\vep}{\mdpsc } \sum_{k=0}^{k_0}\lf(u(\tilde t_k, \bar Z_\vep(\tilde t_{k+1})) - u(\tilde t_k, \bar Z_\vep(\tilde t_{k}))\ri)\\ \non
		= &\ \SC{E}^\vep_0(t)+\SC{E}^\vep_1(t)  - \f{1}{\mdpsc } \int_0^t  f(\eta_\vep(s), \bar Z_\vep(s)) ds \\ \non
		& + \int_0^t Du(\eta_\vep(s), \bar Z_\vep(s)) \s(\bar Z_\vep(s)) \psi_\vep(s)ds\\ \non
		& \hs{0.5cm} + \f{\sqrt{\vep}}{\mdpsc }\int_0^t Du(\eta_\vep(s), \bar Z_\vep(s))\s(\bar Z_\vep(\vr_\vep(s)))dW(s)\\ \non
		 =& \ - \bar \Up_\vep(f)(s) +  \int_{\B_t}  Du(s, x) \s(x) z \bar R_\vep(d\y)\\ \non
		 &\ +\ \f{\sqrt{\vep}}{\mdpsc }\int_0^t  Du(\eta_\vep(s), \bar Z_\vep(s))\s(\bar Z_\vep(\vr_\vep(s)))dW(s)\\ \label{u-Ito}
		 &\ + \SC{E}^\vep_0(t)+\SC{E}^\vep_1(t)+\SC{E}^\vep_2(t) + \SC{E}^\vep_3(t), 
	\end{align}
where the quantities $\SC{E}^\vep_i(t)$ are defined below:
\begin{align*}
\SC{E}^\vep_1(t) \doteq &  \f{\vep}{\mdpsc }\sum_k\lf(u(\tilde t_{k+1}, \bar Z_\vep(\tilde t_{k+1})) - u(\tilde t_k, \bar Z_\vep(\tilde t_{k+1}))\ri);\\
\SC{E}^\vep_2(t) \doteq & \f{1}{\mdpsc }\lf( \int_0^t  f(s, \bar Z_\vep(s)) ds -  \int_0^t  f(\eta_\vep(s), \bar Z_\vep(s)) \ri) ds;\\
\SC{E}^\vep_3(t) \doteq & \int_0^t Du(\eta_\vep(s), \bar Z_\vep(s)) \s(\bar Z_\vep(\vr_\vep(s))) \psi_\vep(s) ds - \int_0^t Du(s, \bar Z_\vep(s)) \s(\bar Z_\vep(s)) \psi_\vep(s) ds,
\end{align*}	
and $\SC{E}^\vep_0$ is of course given by $\SC{E}^\vep_0 = (\SC{E}^\vep_{0,1}, \hdots,\SC{E}^\vep_{0,n})^T.$

Recalling that $\tilde \Delta = \vep$ , it follows from Assumption \ref{assum-f}-(ii), the fact that $q_0\leq 2\alpha$ (Assumption \ref{assum-exp-u}-(iii)), and Lemma \ref{ctrlproc-bd-2} that
\begin{align}\label{errf-diff}
\EE\lf(\sup_{s\leq t}\|\SC{E}^\vep_2(s)\|\ri) \leq &\  \const(T) \f{\sqrt \vep}{\mdpsc }\lf(\int_0^t (1+ \EE(\|\bar Z_\vep(s)\|^{q_0})ds\ri) \rt 0,
\end{align}	
as $\vep \rt 0$. Next, by Condition \ref{u-reg-assum}-(ii) and H\"older continuity of $\s$,
\begin{align*}
	\|\SC{E}^\vep_3(s)\| \leq & \  \const_1(T)\|\s\|_\infty \sqrt \vep \int_0^t (1+ \|\bar Z_\vep(s)\|)^{q_2} \|\psi_\vep(s)\|ds \\ 
	& \ + \const_1(T)L_\s \int_0^t (1+ \|\bar Z_\vep(s)\|)^{p_2} \|\bar Z_\vep(s) - \bar Z_\vep(\vr_\vep(s))\|^\nu \|\psi_\vep(s)\| ds \\
	& \ \equiv  \SC{E}^\vep_{3,1}(t) + \SC{E}^\vep_{3,2}(t). 
\end{align*}
By (a) the assumption that $q_2 \leq \alpha$ (Assumption \ref{assum-exp-u}-(iii)), (b) Lemma \ref{ctrlproc-bd-2}, and (c) Cauchy-Schwarz inequality,
\begin{align*}
\EE\lf(\sup_{s\leq t}|\SC{E}^\vep_{3,1}(s)|\ri) \leq &\  \const_1(T)\|\s\|_\infty \sqrt \vep \EE \lf(\int_0^t (1+ \|\bar Z_\vep(s)\|)^{q_2} \|\psi_\vep(s)\|ds\ri) \\ 
 \leq & \const_1(T)\|\s\|_\infty \sqrt \vep  \EE \lf(\int_0^t\EE (1+\|\bar Z_\vep(s)\|)^{2q_2}ds\ \EE\int_0^t \|\psi_\vep(s)\|^2 ds\ri)^{1/2}\\
 \leq & \const_1(T)\|\s\|_\infty  M^{1/2} \sup_{\vep} \lf(\int_0^t \EE\lf(1+\|\bar Z_\vep(s)\|\ri)^{2q_2} ds\ri)^{1/2}\sqrt \vep\  
 \rt \  0,
\end{align*}	
as $\vep \rt0$. Also, 
\begin{align*}
\EE\lf(\sup_{s\leq t}|\SC{E}^\vep_{3,2}(s)|\ri) \leq &\  \const_1(T)L_\s \EE\lf( \int_0^t (1+ \|\bar Z_\vep(s)\|)^{p_2} \|\bar Z_\vep(s) - \bar Z_\vep(\vr_\vep(s))\|^\nu \|\psi_\vep(s)\| ds\ri)\\
\leq &  \const_1(T)L_\s   \EE\lf(  \sup_{s\leq T}(1+ \|\bar Z_\vep(s)\|)^{p_2}\int_0^t \|\bar Z_\vep(s) - \bar Z_\vep(\vr_\vep(s))\|^\nu \|\psi_\vep(s)\| ds\ri)\\
\leq &  \const_1(T)L_\s   \EE\lf[  \sup_{s\leq T}(1+ \|\bar Z_\vep(s)\|)^{p_2}\lf(\int_0^t \|\bar Z_\vep(s) - \bar Z_\vep(\vr_\vep(s))\|^{2\nu}\ri)^{1/2} \lf(\int_0^t\|\psi_\vep(s)\|^2 ds\ri)^{1/2}\ri]\\
\leq &  \const_1(T)L_\s M^{1/2}  \EE\lf[  \sup_{s\leq T}(1+ \|\bar Z_\vep(s)\|)^{p_2}\lf(\int_0^t \|\bar Z_\vep(s) - \bar Z_\vep(\vr_\vep(s))\|^{2\nu}ds\ri)^{1/2}\ri]\\
\leq &  \const_1(T)L_\s M^{1/2}  \lf[\EE\lf( \sup_{s\leq T}(1+ \|\bar Z_\vep(s)\|)^{2p_2}\ri)\ri]^{1/2}\lf[\int_0^t\EE \|\bar Z_\vep(s) - \bar Z_\vep(\vr_\vep(s))\|^{2\nu}ds\ri]^{1/2}\\
\leq &  \const_1(T)L_\s M^{1/2} \tilde C^4(T)  \f{(\Delta(\vep)/\vep)^{\nu/2}}{\sqrt \vep}\lf[\vep\EE\lf( \sup_{s\leq T}(1+ \|\bar Z_\vep(s)\|)^{2p_2}\ri)\ri]^{1/2}\\
& \ \times \lf[\int_0^t (1+\EE \|\bar Z_\vep(s)\|^{2\nu\alpha})ds\ri]^{1/2}
\leq \ \bar C^0(T)   \f{(\Delta(\vep)/\vep)^{\nu/2}}{\sqrt \vep},
\end{align*}	
for some constant $\bar C^0(T)$, where we used  (a) Corollary \ref{mmt-bd-ctrl} (b) the assumption that $p_2 \leq (1+\alpha)/2$ (Assumption \ref{assum-exp-u}-(ii)), and (c) Lemma \ref{ctrlproc-bd-2}.
Thus by the choice of discretization step $\Delta(\vep)$ (see (i) in the hypotheses of the proposition), as $\vep \rt 0,$
$$\EE\lf(\sup_{s\leq t}|\SC{E}^\vep_{3,2}(s)|\ri) \rt 0.$$
We now consider $\SC{E}^\vep_1$. Note that because of  Condition \ref{u-reg-assum}-(iii)
\begin{align*}
|\SC{E}^\vep_1(t)| \leq &\ \f{\vep}{\mdpsc } \sum_{k=0}^{k_0} \const_1(T) \lf(1+ \|\bar Z_\vep(\tilde t_{k+1})\|^{q_1}\ri) \tilde\Delta^{1/2}\\
                          = &\ \f{\sqrt\vep}{\mdpsc } \sum_{k=1}^{k_0+1} \const_1(T) \lf(1+ \|\bar Z_\vep(\tilde t_{k})\|^{q_1}\ri) \vep\\
      \leq & \  \f{\sqrt\vep}{\mdpsc } \const_1(T) \int_0^T  \lf(1+ \|\bar Z_\vep(\eta_\vep(s))\|^{q_1}\ri) ds,
  \end{align*}
 and because of Assumption \ref{assum-exp-u}-(iii), it follows either by Remark \ref{int-disc-est-1} or by Corollary \ref{dis-mmt-bd-ctrl} (depending on whether $0<\alpha \leq 1$ or not)  that $\EE\lf(\sup_{s\leq T}\|\SC{E}^\vep_1(s)\|\ri) \rt 0$ as $\vep \rt 0$. 
 
 To estimate $\SC{E}^{\vep}_{0}$, note that for each $l$, by Condition \ref{u-reg-assum}-(ii) \& (v),
 \begin{align*}
 \sup_{t\leq T}\|\SC{E}^{\vep}_{0,l}(t)\| \leq & \ \f{1}{\mdpsc } \lf(\int_0^T \|\nabla^Tu_l(\eta_\vep(s), \bar Z_\vep(s))\| \|b(\bar Z_\vep(s)) - b(\bar Z_\vep(\vr_\vep(s)))\| ds \ri.\\
 & \ + \lf. \int_0^T \|D^2u_l(\eta_\vep(s), \bar Z_\vep(s))\|\|a(\bar Z_\vep(s)) - a(\bar Z_\vep(\vr_\vep(s)))\| ds\ri)\\
 \leq & \f{\bar C^1(T)}{\mdpsc } \int_0^T (1+\| \bar Z_\vep(s)\|^{p_2 \vee p_3})\|\bar Z_\vep(s) - \bar Z_\vep(\vr_\vep(s))\|^{\nu} ds,
 \end{align*}
 for some constant $\bar C^1(T)$.
 Hence by (a) Cauchy-Schwarz inequality, (b) Lemma \ref{lem-diff-est}-(iv), (c) Assumption \ref{assum-exp-u}-(ii) \& (iv), and  (d) Lemma \ref{ctrlproc-bd-2},
 \begin{align*}
 \EE\lf(\sup_{t\leq T}\|\SC{E}^{\vep}_{0,l}(t)\| \ri) \leq & \ \f{\bar C^3(T)}{\mdpsc } \lf(\int_0^T\EE(1+\| \bar Z_\vep(s)\|^{2(p_2 \vee p_3)})ds \int_0^T \EE\|\bar Z_\vep(s) - \bar Z_\vep(\vr_\vep(s))\|^{2\nu}ds\ri)^{1/2}\\
 \leq &\ \bar C^3(T) \tilde C^4(T) \f{(\Delta/\vep)^{\nu/2}}{\mdpsc}  \lf(\int_0^T\EE(1+\| \bar Z_\vep(s)\|^{2(p_2 \vee p_3)})ds \int_0^T \EE (1+\|\bar Z_\vep(s)\|^{2\nu \alpha}ds\ri)^{1/2}\\
 \leq & \  \bar C^4(T) \f{(\Delta/\vep)^{\nu/2}}{\mdpsc} = \bar C^4(T) \f{(\Delta/\vep)^{\nu/2}}{\sqrt \vep}\f{\sqrt \vep}{\mdpsc}     \rt 0,
 \end{align*}
 as $\vep \rt 0$ by the choice of grid size $\Delta(\vep)$. Here $\bar C^3(T), \bar C^4(T)$ are appropriate constants.

	We  next show that as $\vep \rt 0$
	\begin{align} \label{conv-u}
		\f{\vep}{\mdpsc } \EE\lf[\sup_{t\leq T}\|u(t, \bar Z_\vep(t)) - u(0, x_0)\|\ri] \rt 0.
	\end{align}
	Since $p_1 \leq (1+\alpha)/2$ (Assumption \ref{assum-exp-u}-(i)), there exists a constant $\bar C^5$ such that $\|x\|^{p_1} \leq \bar C^5(1+\|x\|^{(1+\alpha)/2})$. Consequently,
	\begin{align*}
		\f{\vep}{\mdpsc } \EE\lf(\sup_{s\leq t} \|\bar Z_\vep(s)\|^{p_1}\ri) \leq  & \f{\vep}{\mdpsc }\bar C^5+ \f{\vep}{\mdpsc } \EE\lf(\sup_{s\leq t}  \|\bar Z_\vep(s)\|^{(1+\alpha)/2}\ri)\\
		\leq & \f{\vep}{\mdpsc }\bar C^5+ \f{\sqrt \vep}{\mdpsc } \lf[\vep \EE\lf(\sup_{s\leq t}  \|\bar Z_\vep(s)\|^{1+\alpha}\ri)  \ri]^{1/2} \rt 0, 
	\end{align*}	
	as $\vep \rt 0$ by Corollary \ref{mmt-bd-ctrl}, and \eqref{conv-u} follows because of Condition \ref{u-reg-assum}-(i). 
	
	For the martingale term we use Burkholder-Davis-Gundy inequality, Lemma \ref{ctrlproc-bd-2} and the fact that $p_2< \alpha$, to get for each $l=1,2,\hdots,n,$
	\begin{align*}
		\EE\lf[\sup_{r\leq t}|\f{\sqrt{\vep}}{\mdpsc }\int_0^r \nabla^T u_l(\eta_\vep(s), \bar Z_\vep(s))\s( \bar Z_\vep(s))dW(s)|^2\ri] \leq &
		\f{\vep}{\mdpscii} \EE \int_0^t \|\nabla^T u_l(\eta_\vep(s), \bar Z_\vep(s))\s(\bar Z_\vep(s))\|^2 ds\\
		\leq & \|\s\|^2_\infty \const_1(T)  \f{ \vep}{\mdpscii} \EE \int_0^t (1+ \|\bar Z_\vep(s)\|^{p_2})^2 ds\\
		& \ \rt  0,
	\end{align*}
	as $\vep \rt 0$. 

		It now follows from \eqref{u-Ito} that to show tightness $\bar\Up_\vep(f)$ we only need to show tightness of $\bar\L_\vep$, where
		\begin{align*}
			\bar\L_\vep(t) = 	\int_{\B_t} Du(s, x) \s(x) z \bar R_\vep(d\y) = \int_0^t Du(s, \bar Z_\vep(s)) \s(\bar Z_\vep(s)) \psi_\vep(s)ds. 
		\end{align*}	
		%
		Toward this end, notice that by Condition \ref{u-reg-assum}-(ii) for any $K>0$,
		\begin{align*}
			\|\bar\L_\vep(t+h) - \bar\L_\vep(t) \| \leq & \	\bar C^{5,0}(T)\int_{t}^{t+h} (1+\|\bar Z_\vep(s)\|^{p_2})\|\psi_\vep(s)\| ds\\
			\leq &	C^{5,0}(T)\lf[  M^{1/2}h^{1/2} + \int_{t}^{t+h}\|\bar Z_\vep(s)\|^{p_2}\|\psi_\vep(s)\|ds\ri]\\
			\leq & C^{5,0}(T)\lf[M^{1/2}h^{1/2} + K^{p_2}\int_{t}^{t+h}\|\psi_\vep(s)\|ds + \int_{t}^{t+h}\|\bar Z_\vep(s)\|^{p_2}1_{\{\|\bar Z_\vep(s)\| >K\}}\|\psi_\vep(s)\|ds\ri]\\
			\leq & C^{5,0}(T)\lf[ M^{1/2}h^{1/2} + K^{p_2}M^{1/2}h^{1/2} + \f{1}{K^{\alpha-p_2}}\int_{t}^{t+h}\|\bar Z_\vep(s)\|^{\alpha}\|\psi_\vep(s)\|ds\ri]\\
			\leq&  C^{5,0}(T)\lf[M^{1/2}h^{1/2} + K^{p_2}M^{1/2}h^{1/2} + \f{1}{K^{\alpha-p_2}}\lf(\int_{0}^{T}\|\bar Z_\vep(s)\|^{2\alpha}ds +M\ri)\ri],
		\end{align*}
		where $C^{5,0}(T) = \const_1(T)\|\s\|_\infty$ is a constant independent of $K$. 
		Taking $K = h^{-1/4p_2}$, and using Lemma \ref{ctrlproc-bd-2} we have that for some constant $\bar C^{5,1}(T)$,
		\begin{align*}
			\EE\lf[\sup_{0\leq t\leq t+h\leq T}  \|\bar\L_\vep(t+h) - \bar\L_\vep(t) \|\ri] \leq \bar C^{5,1}(T)\lf(h^{1/2}+h^{1/4}+h^{(\alpha-p_2)/4p_2}\ri)
		\end{align*}	
		Recalling that $\alpha>p_2$, tightness of 	$\bar\L_\vep$ is now immediate. Here, of course, we assumed $p_2>0$. The argument for $p_2=0$ (that is, when $Du$ is bounded) is much simpler.
		
	Let $(R, \xi)$ be a limit point of $\{(\bar R_\vep, \bar \Up_\vep(f))\}$ and by Skorohod representation theorem assume without loss of generality that $(\bar R_\vep, \bar \Up_\vep(f)) \rt (R, \xi)$ a.s in $\SC{M}_1(\B_T)\times C([0,T]:\R^d)$ as $\vep \rt 0$, at least, along some subsequence.
	Note that \eqref{Rset-1} follows from \eqref{R-mmtbd-1}
	and Fatou's lemma.	
		
		Now Condition \ref{u-reg-assum}-(ii),  
		the fact that $p_2<\alpha$  (Assumption \ref{assum-exp-u}-(ii)),  \eqref{R-mmtbd-1}, \eqref{R-mmtbd-2}, and an application of Lemma \ref{lem-weakconv} imply that as $\vep \rt 0$,
		\begin{align*}
			\int_{\B_t} Du(s, x) \s(x) z \bar R_\vep(d\y) \rt  \int_{\B_t} Du(s, x) \s(x) z  R(d\y).
		\end{align*}
	Thus from \eqref{u-Ito} and the above calculations  it follows that \eqref{Rset-2} holds, that is,
	$$\xi(t) = \int_{\B_t} Du(s, x) \s(x) z  R(d\y).$$
	
%
	Finally, for \eqref{Rset-3}, let  $g \in C^2_b(\R^d, \R)$.  Then a simpler version of \eqref{u-Ito-0} with $u$ replaced by $g$ and much easier calculations  reveal  that 
		$$\int_{\B_t} \SC{L}g(x) R(d\y) = 0,\quad  0\leq t \leq T.$$

	For the result on $\bar \Xi^R_\vep(f),$ notice we only need to show that $\tilde{\SC{E}}^\vep_2$, defined by
	\begin{align*}
	\tilde{\SC{E}}^\vep_2(t) = & \f{1}{\mdpsc } \lf(\int_0^t  f(\eta_\vep(s), \bar Z_\vep(s)) ds -\int_0^t  f(\vr_\vep(s), \bar Z_\vep(\vr_\vep(s))) ds \ri) 
	\end{align*}	
goes to $0$. We work only under Assumption \ref{assum-f-2}-B. The steps under Assumption \ref{assum-f-2}-A are similar and simpler. 

Writing
\begin{align*}
\tilde{\SC{E}}^\vep_2(t) = &\ \f{1}{\mdpsc } \int_0^t \lf( f(\eta_\vep(s), \bar Z_\vep(s))  - f(\vr_\vep(s),  \bar Z_\vep(s)) \ri)ds \\
& \hs{.2cm}- \f{1}{\mdpsc } \int_0^t \lf( f(\vr_\vep(s),  \bar Z_\vep(s)) - f(\vr_\vep(s),  \bar Z_\vep(\vr_\vep(s))) \ri)ds \\
	  \equiv & \ \tilde{\SC{E}}^\vep_{2,1}(t) + \tilde{\SC{E}}^\vep_{2,2}(t), 
\end{align*}
it is immediate that (c.f. \eqref{errf-diff})  
\begin{align*}
\EE\lf[\sup_{t\leq T }|\tilde{\SC{E}}^\vep_{2,1}(t)|\ri] \ \rt 0,
\end{align*}
as $\vep \rt 0$.

Next, for each $l=1,\hdots,n$, by the mean value theorem,
\begin{align*}
	\tilde{\SC{E}}^\vep_{2,2,l}(t) = \f{1}{\mdpsc } \int_0^t  \nabla f_l\lf(\eta_\vep(\vr_\vep(s)), \theta_l(s) \bar Z_\vep(\vr_\vep(s))+(1-\theta_l(s))\bar Z_\vep(s)\ri) \lf( \bar Z_\vep(\vr_\vep(s)) -  \bar Z_\vep(s)\ri)ds 
	\end{align*}
for some $\theta_l(s) \in (0,1).$ Thus by Assumption \ref{assum-f-2}-B,
\begin{align*}
\sup_{t\leq }|\tilde{\SC{E}}^\vep_{2,2,l}(t)| \leq & \f{\bar C^6(T)}{\mdpsc}\int_0^t \lf(\|\bar Z_\vep(\vr_\vep(s))\|^{p_0'}+\|\bar Z_\vep(s)\|^{p_0'}\ri) \| \bar Z_\vep(\vr_\vep(s)) -  \bar Z_\vep(s)\| ds\\
 \leq & \f{\bar C^7(T)}{\mdpsc}\int_0^t \lf(\|\bar Z_\vep(s)\|^{p_0'} \| \bar Z_\vep(\vr_\vep(s)) -  \bar Z_\vep(s)\| + \| \bar Z_\vep(\vr_\vep(s)) -  \bar Z_\vep(s)\|^{p_0'+1}\ri) ds,
\end{align*}
where $\bar C^6(T)$ and $\bar C^7(T)$ are appropriate constants.
Now, by Lemma \ref{lem-diff-est}-(iv) and Cauchy-Schwarz inequality,
\begin{align*}
\EE\lf[\sup_{t\leq }|\tilde{\SC{E}}^\vep_{2,2,l}(t)|\ri] \leq & \f{\bar C^7(T)}{\mdpsc} \lf[ (\Delta/\vep)^{1/2}\lf(\int_0^t\EE\|\bar Z_\vep(s)\|^{2p_0'}\ri)^{1/2} \lf(\int_0^t\EE\|\bar Z_\vep(s)\|^{2\alpha}\ri)^{1/2}\ri.\\
&  \hs{2cm}\lf. + (\Delta/\vep)^{(p_0'+1)/2}\int_0^t\EE\|\bar Z_\vep(s)\|^{(p_0'+1)\bar\alpha}ds\ri],
\end{align*}
and since $2p_0' \leq 2\alpha$, $(p_0'+1)\bar\alpha\leq 2\alpha$, we have by Lemma \ref{ctrlproc-bd-2}
\begin{align*}
\EE\lf[\sup_{t\leq T}|\tilde{\SC{E}}^\vep_{2,2,l}(t)|\ri] \leq &\ \f{\bar C^7(T) (\Delta/\vep)^{1/2}}{\mdpsc}\leq \bar C^7(T)(\Delta/\vep)^{(1-\nu)/2} \f{ (\Delta/\vep)^{\nu/2}}{\sqrt \vep}\f{\sqrt \vep}{\mdpsc}  \rt 0,
\end{align*}
as $\vep \rt 0$.
\end{proof}	

\subsection{Proof of Theorem \ref{th:cltZ}} \label{sec:cltZ-pf}
Notice because of Lemma \ref{Z-proc-mmt-bd}, $\Xi^\vep$ is tight and as in proof of Proposition \ref{prop-tight}, any limit point $\Xi$ satisfies
\begin{align*}
	\int_{\R^d\times[0,t]} \SC{L}g(x) \Xi(dx\times ds) = 0,\quad  0\leq t \leq T.
\end{align*}	
Writing $\Xi(dx\times ds) = \Xi_{2|1}(dx|s)ds$ it follows from the uniqueness of the invariant measure $\pi$ that
\begin{align}\label{pi-limpt}
\Xi(dx\times ds) = \pi(dx)ds.
\end{align}	 
Next, under the hypotheses, by Proposition \ref{u-est}, the solution $u$ of the Poisson equation exists and satisfies Condition \ref{u-reg-assum} for some nonnegative $p_1, p_2, p_3, q_1$ and $q_2$.\\
\np 
Now, again using the coarser partition $\{\tilde t_k\}$ with $\tilde\Delta(\vep) = t_k -t_{k-1} =\vep$, similar to \eqref{u-Ito}, we have for the original process $Z^\vep$,
\begin{align}\non
		\sqrt{\vep} \lf(u(t,  Z^\vep(t)) - u(0, x_0)\ri)  = &  \sqrt{\vep}\sum_{k=0}^{k_0} \lf(u(\tilde t_{k+1},  Z^\vep(\tilde t_{k+1})) - u(\tilde t_k,  Z^\vep(\tilde t_{k+1}))\ri) \\ \non
	& + \sqrt\vep \sum_{k=0}^{k_0}\lf(u(\tilde t_k,  Z^\vep(\tilde t_{k+1})) - u(\tilde t_k,  Z^\vep(\tilde t_{k}))\ri)\\ \non
	= &\ \hat{\SC{E}}^\vep_0(t)+\hat{\SC{E}}^\vep_1(t)  - \f{1}{\sqrt \vep } \int_0^t  f(\eta_\vep(s),  Z^\vep(s)) ds \\ \non
	& \hs{0.5cm} + \int_0^t Du(\eta_\vep(s), Z^\vep(s))\s( Z^\vep(\vr_\vep(s)))dW(s)\\  \label{u-Ito-Z}
	=& \ - \vep^{-1/2}\Xi_\vep(f)(t) +  \hat\mart^\vep(t)
 +\hat{\SC{E}}^\vep_0(t)+\hat{\SC{E}}^\vep_1(t)+\hat{\SC{E}}^\vep_2(t) + \hat{\SC{E}}^\vep_3(t), 
\end{align}
where 
\begin{align*}
	\hat{\SC{E}}^\vep_{l,0}(t)   =&\  \f{1}{\sqrt\vep }\lf(\int_{0}^t \lf[\nabla^Tu_l(\eta_\vep(s),  Z^\vep(s))b(Z^\vep(\vr_\vep(s))) ds 
	+ \f{1}{2}  tr\lf(D^2u_l(\eta_\vep(s), Z^\vep(s))a( Z^\vep(\vr_\vep(s)))\ri)\ri]  ds\ri.\\
	&\ -  \lf. \int_0^t \SC{L}u_l(\eta_\vep(s), \cdot)(Z^\vep(s))ds\ri)\\
	\hat{\SC{E}}^\vep_1(t) \doteq &  \sqrt\vep\sum_k\lf(u(\tilde t_{k+1},  Z^\vep(\tilde t_{k+1})) - u(\tilde t_k,  Z^\vep(\tilde t_{k+1}))\ri);\\
	\hat{\SC{E}}^\vep_2(t) \doteq & \f{1}{\sqrt\vep }\lf( \int_0^t  f(s,  Z^\vep(s)) ds -  \int_0^t  f(\eta_\vep(s),  Z^\vep(s)) \ri) ds;\\
	\hat{\SC{E}}^\vep_3(t) \doteq & \int_0^t Du(\eta_\vep(s), Z^\vep(s) \s(Z^\vep(\vr_\vep(s)))  -  Du(s, Z^\vep(s)) \s( Z^\vep(s)) dW(s),\\
	\hat \mart^\vep(t) \doteq& \int_0^t Du(s, Z^\vep(s)) \s( Z^\vep(s)) dW(s).
\end{align*}
Notice that
\begin{align*}
\EE[\sup_{t\leq T}\|\hat{\SC{E}}^\vep_1(t)\|] \leq &\ \sqrt \vep \const_1(T)\sum_{k} \EE(1+\|Z^\vep(\tilde t_{k+1})\|^{q_1}) \modu(\vep)   \\
\leq &\ \f{\modu(\vep)}{\sqrt\vep} \const_1(T) \int_0^T \EE (1+\|Z^\vep(\eta_\vep(s))\|^{q_1}) ds \\
\rt 0,
\end{align*}	
as $\vep \rt 0.$ Also, since $\modu(\vep) \rt 0$, by Lemma \ref{lem-diff-est-Z} and Lemma \ref{Z-proc-mmt-bd}
\begin{align*}
	\EE[\sup_{t\leq T}\|\hat{\SC{E}}^\vep_3(t)\|] \leq &\ \sqrt \const_1(T)\|\s\|_\infty\modu(\vep)\int_0^T\EE(1+\|Z^\vep(s)\|^{q_2}) ds   \\
	&\  +  L_b\const_1(T) \int_0^T \EE (1+\|Z^\vep(s)\|^{p_2})\|Z^\vep(s) - Z^\vep(\vr_\vep(s)))\|^\nu ds \\
	\rt 0,
\end{align*}
as $\vep\rt 0$.

Similarly, it easily follows that $\EE[\sup_{s\leq T}\|\hat{\SC{E}}^\vep_2(s)\|] \rt 0$, and by similar techniques used in the proof of Proposition \ref{prop-tight}, $(\Delta(\vep)/\vep)^{\nu/2}/\sqrt\vep \rt 0$ implies that $\EE[\sup_{s\leq T}\|\hat{\SC{E}}^\vep_0(s)\|] \rt 0$  as $\vep \rt 0$.

Moreover, since Corollary \ref{mmt-bd-Z} holds for any $q$, using Condition \ref{u-reg-assum}, it could be seen that 
$\sqrt \vep \EE\lf[\sup_{t\leq T}|u(t, Z^\vep(t))|\ri]  \rt 0$, as $\vep \rt 0$ (c.f. the proof of \eqref{conv-u}). 

For the martingale term we look at its quadratic variation. By \eqref{pi-limpt} and Lemma \ref{lem-weakconv}, it follows that as $\vep\rt 0$,
\begin{align*}
[\hat\mart^\vep]_t=&\ 	\int_0^t  Du(s, Z^\vep(s)) a(Z^\vep(s)) (Du(s, Z^\vep(s)))^T ds\\
                          =& \ \int_{\R^d\times[0,t]}  Du(s, x) a(x) (Du(s, x))^T \Xi^\vep(dx\times ds)\\
                          \rt&\ \int_{\R^d\times[0,t]}  Du(s, x) a(x) (Du(s, x))^T \pi(dx) ds = \int_0^t M_f(s)ds.
\end{align*}	
The last step used the equivalent expression of $M_f$ (defined in \eqref{M-mat-def}) given in Lemma \ref{M-mat-alt}. The result now follows from the martingale central limit theorem \cite[Chapter 7]{EK86}.

Finally, just as in the last part of the proof of Proposition \ref{prop-tight}, and using the same techniques,
$$\vep^{-1/2} \EE\lf[\sup_{t\leq T}\lf| \int_0^t \lf(  f(\eta_\vep(s),  Z^\vep(s)) - f(\vr_\vep(s),  Z^\vep(\vr_\vep(s))) \ri) ds\ri|\ri] \rt 0,$$ as $\vep \rt 0$, and the assertion for $\Xi^{R}_\vep(f)$ follows.

\subsection{LDP / Laplace principle upper bound - Theorem \ref{th:mdpZ}}\label{sec:LP-ub-pf}
The objective of this section is to prove the Laplace principle upper bound, that is, to show that
\begin{align}\label{LP-ub}
	\limsup_{\vep \rt 0}  \mdpscb \ln \EE \lf[\exp\Big(-F(\Up_\vep(f))/\mdpscb \Big)\ri] \leq -\inf_{\xi \in C([0,T], \R^d)} [I_f(\xi) + F(\xi)].
\end{align}

Note that \eqref{vrep} implies that for every $\vep>0$, there exists a sequence of $\{\psi^\vep\}$ such that 

\begin{align}\label{vrep-lb}
	-\mdpscb \ln \EE \lf[\exp\Big(-F(\Up_\vep(f))/\mdpscb \Big)\ri] \geq \f{1}{2}\EE\lf[ \int_0^T \|\psi^\vep(s)\|^2 ds + F(\bar \Up_\vep(f))\ri] - \vep,
\end{align}
Let  $ \bar R_\vep$ be as in Proposition \ref{prop-tight}. 
Since $F$ is bounded, using a standard localization argument, one can assume without loss of generality that 
\begin{align}\label{rate-bd}
\sup_{0<\vep<1} \int_0^T \|\psi^\vep(s)\|^2 ds \leq M
\end{align}	
for some constant $M>0$. By Proposition \ref{prop-tight}, $(\bar R_\vep, \bar \Up_\vep(f))$ is tight and any limit point $(R,\xi)$ satisfies \eqref{Rset-1} - \eqref{Rset-3}, and hence $(R,\xi) \in \SC{R}_\xi$, where $\SC{R}_\xi$ was introduced in Section \ref{sec:eq-rate} before  \eqref{Rset-1} - \eqref{Rset-3}. Assume, without loss of generality, that $(\bar R_\vep, \bar \Up_\vep(f) ) \RT (R,\xi)$ along the full sequence. Then
it follows from \eqref{vrep-lb} and Fatou's lemma that
\begin{align*}
	\liminf_{\vep \rt 0}-\mdpscb \ln \EE \lf[\exp\Big(-F(\Up_\vep (f))/\mdpscb \Big)\ri] & \geq \f{1}{2} \liminf_{\vep\rt 0} \EE\lf[\int_0^T \|\psi^\vep(s)\|^2 ds + \liminf_{\vep \rt 0}F(\bar \Up_\vep(f))\ri]\\
	& \geq  \f{1}{2} \liminf_{\vep\rt 0} \EE\lf[\int_{\B_T} \|z\|^2 \bar R_\vep(d\y) + F(\xi) \ri]\\
	& =  \EE\lf[\f{1}{2}  \int_{\B_T} \|z\|^2  R (d\y) + F(\xi)\ri]\\
	& \geq I(\xi) +  F(\xi), 
\end{align*}	
which proves \eqref{LP-ub}. Here we used the equivalent form of the rate function given in Lemma \ref{rate-equivalence}.

The proof for the Laplace principle upper bound for $\Xi^{R}_\vep(f)/\mdpsc$ follows by the exact same steps.

\section{LDP / Laplace principle lower  bound - Theorem \ref{th:mdpZ}} \label{sec:LP-lb}
The goal of this section is to prove the Laplace principle lower bound, which is equivalent to proving the LDP lower bound. Specifically, we will show that
\begin{align}\label{LP-lb}
	\liminf_{\vep \rt 0} \mdpscb \ln \EE \lf[\exp\Big(-F(\Up_\vep(f))/\mdpscb \Big)\ri] \geq - \inf_{\xi \in C([0,T], \R^d)} [I_f(\xi) + F(\xi)]
\end{align}	
for a bounded Lipschitz continuous function $F: C([0,T], \R^n) \rt \R$,
Fix $\kappa>0$. Let $\xi$ be such that 
$$I_f(\xi)+F(\xi) \leq \inf_{\xi \in C([0,T], \R^d)} [I_f(\xi) + F(\xi)] +\kappa/2.$$ Recall that by Theorem \ref{rate-equivalence}, $I_f = \bar I_f$. Choose $\phi \in \SC{A}_\xi$ such that
\begin{align*}
\f{1}{2} \int_{\R^d \times [0,T]} \|\phi(x,s)\|^2 \pi(dx) ds + F(\xi) 
                         &\leq I_f(\xi)+F(\xi)+\kappa/2\\
                         & \leq  \inf_{\xi \in C([0,T], \R^d)} [I_f(\xi) + F(\xi)] +\kappa.
\end{align*}	

Using the denseness of $C_c^\infty([0,T],\R^d)$ in $L^2(\pi\times \l_T)$, find $\phi_\kappa \in C_c^\infty([0,T],\R^d)$ such that
$$\|\phi_\kappa - \phi\|_2 \leq \kappa.$$
Define $\xi^\kappa$ by
\begin{align}\label{def-xikappa}
	\xi^\kappa(t) = \int_{\R^d\times [0,t]}  Du(x,s) \s(x) \phi_\kappa(x,s) \pi(dx) ds.
\end{align} 		
Notice that by the boundedness of $\s$, and by  Condition \ref{u-reg-assum}-(ii), there exists a constant $\bar \const_1(T)$ such that
\begin{align}\non
|\xi(t) -  \xi^\kappa(t)| & \leq  \int_{\R^d \times [0,t]}  \|Du(x,s)\|_{op} \|\s(x)\|_{op} \|\phi(x,s) - \phi_\kappa(x,s)\| \pi(dx) ds\\ \non
                                  & \leq \bar \const_1(T) \|\s\|_\infty \int_{\R^d \times [0,t]} (1+\|x\|^{p_2})  \|\phi(x,s) - \phi_\kappa(x,s)\| \pi(dx) ds\\ \non
                                  & \leq \bar \const_1(T)\|\s\|_\infty \lf( \int_{\R^d \times [0,t]} (1+\|x\|^{p_2})^2\pi(dx) ds \int_{\R^d \times [0,t]}  \|\phi(x,s) - \phi_\kappa(x,s)\|^2 \pi(dx) ds\ri)^{1/2}\\ 
                                  \label{xi-diff}
                                   & \leq \Theta_2 \bar \const_1(T) \|\s\|_\infty T^{1/2}\kappa,
\end{align}	
where $\Theta_2 \equiv  \int_{\R^d} (1+\|x\|^{p_2})^2\pi(dx) <\infty.$

Let $\bar Z^\kappa_\vep$ be the solution to the following SDE
\begin{align} \non
	\bar Z^\kappa_\vep(t) =&\ x_0 + \f{1}{\vep}\int_0^t b(\bar Z^\kappa_\vep(\vr_\vep(s)))ds+  \f{1}{\sqrt \vep}\int_0^t\s(\bar Z^\kappa_\vep(\vr_\vep(s))) dW(s) \\ \label{ctrlSDE-lb}
	& \ + \f{\mdpsc }{\vep} \int_0^t \s(\bar Z^\kappa_\vep(\vr_\vep(s)))\phi_\kappa(\bar Z^\kappa_\vep(s), s) ds.
\end{align}	
Since $\phi_\kappa \in C_c^\infty([0,T],\R^d)$ and hence Lipschitz, it readily follows that there exists a unique solution to \eqref{ctrlSDE-lb}. Define $\psi_\vep(s) = \phi_\kappa(\bar Z^\kappa_\vep(s), s)$ and by the variational representation we have
\begin{align}\label{vrep-ub}
	-\mdpscb \ln \EE \lf[\exp\Big(-F(\Up_\vep(f))/\mdpscb \Big)\ri] \leq \EE\lf[\f{1}{2} \int_0^T \|\psi^\vep(s)\|^2 ds + F(\bar \Up^{\kappa}_\vep(f))\ri],
\end{align}		
where $\bar \Up^{\kappa}_\vep(f) = \f{1}{\mdpsc }\int_0^t f(s,\bar Z^\kappa_\vep(s))ds.$
Let $\bar \Xi^\kappa_\vep$, defined by \begin{align*}
\bar \Xi^\kappa_\vep(A\times [0,t]) = \int_0^t 1_{\{\bar Z^\kappa_\vep(s) \in A\}} ds,
\end{align*}	
denote the occupation measure of $\bar Z^\kappa_\vep(s)$ on $\R^d\times [0,T].$
We now study the limit of $\bar \Up^{\kappa}_\vep(f)$. Since
\begin{align*}
	\sup_\vep \EE\lf[\int_{\R^d\times[0,T]} \|x\|^{2\alpha}\bar \Xi^\kappa_\vep(dx\times ds)\ri] = 	\sup_\vep \EE\lf[\int_0^T \|\bar Z^\kappa_\vep(s)\|^{2\alpha}ds \ri] <\infty,
\end{align*}	
$\bar \Xi^\kappa_\vep$ is tight in $\SC{M}_1(\R^d\times [0,T]).$
 Let $\Xi^\kappa$ be a limit point of $\bar \Xi^\kappa_\vep$ and assume without loss of generality that $\bar \Xi^\kappa_\vep \rt \Xi^\kappa$ as $\vep \rt 0$.
Now observe that from \eqref{u-Ito} using $\phi_k$ in place of $\phi$
	\begin{align}\non
		\f{\vep}{\mdpsc }	\lf(u(t, \bar X^\kappa_\vep(t)) - u(0, x_0)\ri)  = & \ - \bar \Up^\kappa_\vep(f)(s) +  \int_{\R^d\times [0,t]} Du(x,s) \s(x) \phi_\kappa(x,s) \bar \Xi^\kappa_\vep(dx\times ds)\\ \non
		 & +\ \f{\sqrt{\vep}}{\mdpsc }\int_0^t  Du(\eta_\vep(s), \bar Z^\kappa_\vep(s))\s(\bar Z^\kappa_\vep(\vr_\vep(s)))dW(s)\\ \label{u-Ito-2}
		 & +\ \SC{E}^\vep_0(t) + \SC{E}^\vep_1(t)  + \SC{E}^\vep_2(t) + \SC{E}^\vep_3(t), 
	\end{align}
where $\SC{E}_j^\vep$ are defined analogously. Thus invoking the same calculations in the proof of Proposition \ref{prop-tight}, $\EE(\sup_{s\leq t}\|\SC{E}^{\vep}_j(s)\|) \rt 0, \ j=0, \hdots, 3$ and
$$\EE\lf[\sup_{r\leq t}|\f{\sqrt{\vep}}{\mdpsc }\int_0^r \nabla^T u_l(\eta_\vep(s), \bar Z^\kappa_\vep(s))\s( \bar Z^\kappa_\vep(\vr_\vep(s)))dW(s)|^2\ri] \rt 0,$$
 as $\vep \rt 0$.	

Since $(x,s) \rt  \s(x) \phi_\kappa(x,s)$ is continuous and bounded and 
Condition \ref{u-reg-assum}-(ii) holds, we have by Lemma \ref{lem-weakconv} (actually by a much easier version)  that as $\vep \rt 0$
$$\int_{\R^d\times [0,t]} Du(x,s)\s(x) \phi_\kappa(x,s) \bar \Xi^\kappa_\vep(dx\times ds) \rt \int_{\R^d\times [0,t]} Du(x,s)\s(x) \phi_\kappa(x,s) \Xi^\kappa(dx\times ds).$$
Consequently, it follows that $\bar \Up^\kappa_\vep(f) \rt \int_{\R^d\times [0,\cdot]}  Du(x,s) \s(x) \phi_\kappa(x,s) \Xi^\kappa(dx\times ds).$ 
Now just as in the proof of Proposition \ref{prop-tight}, much easier calculation shows that for any $g \in C^2_b(\R^d)$, $\int_{\R^d\times [0,t]}\SC{L}g(x)\Xi^\kappa(dx\times ds) = 0$ for all $t \in [0,T].$ Writing 
$\Xi^\kappa(dx\times ds) = \g_s^\kappa(dx)ds$, we have
by the uniqueness of the invariant distribution of $\pi$,  $\Xi^\kappa(dx\times ds) = \pi(dx)ds$. Thus  $\bar \Up^\kappa_\vep(f) \rt \xi^\kappa$, where $\xi^\kappa$ is defined by \eqref{def-xikappa}.

Next we observe that since $(x,s) \rt \phi_\kappa(x,s)$ is continuous and bounded, and $\Xi^\kappa_\vep \RT \Xi^\kappa$, where $\Xi^\kappa(dx\times ds) = \pi(dx)ds$,
$$\int_0^T \phi_\kappa(\bar Z^\kappa_\vep(s),s) ds =\int_{\R^d\times [0,T]} \phi_\kappa(x,s) \Xi^\kappa_\vep(dx\times ds)  \rt \int_{\R^d\times [0,T]} \phi_\kappa(x,s) \pi(dx)ds. $$

 Now taking limits in \eqref{vrep-ub}, we have
\begin{align*}
\limsup_{\vep \rt 0}	- \mdpscb \ln \EE \lf[\exp\Big(-F(\Up_\vep(f))/\mdpscb \Big)\ri] \leq &\ \limsup_{\vep \rt 0} \EE\lf[\f{1}{2} \int_0^T \|\phi_\kappa(\bar Z^\kappa_\vep(s),s)\|^2 ds + F(\bar \Up^{\kappa}_\vep(f))\ri]\\
 = &\ \f{1}{2} \int_0^T \|\phi_\kappa(x,s)\|^2 \pi(dx) ds + F(\xi^\kappa)\\ 
 \leq &\ \ \f{1}{2} \int_0^T \|\phi (x,s)\|^2 \pi(dx) ds + F(\xi) +  \|\phi-\phi_\kappa\|_2^2 \\
  &\ + L^{F}_{lip}\|\xi - \xi_\kappa\|_T\\
 \leq & \ \f{1}{2} \int_0^T \|\phi (x,s)\|^2 \pi(dx) ds + F(\xi) + \kappa^2\\
   & \ + L^{F}_{lip} \Theta_2 \bar \const_1(T) \|\s\|_\infty T^{1/2}\kappa\\
\leq & \  \inf_{\xi \in C([0,T], \R^d)} [I_f(\xi) + F(\xi)] +\kappa+\kappa^2 \\
& \ + L^{F}_{lip} \Theta_2 \bar \const_1(T) \|\s\|_\infty T^{1/2}\kappa.
\end{align*}	
Here $L^F_{lip}$ denotes the Lipschitz constant of $F$, and the fourth step used \eqref{xi-diff}. Sending $\kappa \rt 0$, we have \eqref{LP-lb}. 

Again, the proof for the Laplace principle lower bound for $\Xi^{R}_\vep(f)/\mdpsc$ follows by the exact same steps.

\setcounter{section}{0}
\setcounter{theorem}{0}
\setcounter{equation}{0}
\renewcommand{\theequation}{\thesection.\arabic{equation}}

\appendix
\section*{Appendix}
\renewcommand{\thesection}{A} 

The following version of \cite[Theorem 9.11]{GiTu01} is used in proving Proposition \ref{u-est}. The proof is included simply for completeness.

\begin{lemma}\label{Lemma-D2v}
Let $g\in L^p_{loc}(\R^d,\R)$ 
and $v \in W^{2,p}_{loc}(\R^d,\R)$  a solution to the elliptic equation $\SC{L}v=g$, where the coefficient $a$ is uniformly continuous and satisfies Condition \ref{cond_SDE_inv}-(ii), and $b$ satisfies
 $$\|b(x)\| \leq \cnst (1+\|x\|)^{\bar \alpha}$$
for some constant $\cnst>0$ and exponent $\bar\alpha>0$. Then for any $R>0$ and $0<\theta<1$, there exists a constant $\bar{\mfk c}^0$ depending on $\cnst, \l_1, \l_2, \theta, R, d$ and $p$ such that
\begin{align*}
\|v\|_{W^{2,p}B(y, \theta R)} \leq \bar{\mfk c}^0(\|g\|_{L^p(B(y,R))} + (1+\|y\|^{2\bar\alpha})\|v\|_{L^p(B(y,R))} ),
\end{align*}

\end{lemma}

\begin{proof}
Fix $y>0$.  Then there exist  constants $\bar{\mfk c}_1$ and $r_0$  depending on $d,p$ and $\l_1$ (given in \ref{cond_SDE_inv}-(ii))  
such that for any function $\tilde v \in C^2_0(B(y,r)),$ with $r\leq r_0$
$$\|D^2\tilde v\|_p \leq \bar{\mfk c}_1 \|\tilde{\SC{L}}\tilde v\|_p,$$
where $\tilde{\SC{L}} = \sum_{ij}a_{ij}\partial^2_{ij}.$

For $0<\theta<1$, let $\theta'=(1+\theta)/2,$ and let $\eta \in C^2_0(B(y,\theta' r))$ be  a cutoff function such that
\begin{align*}
0\leq \eta \leq 1,\quad \eta\Big|_{B(y,\theta r)}\equiv 1,  \quad |\nabla \eta| \leq \f{2}{(1-\theta')r}, \quad |D^2 \eta| \leq \f{4}{(1-\theta')^2r^2}. 
\end{align*}
Putting $\tilde v = \eta v$ note that 
\begin{align*}
\partial^2_{ij}\tilde v 
 = \partial_{i}\eta\partial_{j}v+\eta \partial^2_{ij}v+ \partial_{i}v\partial_{j}\eta+ v \partial^2_{ij}\eta.
\end{align*}
Thus
\begin{align*}
\|D^2v\|_{L^p(B(y,\theta r))} = &\  \|D^2\tilde v\|_{L^p(B(y, \theta r))} \leq  \bar{\mfk c}_1\|\tilde{L}\tilde v\|_{L^p(B(y,\theta' r))}\\
\leq &  \bar{\mfk c}_1\ \big\|\eta\tilde{L}v + \sum_{ij} a_{ij} \partial_i v \partial_i \eta+ v\tilde{L}\eta\big\|_{L^p(B(y,\theta' r))}\\
=  &  \bar{\mfk c}_1\ \big\|\eta(-g+b^T\nabla v) + \sum_{ij} a_{ij} \partial_i v\partial_i \eta+ v\tilde{L}\eta\big\|_{L^p(B(y,\theta' r))}\\
\leq & \  \bar{\mfk c}_1\lf(\|g\|_{L^p(B(y,\theta' r))} +\Big\|\|b\|\|\nabla v\|\Big\|_{L^p(B(y,\theta' r))} + \l_2\Big\|\|\nabla v\|\|\nabla \eta\|\Big\|_{L^p(B(y,\theta' r))}\ri. \\
& \ \lf.\hs{1cm}+ \|a\|_\infty \Big\||v|\|D^2\eta\|\Big\|_{L^p(B(y,\theta' r))}\ri) \\
\leq&\ \bar{\mfk c}_2\Big(\|g\|_{L^p(B(y,\theta' r))}+ (1+\|y\|^{\bar\alpha})\|\nabla v\|_{L^p(B(y,\theta' r))}+ \|\nabla v\|_{L^p(B(y,\theta' r))}/(1-\theta')r\\
&\ \hs{1cm} + \|v\|_{L^p(B(y,\theta' r ))}/(1-\theta')^2r^2 \Big). \tag{\dag}
\end{align*}
Following \cite{GiTu01}, we define the weighted seminorms for $k=0,1,2$
\begin{align*}
\Phi_k = \sup_{0<\theta<1} (1 - \theta)^kr^k\|D^kv\|_{L^p(B(y,\theta r))}.
\end{align*}
Multiplying (\dag) by $(1-\theta)^2r^2$ and observing that $1 - \theta' = (1-\theta)/2$, we have
\begin{align}\label{deriv2-ineq}
\Phi_2 \leq \bar{\mfk c}_3\lf(\|g\|_{L^p(B(y,r))} + (1+\|y\|^{\bar\alpha})\Phi_1 + \Phi_1+ \Phi_0\ri),
\end{align}
for some constant $\bar{\mfk c}_3$. Using Sobolev's interpolation inequality (\cite[Theorem 7.28]{GiTu01})
it could easily be seen that the $\Phi_k$ satsify the following interpolation inequality ( see (9.39) in \cite{GiTu01})
\begin{align*}
\Phi_1 \leq \ep\Phi_2 + \mfk C\Phi_0/\ep
\end{align*}
for any $\ep>0$, where the constant $\mfk C$ depends only on $d$. Using this inequality twice for the two $\Phi_1$ terms in \eqref{deriv2-ineq}, once with $\ep = (4\bar{\mfk c}_3 (1+\|y\|^{\bar\alpha}))^{-1}$, and the second with $\ep =(4\bar{\mfk c}_3)^{-1}$, we have
\begin{align*}
\Phi_2 \leq &\ \bar{\mfk c}_5 \Big(\|g\|_{L^p(B(y,r))}   + (1+\|y\|^{2\bar\alpha}) \Phi_0 \Big) + \f{1}{2}\Phi_2
\end{align*}
for some constant $\bar{\mfk c}_5$, whence it follows for some constant $\bar{\mfk c}_6$
\begin{align*}
	\|v\|_{W^{2,p}B(y,\theta r)} \leq \bar{\mfk c}_6\lf(\|g\|_{L^p(B(y,r))} + (1+\|y\|^{2\bar\alpha})\|v\|_{L^p(B(y,r))} \ri).
\end{align*}
The case for general $R>0$ now easily follows by covering the ball $B(y,R)$ with finite number of balls of radius $r_0$.
\end{proof}

\begin{remark}\label{rem:deriv-v-bd}{\rm
In particular, it follows  that if $|g(x)| = O(\|x\|^{p_0}),$ $|v(x)| = O(\|x\|^{p_1})$, then for some constant $\bar{\mfk c}^1$
\begin{align*}
	\|v\|_{W^{2,p}B(y, \theta r)} \leq \bar{\mfk c}^0_1(1+\|y\|^{p_2}),
\end{align*}
where $p_2 = \max\{p_0, p_1+2\bar \alpha\}.$
Next choose $p>d$. Then by Sobolev's embedding theorem, there exists a constant $\bar{\mfk c}_1' \equiv \bar{\mfk c}_1'(\theta r)$  such that
\begin{align*}
\|\nabla v(y)\| \leq \bar{\mfk c}_1' \|v\|_{W^{2,p}B(y,\theta r)} \leq  \bar{\mfk c}_1' \bar{\mfk c}^0_1(1+\|y\|^{p_2}).
\end{align*} 
}
\end{remark}

We now state the result on pointwise bounds of $\|D^2v(\cdot)\|.$ 

\begin{lemma}\label{Lemma-D2v-pt}
Assume the setup and hypothesis of Lemma \ref{Lemma-D2v}. Furthermore, suppose that that the coefficients of $\SC{L}$ are in $C^1(\R^d)$, and that for each $k$, $\|a^{(k)}\|_\infty < \infty$ and $\|b^{(k)}(x)\| \leq \cnst (1+\|x\|^{\bar \alpha})$ for some constant $\cnst$ and some exponent $\bar \alpha'$, where 
$$a^{(k)}_{ij}(x) = \partial_k a_{ij}(x), \qquad b^{(k)}_i(x)  = \partial_{k} b_i(x).$$ 
Also, as in Remark \ref{rem:deriv-v-bd}, assume that  $|g(x)| = O(\|x\|^{p_0}), \|\nabla g(x)\| = O(\|x\|^{p_0}),$ and $|v(x)| = O(\|x\|^{p_1})$ for some exponents $p_0$ and $p_1$.
Then for some constant $\bar{\mfk c}^2$
\begin{align*}
\|D^2v(y)\| \leq &\  \bar{\mfk c}^2(1+ \|y\|)^{p_3}
\end{align*}
where $p_3 =\max \{p_0+2\bar\alpha, p_1+4\bar \alpha\}.$
 \end{lemma}

\begin{proof}

First notice that $v^{(k)} = \partial_k v$ satisfies 
\begin{align*}
\SC{L}v^{(k)} = \tilde g_k,
\end{align*}
where $\tilde g_k = g^{(k)} -  b^{(k)} \cdot \nabla v - \f{1}{2}tr(a^{(k)}D^2v).$ It now follows from Lemma \ref{Lemma-D2v} that for each $k=1,2,\hdots,d$, 
\begin{align*}
\|v^{(k)}\|_{W^{2,p}B(y,\theta r)}   \leq & \ \leq \bar{\mfk c}_7\lf(\|\tilde g_k\|_{L^p(B(y,r))} + (1+\|y\|^{2\bar\alpha})\|v^{(k)}\|_{L^p(B(y,r))} \ri)\\
   & \leq \ \bar{\mfk c}_8 \Big(\|g^{(k)}\||_{L^p(B(y,r))}+ (1+\|y\|^{\bar\alpha}) \|\nabla v\|_{L^p(B(y,r))} + \|D^2v\|_{L^p(B(y,R))} \\
   & \ \hs*{1cm}  + (1+\|y\|^{2\bar\alpha}) \|\nabla v\|_{L^p(B(y,r))}  \Big).
\end{align*}
Thus, 
\begin{align*}
\|v\|_{W^{3,p}B(y,\theta r)}  \leq & \ \bar{\mfk c}_{9}\ \lf(\|\nabla g\|_{L^p(B(y,r))} + (1+\|y\|^{2\bar\alpha})\| v\|_{W^{2,p}(B(y,r))}\ri)\\
\leq & \bar{\mfk c}_{10}\ (\|\nabla g\|_{L^p(B(y,r))}+(1+\|y\|^{2\bar\alpha})\|g\|_{L^p(B(y,r))}+(1+\|y\|^{4\bar\alpha})\|v\|_{L^p(B(y,r))}).
\end{align*}
 The desired pointwise bound now again follows from Sobolev's embedding theorem (by choosing $p>d$), and the assumption on $g$ and $\nabla g$. 

\end{proof}

\begin{lemma}\label{lem-weakconv}
Let $E_1$ and $E_2$ be  separable Banach spaces and
$h: E_1\times E_2 \rt \R^d$  a continuous function such that
$\|h(x,z)\| \leq B_h(1+\|x\|^\beta_1)(1+\|z\|^\rho_2)$ for some constants $B_h\geq 0, \beta\geq 0, \rho\geq 0.$
Let $\mu_n$ be a sequence of $\SC{P}(E_1\times E_2)$-valued random variables  and $\mu_n \rt \mu$ a.s as $n\rt \infty$.
Suppose that for some $\alpha>\beta$
\begin{align}\label{assum-1}
\sup_n\EE \int_{E_1\times E_2} (\|x\|^{2\alpha} +\|z\|^{2\rho}) \mu_n(dx\times dz) <\infty.
\end{align}
Then as $n\rt \infty$
\begin{align*}
\EE\lf|	\int_{E_1\times E_2} h(x,z) \mu_n(dx\times dz) - 	\int_{E_1\times E_2} h(x,z) \mu(dx\times dz)\ri| \rt 0.
\end{align*}	
\end{lemma}	

\begin{proof} First observe that by Fatou's lemma,
	$$ \EE\int_{E_1\times E_2} (\|x\|^{2\alpha} +\|z\|^{2\rho}) \mu(dx\times dz) <\infty.$$
	The assertion is easier if $\beta = 0$ or $\rho=0$, and therefore, we assume that both $\beta>0$ and $\rho>0$.
Let $e^1_N : E_1 \rt [0,1]$ be a continuous and bounded function on $\R$ such that $e^1_N(x) = 1$ if $\|x\|_1 \leq N$ and $e^1_N(x) = 0$ if $\|x\|_1 > N+1$. Similarly, $e^2_M : E_2 \rt [0,1]$ be a continuous and bounded function on $\R$ such that $e^2_M(z) = 1$ if $\|z\|_2 \leq M$ and $e^2_M(z) = 0$ if $\|z\|_1 > M+1$.
Let 
$$D_n = \|\int_{E_1\times E_2} h(x,z) \mu_n(dx\times dz)-	\int_{E_1\times E_2} h(x,z) \mu(dx\times dz)\|.$$
Clearly,
\begin{align*}
	D_n \leq &  \|\int_{E_1\times E_2} h(x,z)e^1_N(x)e^2_M(z) \mu_n(dx\times dz)-	\int_{E_1\times E_2} h(x,z)e^1_N(x)e^2_M(z) \mu(dx\times dz)\|\\
	& +  \|\int_{E_1\times E_2} h(x,z)(1-e^1_N(x)) \mu_n(dx\times dz)\|+ \|\int_{E_1\times E_2} h(x,z)(1-e^1_N(x)) \mu(dx\times dz)\|\\
	&+\|\int_{E_1\times E_2} h(x,z)e^1_N(x)(1-e^2_M(z)) \mu_n(dx\times dz)\|+ \|\int_{E_1\times E_2} h(x,z)e^1_N(x)(1-e^2_M(z)) \mu(dx\times dz)\|\\
	& \equiv D_n^1 + D_n^2 +D_2+D_n^3+D_3.
\end{align*}
Now observe that
\begin{align*}
 D^2_n \equiv&  \|\int_{E_1\times E_2} h(x,z)(1-e^1_N(x)) \mu_n(dx\times dz)\|\\
        \leq & B_h \int_{E_1\times E_2} (1+\|x\|^\beta_1)(1+\|z\|^\rho_2)1_{\{\|x\|_1 \geq N\}}\mu_n(dx\times dz)\\
      \leq &   2B_h \int_{E_1\times E_2} \|x\|^\beta_1(1+\|z\|^\rho_2)1_{\{\|x\|_1 \geq N\}}\mu_n(dx\times dz)\\
       \leq &   \f{2B_h}{N^{\alpha-\beta}} \int_{E_1\times E_2} \|x\|^\alpha_1(1+\|z\|^\rho_2)1_{\{\|x\|_1 \geq N\}}\mu_n(dx\times dz)\\
       \leq &  \f{B_h}{N^{\alpha-\beta}} \int_{E_1\times E_2} \lf(\|x\|^{2\alpha}_1+(1+\|z\|^\rho_2)^2\ri)\mu_n(dx\times dz).
\end{align*}
Therefore,
\begin{align}\label{eq:Dn1}
	\sup_n\EE(D^2_n) \leq & B_h\Theta_1/N^{\alpha-\beta},
\end{align}	
where $\Theta_1 \equiv \sup_n\lf( \int_{E_1\times E_2}\lf( \|x\|^{2\alpha}_1+ (1+\|z\|^\rho_2)^2\ri)\mu_n(dx\times dz)\ri) < \infty$ by the assumption in \eqref{assum-1}. Similarly,
\begin{align*}
	\EE(D_2)\equiv\EE\|\int_{E_1\times E_2} h(x,z)(1-e^1_N(x)) \mu(dx\times dz)\| \leq & B_h\Theta_1/N^{\alpha-\beta}.
\end{align*}	
Next, 
\begin{align*}
	D^3_n \equiv & \|\int_{E_1\times E_2} h(x,z)e^1_N(x)(1-e^1_M(z)) \mu_n(dx\times dz)\|\\
	\leq & B_h \int_{E_1\times E_2} (1+\|x\|^\beta_1)(1+\|z\|^\rho_2)1_{\{\|x\|_1 \leq N+1\}}1_{\{\|z\|_2 \geq M\}}\mu_n(dx\times dz)\\
	\leq & B_h(1+(N+1)^\beta)  \int_{E_1\times E_2} (1+\|z\|^\rho_2)1_{\{\|z\|_2 \geq M\}}\mu_n(dx\times dz)\\
	\leq & \f{B_h(1+(N+1)^\beta)}{1+M^\rho}  \int_{E_1\times E_2} (1+\|z\|^\rho_2)^2 \mu_n(dx\times dz).
\end{align*}	
Thus,
\begin{align*}
	\sup_n \EE(D^3_n) \leq  \f{B_h(1+(N+1)^\beta) \Theta_2}{1+M^\rho}, 
\end{align*}	 
where $\Theta_2\equiv  \sup_n \EE\int_{E_1\times E_2} (1+\|z\|^\rho_2)^2 \mu_n(dx\times dz) <\infty$ by  \eqref{assum-1}. Similarly,
\begin{align*}
\EE(D^3) \equiv\EE	\|\int_{E_1\times E_2} h(x,z)e^1_N(x)(1-e^1_M(z)) \mu(dx\times dz)\| \leq &  \f{B_h(1+(N+1)^\beta) \Theta_2}{1+M^\rho}.
\end{align*}
Finally,  since $\mu_n \Rt \mu$ and $(x,z) \rt  h(x,z)e^1_N(x)e^2_M(z)$ is a continuous and bounded function,  
$$ D^1_n\equiv \|\int_{E_1\times E_2} h(x,z)e^1_N(x)e^2_M(z) \mu_n(dx\times dz)-	\int_{E_1\times E_2} h(x,z)e^1_N(x)e^2_M(z) \mu(dx\times dz)\| \rt 0, \quad a.s$$
as $n\rt \infty.$ Next since $\|\int_{E_1\times E_2} h(x,z)e^1_N(x)e^2_M(z) \mu_n(dx\times dz)\| \leq B_h(1+(N+1)^\beta)(1+(M+1)^\rho)$, it follows by the dominated convergence theorem that $\EE(D^1_n)\rt 0$ as $n\rt \infty.$ Putting things together, we have
$$\EE(D_n) \leq E(D_n^1)+2B_h\Theta_1/N^{\alpha-\beta}+2B_h(1+(N+1)^\beta) \Theta_2/(1+M^\rho).$$
The assertion now follows by first letting $n\rt \infty$, then $M\rt \infty$ and finally, $N\rt \infty$.

\end{proof}	

\begin{lemma}\label{quad-max}
Let $\SC{L}^2_d\equiv \SC{L}^2(\Om, \R^d)$ denote the space of square integrable $\R^d$-valued random variables on a probability space $(\Om, \SC{F}, \PP)$, and $H$ an $n\times d$ random matrix. Assume that $M= \EE(HH^T)$ is invertible. Then for any $b \in \R^n$,
\begin{align*}
	\min\{\EE\|Y\|^2: Y \in \SC{L}^2_d,\ \EE(HY)=b \} = b^TM^{-1}b,
\end{align*}

\begin{proof}
Let $Y$ be such that $E(HY)=b$. Then notice that
\begin{align*}
0\leq & \ \EE\|Y - H^TM^{-1}b\|^2 = \EE\|Y\|^2 - 2\EE\lf(\<Y, H^TM^{-1}b\>\ri)+ \EE\|H^TM^{-1}b\|^2\\
=&\ \EE\|Y\|^2 - 2\EE \lf(\<HY, M^{-1}b\>\ri) + \EE(b^TM^{-1}(HH^T)M^{-1}b)\\
= &\ \EE\|Y\|^2 - 2 \<\EE(HY), M^{-1}b\> + b^TM^{-1}\EE(HH^T)M^{-1}b\\
= & \EE\|Y\|^2 - b^TM^{-1}b,
\end{align*}
which proves that $ \EE\|Y\|^2 \geq b^TM^{-1}b$. Finally, observe that equality holds for $Y = H^TM^{-1}b$.
	
\end{proof}	

\end{lemma}

\bibliographystyle{plainnat}
\bibliography{StoAvg}

%
%
%
%
%
%
%
%
\end{document}